\documentclass[preprint,12pt]{elsarticle}
\usepackage{enumitem}
\usepackage{hyperref}
\usepackage{bm}

\usepackage{placeins}

\usepackage{tikz}

\usetikzlibrary{intersections} % 引入交点库
\usetikzlibrary{3d, arrows.meta}

\usetikzlibrary{calc}
\usetikzlibrary{positioning}	
\usepackage{tikz-3dplot}
\usetikzlibrary{3d}
\usepackage{ifthen}
\usepackage{msc}
\usepackage{asymptote}
\usepackage{pgfplots}
\pgfplotsset{compat=1.18}

\usetikzlibrary{arrows.meta, positioning, backgrounds, calc, shapes}

\usepackage{amssymb}
\usepackage{amsmath}
\usetikzlibrary{shapes.geometric, arrows.meta, positioning, fit, calc, shadows}

\usepackage{subcaption} % ← 需要这个宏包来并排放图

\usepackage{appendix}

\usepackage{mathrsfs}
\usepackage{CJK}

\usepackage{extarrows}

\usepackage{enumerate}
\usepackage{amsmath,amsthm,amssymb,amscd}

\usepackage{graphicx}
\usepackage{eepic,color,colordvi,amscd}

\usepackage{tikz-cd}
\usepackage{geometry}

\newtheorem{theorem}{Theorem}[section]
\newtheorem{lemma}{Lemma}[section]
\newtheorem{claim}{Claim}[section]
\newtheorem{corollary}{Corollary}[section]

\newtheorem{conjecture}{Conjecture}[section]

\newtheorem{problem}{Problem}
\newtheorem{definition}{Definition}[section]

\newtheorem{remark}{Remark}[section]

\usepackage{graphics}
\usepackage{verbatim}
\makeatletter

\newcommand{\Rmnum}[1]{\expandafter\@slowromancap\romannumeral #1@}
\makeatother

\journal{}
\begin{document}
	%\begin{CJK}{GBK}{song}
	%\baselineskip=4ex
	
	\begin{frontmatter}

		\title{Geometrization of Graphs: Towards Bounding the Chromatic Number via High-Dimensional Embedding\\}
		\author{Qiming Fang$^{a}$, \ Sihong Shao$^b$}

		\affiliation{organization={Beijing International Center for Mathematical Research, Peking University},%Department and Organization
			%addressline={No.5 Yiheyuan Road, Haidian District}, 
			city={Beijing},
			postcode={100871}, 
			%state={Beijing},
			country={China}}
		
		\affiliation{organization={CAPT, LMAM and School of Mathematical Sciences, Peking University},%Department and Organization
			%addressline={No.5 Yiheyuan Road, Haidian District}, 
			city={Beijing},
			postcode={100871}, 
			%state={Beijing},
			country={China}}

		\begin{abstract}
			%Since the $1$-skeleton of a polytope or a CW complex can be regarded as a graph, we can geometrize a graph by restoring certain induced subgraphs into simplices or polytopes. Based on this idea, 
			%This paper proposes a method of geometrizing graphs, which lifts a $(d+1)$-connected graph $G$ to dimension $(d-1)$ (the lifted graph is denoted by $U^{d-1}(G)$). 
			%	This paper proposes a geometrization method that transforms a $(d+1)$-connected graph $G$ into a $(d-1)$-dimensional manifold $U^{d-1}(G)$. 
			%	A geometrization method that transforms a $(d+1)$-connected graph $G$ into a $(d-1)$-dimensional manifold $U^{d-1}(G)$ is first established through adding some $i$-balls with $2\le i \le d-1$ into $G$ such that the $j$-th homotopy group is trivial for $j=0, 1, \dots, d-2$. On this basis, we establish a sufficient condition for $U^{d-1}(G)$ to be embedded into $\mathbb{R}^d$ and an upper bound for $\chi(G)$,  the chromatic number of $G$. 
			
			We establish a geometric framework by transforming a graph $G$ into a $(d-1)$-dimensional CW complex $U^{d-1}(G)$. This construction is achieved by systematically attaching $i$-spheres ($2 \le i \le d-1$) to $G$ according to specific rules, ensuring that the $j$-th homotopy group of $U^{d-1}(G)$ are trivial for $j = 0, 1, \dots, d-2$. Building upon this construction, we provide a necessary and sufficient condition for $U^{d-1}(G)$ to be embeddable into $\mathbb{R}^d$, which yields an upper bound for the chromatic number $\chi(G)$.
			To be more specific, we prove that if $G$ does not contain $K_{d+3}$ and $K_{i, d+4-i}$ ($i \in \{2, 3, \dots, \lfloor \frac{d+4}{2} \rfloor \}$) as a minor, 
			then $U^{d-1}(G)$ embeds into $\mathbb{R}^d$ and $\chi(G) \leq 3\cdot 2^{d-1}$. 
			Finally, as a preliminary attempt, we extend the Discharging method to $\mathbb{R}^d$ and investigate the coloring problem for $(d-2)$-faces in $\mathbb{R}^d$.
			
			%Furthermore, based on the above theorem, we extend the discharging method, originally developed for the study of the Four Color Theorem, to $\mathbb{R}^d$. This generalized approach can be applied to investigateD the coloring problems in Euclidean space $\mathbb{R}^d$. 
			
			%	 when $U^{d-1}(G)$ admits an embedding into $\mathbb{R}^d$. Specifically, we prove that $U = U^{d-1}(G)$ embeds into $\mathbb{R}^d$ if $G$ contains neither $K_{d+3}$ nor $K_{3,d+1}$ as a minor, and the chromatic number of $G$ satisfies $\chi(G) \leq d(d+1)$. 
		\end{abstract}

		\begin{keyword}
			embedding \sep the Hadwiger conjecture \sep homotopy group \sep minor \sep Discharging 
			%% keywords here, in the form: keyword \sep keyword
			%% PACS codes here, in the form: \PACS code \sep code
			%% MSC codes here, in the form: \MSC code \sep code
			%% or \MSC[2008] code \sep code (2000 is the default)
		\end{keyword}
		
		%\begin{CCSXML}
		%	<ccs2012>
		%	<concept>
		%	<concept_id>10002950.10003624.10003633.10003643</concept_id>
		%	<concept_desc>Mathematics of computing~Graphs and surfaces</concept_desc>
		%	<concept_significance>500</concept_significance>
		%	</concept>
		%	<concept>
		%	<concept_id>10002950.10003624.10003633.10003639</concept_id>
		%	<concept_desc>Mathematics of computing~Graph coloring</concept_desc>
		%	<concept_significance>500</concept_significance>
		%	</concept>
		%\end{CCSXML}
		
		%\ccsdesc[500]{Mathematics of computing~Graphs and surfaces}
		%\ccsdesc[500]{Mathematics of computing~Graph coloring}
		
		%\textbf{MSC Classification:} 05C10

	\end{frontmatter}
	
	%\maketitle

	%\newpage
	%\tableofcontents % 插入目录
	%\newpage         % 通常在目录后另起一页
	
	\section{Introduction}
	Throughout this paper, all graphs are assumed to be finite, connected, and simple. 
	For definitions not explicitly provided in this paper, we refer the reader to~\cite{armstrong2013basic,bondy2008graph,hatcher2002algebraic}.
	
	\subsection{Motivation}
	
	The Hadwiger conjecture states that if $G$ is loopless and has no $K_t$ minor, then its chromatic number satisfies $\chi(G) < t$. For $\mathbb{R}^2$, the Four Color Theorem and Wagner’s Theorem are equivalent to the case $t=5$ of the Hadwiger conjecture \cite{bondy2008graph}. 
	For $t \geq 6$, one may ask whether the Hadwiger conjecture can be approached via the following two problems. 
	
	\begin{problem}[embedding problem]\label{problem1}
		Let $G$ be a graph, and define a dimensional-raising function $U^{d-1}(G)$ that transforms $G$ into a $(d-1)$-dimensional CW complex $U^{d-1}(G)$. 
		Can one characterize necessary and sufficient conditions under which $U^{d-1}(G)$ is embeddable in the $d$-dimensional Euclidean space $\mathbb{R}^{d}$ (When $d=2$, one has Wagner’s Theorem and $t=d+3=5$)?
		At a minimum, such a characterization should imply that $G$ contains no $K_{d+3}$-minor.
	\end{problem}
	
	\begin{problem}[coloring problem]\label{problem2}
		Let $G$ be a graph, and define a dimensional-raising function $U^{d-1}(G)$ that transforms $G$ into a $(d-1)$-dimensional CW complex $U^{d-1}(G)$. 
		If $U^{d-1}(G)$ is embeddable in $\mathbb{R}^{d}$, can one derive an upper bound for the chromatic number of $G$?
		More specifically, can one prove that $\chi(G) \leq d+2$?
	\end{problem}
	
	In this paper, we settle Problem \ref{problem1} affirmatively, as stated in Theorem \ref{anti-minor}. 
	%Regarding Problem \ref{problem2}, we extend the Discharging method for planar graphs to the setting of $\mathbb{R}^d$, and propose a general technical framework for addressing this problem.
	
	%It should be noted that, although a large number of configurations (reducible and unavoidable) were identified in the classical Discharging-based proof of the Four Color Theorem, such configurations do not admit a direct extension to higher-dimensional spaces. 
	%As a result, even though we generalize the Discharging method to $\mathbb{R}^d$, applying this method to solve Problem~\ref{problem2} remains highly nontrivial, and the complexity of the proof is expected to exceed that of the Four Color Theorem. 
	%Therefore, this paper focuses exclusively on the theoretical framework of the Discharging method in $\mathbb{R}^d$. We do not pursue a detailed analysis of configurations or Discharging rules, as these will be the primary focus of our future research. 
	
	Many researchers have previously regarded the $1$-skeleton of a polytope or a CW complex as a graph and studied its related properties. The reason they did not discover the connection between higher-dimensional embeddings and the Hadwiger conjecture is that they imposed no restrictions on homotopy groups. 
	In contrast, we adopt an opposite approach. 
	Starting from an arbitrary graph $G$, we construct, in a prescribed manner, a $(d-1)$-dimensional CW complex $U^{d-1}(G)$ while keeping the $1$-skeleton unchanged. 
	We then prove that, for every $i \in \{1,2,\dots,d-2\}$, the $i$-th homotopy group of $U^{d-1}(G)$ is trivial. 
	Finally, we establish necessary and sufficient conditions for the embeddability of $U^{d-1}(G)$ into $\mathbb{R}^d$,  
	and then using the correspondence between $U^{d-1}(G)$ and the graph $G$, we reveal a connection between the Hadwiger conjecture and higher-dimensional embeddings.

	\subsection{Background}
	
	We begin by presenting the relevant background of the Hadwiger conjecture, briefly reviewing the difficulties encountered by previous researchers, and outlining the central idea of graph embeddings in higher-dimensional spaces. 
	
	It is easy to see that the Hadwiger conjecture is trivial for $t=2,3$, and Hadwiger himself proved the case $t=4$ \cite{diestel2025graph}.  
	In 1937, Wagner established that a finite graph is planar if and only if it contains neither $K_5$ nor $K_{3,3}$ as a minor \cite{bondy2008graph}; based on this result, he showed that the Four Color Theorem is equivalent to the Hadwiger conjecture when $t=5$.  
	In 1993, Robertson, Seymour, and Thomas \cite{robertson1993hadwiger} proved the conjecture for $t=6$.  
	For the case $t \geq 7$, some partial results are known. In the 1980s, Kostochka \cite{kostochka1984lower} and Thomason \cite{thomason1984extremal} independently proved that every graph with no $K_t$ minor has average degree $O(t \sqrt{\log t})$, and thus can be colored with $O(t \sqrt{\log t})$ colors. Subsequent improvements of this bound have led to proofs that such graphs are $O(t \log^{\beta} t)$-colorable (with $\beta > \tfrac{1}{4}$) \cite{norin2019breaking} and even $O(t \log \log t)$-colorable \cite{delcourt2025reducing}.

	When $t=5$, the Hadwiger conjecture can be decomposed into a planar embedding problem and a coloring problem. 
	For the case $t \geq 6$, we also aim to adopt a similar approach, decomposing the Hadwiger conjecture into the problem of graph embedding in $\mathbb{R}^d$ and the problem of graph coloring in $\mathbb{R}^d$. 
	Actually, it is not feasible to directly study the embedding of graphs into high-dimensional Euclidean spaces. 
	Cohen et al. \cite{cohen1997three} proved in 1995 that every finite graph embeds into $\mathbb{R}^3$. Since graphs are $1$-dimensional CW complexes, this result implies that studying graph embeddings into $\mathbb{R}^4$ or higher dimensions is trivial. In fact, Fritsch \cite{fritsch1990cellular} proved in 1993 a more general result concerning the embedding of CW complexes into higher-dimensional spaces, stating that every countable and locally compact CW complex of dimension $m$ can be embedded in $\mathbb{R}^{2m+1}$.

	%\begin{lemma}[\cite{fritsch1993cw}]
	%	Every countable and locally compact CW complex of dimension $m$ can be embedded in $\mathbb{R}^{2m+1}$. 
	%\end{lemma}

	In other words, in order to study the embedding of graphs into Euclidean spaces of dimension greater than two, it is necessary to geometrize the graphs, that is, to increase their dimension through certain mathematical constructions.
	
	A natural idea for finding a high-dimensional analogue of planar graphs in Euclidean spaces is to regard graphs as the skeletons of CW complexes, where the dimension of the CW complex determines the dimension of the graph. For example, a $3$-connected planar graph can be regarded as the $1$-skeleton of a polyhedron. 
	%Similarly, if a graph $G$ can be realized as the $1$-skeleton of a $d$-dimensional polytope, then we consider the dimension of $G$ to be $d$.
	
	Based on this idea, one may naturally generalize the concept of planar graphs to higher dimensions. Concerning the embeddability of graphs in higher-dimensional spaces, some results are already known. For instance, when $d=3$, Carmesin \cite{carmesin2023embedding} characterized the embeddability of simply connected $2$-dimensional simplicial complexes in $3$-space, in a way analogous to Kuratowski’s (Wagner’s) characterization of graph planarity, by nine excluded minors. 
	
	Although the above approach provides a way to extend planar graphs into higher dimensions, it also introduces two major difficulties below that prevent us from establishing a connection between high-dimensional embeddability and the Hadwiger conjecture.
	
	\begin{itemize}
		\item Due to the highly intricate structure of high-dimensional manifolds, we notice that as the dimension increases, the number of forbidden minors grows rapidly. For example, in $\mathbb{R}^2$, the set of forbidden minors consists only of $K_5$ and $K_{3,3}$, whereas in $\mathbb{R}^3$ it already increases to nine distinct cases. 
		\item The upper bound theorem \cite{stanley2017upper} states that cyclic polytopes have the largest possible number of faces among all convex polytopes with a fixed dimension and number of vertices. In other words, there exist $d$-dimensional ($d> 3$) polytopes with $n$ vertices whose $1$-skeleton is the complete graph $K_n$, where $n$ can be arbitrarily large. This implies that the chromatic number is completely unrelated to the dimension, which contradicts our original motivation.
	\end{itemize}
	
	%Therefore, to establish a meaningful connection between high-dimensional embeddings of graphs and the Hadwiger conjecture, the first step is to develop a method for increasing the dimension of a graph without altering its $1$-skeleton. Next, one needs to characterize necessary and sufficient conditions under which the resulting higher-dimensional graph admits an embedding into $\mathbb{R}^d$. Finally, by establishing an upper bound for the chromatic number of this class of graphs, one can establish a connection between high-dimensional embeddings and the Hadwiger conjecture. 
	
	Therefore, to relate high-dimensional embeddings to the Hadwiger conjecture, one first needs a procedure that increases the dimension of a graph while preserving its $1$-skeleton. The next step is to identify necessary and sufficient conditions for when the resulting higher-dimensional graph admits an embedding into $\mathbb{R}^d$. Finally, an upper bound for the chromatic number of this class of graphs allows this embedding theory to be linked directly to the Hadwiger conjecture.

	%it is necessary to impose restrictions on the homotopy groups of the polytope, namely, requiring that the $i$-th homotopy group is trivial for all $i \leq d-2$. 

	%\subsection{Main ideas and techniques}
	
	\subsection{Geometrization of graphs}
	
	Note that standard definitions and notations not explicitly stated here follow \cite{armstrong2013basic,bondy2008graph,hatcher2002algebraic}, such as skeleton, homotopy, $i$-ball, $i$-sphere, $i$-chain, fundamental groups, and the $n$-th homotopy group $\pi_n(X)$. Note: A topological space $X$ is $n$-connected if $\pi_i(X) = 0$ for all $i \leq n$. In contrast, a graph is $k$-connected if its vertex connectivity is at least $k$. These represent fundamentally distinct concepts. To avoid ambiguity, we will explicitly specify whether $n$-connected refers to a space or a graph in all subsequent usage. 
	
	In Section \ref{geo}, we introduce a dimension-raising function $U^{d-1}(G)$ that lifts a graph $G$ into $d-1$ dimensions. 
	
	For example, to increase the dimension of a $2$-connected graph $G$ from one to two without changing its $1$-skeleton, a natural idea is to realize certain cycles of $G$ as the boundaries of two-dimensional disks. The choice of  cycles determines the structural complexity of the resulting higher-dimensional graph. 
	
	We now introduce the notions of induced cycles and chordless cycles. The essential difference between them is that deleting an induced cycle from $G$ preserves connectivity, whereas deleting a chordless cycle does not necessarily do so. We emphasize that the definitions given below may differ from those found in some textbooks; throughout this paper, we adhere to the definitions stated here.

	\begin{definition}[induced cycle]
		Let $G = (V,E)$ be a graph. A cycle $C$ in $G$ is called an \emph{induced cycle} if the subgraph of $G$ induced by the vertex set $V(C)$ is exactly $C$, and $G \backslash C$ is connected. 
	\end{definition}
	
	\begin{definition}[chordless cycle]
		Let $G = (V,E)$ be a graph. A cycle $C$ in $G$ is called a \emph{chordless cycle} if the subgraph of $G$ induced by the vertex set $V(C)$ is exactly $C$. 
	\end{definition}
	
	\begin{figure}[htbp]
		\centering
		% ==== 子图 (a) ====
		\begin{subfigure}[t]{0.45\textwidth}
			\centering
			\tdplotsetmaincoords{70}{110}
			\begin{tikzpicture}[tdplot_main_coords, scale=2, line join=round, line cap=round]
				% ---- 顶点定义 ----
				\coordinate (x1) at (1,0,0);
				\coordinate (x2) at (-1,0,0);
				\coordinate (x3) at (0,1,0);
				\coordinate (x4) at (0,-1,0);
				\coordinate (x5) at (0,0,1);
				\coordinate (x6) at (0,0,-1);
				
				% ---- 面绘制 ----
				\draw[fill=blue!20,opacity=0.7] (x1)--(x3)--(x5)--cycle;
				\draw[fill=blue!20,opacity=0.7] (x3)--(x2)--(x5)--cycle;
				\draw[fill=blue!20,opacity=0.7] (x2)--(x4)--(x5)--cycle;
				\draw[fill=blue!20,opacity=0.7] (x4)--(x1)--(x5)--cycle;
				
				\draw[fill=blue!10,opacity=0.6] (x1)--(x3)--(x6)--cycle;
				\draw[fill=blue!10,opacity=0.6] (x3)--(x2)--(x6)--cycle;
				\draw[fill=blue!10,opacity=0.6] (x2)--(x4)--(x6)--cycle;
				\draw[fill=blue!10,opacity=0.6] (x4)--(x1)--(x6)--cycle;
				
				% ---- 顶点标签 ----
				\node[anchor=west] at (x1) {$x_1$};
				\node[anchor=east] at (x2) {$x_2$};
				\node[anchor=south] at (x3) {$x_3$};
				\node[anchor=north] at (x4) {$x_4$};
				\node[anchor=south east] at (x5) {$x_5$};
				\node[anchor=north west] at (x6) {$x_6$};
			\end{tikzpicture}
			\caption{Octahedron $P_1$ with faces $x_1x_3x_5$, $x_1x_4x_5$, $x_1x_3x_6$, $x_1x_4x_6$, $x_2x_3x_5$, $x_2x_4x_5$, $x_2x_3x_6$, and $x_2x_4x_6$.}
			\label{fig:p1}
		\end{subfigure}
		\hfill
		% ==== 子图 (b) ====
		\begin{subfigure}[t]{0.45\textwidth}
			\centering
			\tdplotsetmaincoords{70}{110}
			\begin{tikzpicture}[tdplot_main_coords, scale=2, line join=round, line cap=round]
				% ---- 顶点定义 ----
				\coordinate (x1) at (1,0,0);
				\coordinate (x2) at (-1,0,0);
				\coordinate (x3) at (0,1,0);
				\coordinate (x4) at (0,-1,0);
				\coordinate (x5) at (0,0,1);
				\coordinate (x6) at (0,0,-1);
				
				% ---- 面绘制 ----
				\draw[fill=blue!20,opacity=0.7] (x1)--(x3)--(x5)--cycle;
				\draw[fill=blue!20,opacity=0.7] (x3)--(x2)--(x5)--cycle;
				\draw[fill=blue!20,opacity=0.7] (x2)--(x4)--(x5)--cycle;
				\draw[fill=blue!20,opacity=0.7] (x4)--(x1)--(x5)--cycle;
				
				\draw[fill=blue!10,opacity=0.6] (x1)--(x3)--(x6)--cycle;
				\draw[fill=blue!10,opacity=0.6] (x3)--(x2)--(x6)--cycle;
				\draw[fill=blue!10,opacity=0.6] (x2)--(x4)--(x6)--cycle;
				\draw[fill=blue!10,opacity=0.6] (x4)--(x1)--(x6)--cycle;
				
				\draw[fill=red!10,opacity=0.6] (x1)--(x3)--(x2)--(x4)--cycle;
				\draw[fill=yellow!10,opacity=0.6] (x1)--(x5)--(x2)--(x6)--cycle;
				\draw[fill=green!10,opacity=0.6] (x3)--(x5)--(x4)--(x6)--cycle;
				
				\draw[dashed] (x1)--(x2);
				\draw[dashed] (x3)--(x4);
				\draw[dashed] (x5)--(x6);
				
				% ---- 顶点标签 ----
				\node[anchor=west] at (x1) {$x_1$};
				\node[anchor=east] at (x2) {$x_2$};
				\node[anchor=south] at (x3) {$x_3$};
				\node[anchor=north] at (x4) {$x_4$};
				\node[anchor=south east] at (x5) {$x_5$};
				\node[anchor=north west] at (x6) {$x_6$};
			\end{tikzpicture}
			\caption{Polyhedron $P_2$, obtained by adding three squares ($x_1x_3x_2x_4$, $x_1x_5x_2x_6$, $x_3x_5x_4x_6$) to $P_1$.}
			\label{fig:p2}
		\end{subfigure}
		
		% ==== 总说明 ====
		\caption{The key difference between induced cycle and chordless cycle. }
		\label{fig:comparison}
	\end{figure}

	Consider the complete tripartite graph $K_{2,2,2}$. If each induced cycle of $K_{2,2,2}$ is filled in with a 2-dimensional disk, then the resulting complex is an octahedron, as shown in Figure~\ref{fig:p1}. If, instead, each chordless cycle of $K_{2,2,2}$ is filled in with a 2-dimensional disk, we obtain the complex depicted in Figure~\ref{fig:p2}. It is straightforward to verify that both constructions preserve the $1$-skeleton. Moreover, the complex in Figure~\ref{fig:p1} admits an embedding into $\mathbb{R}^3$, whereas the complex in Figure~\ref{fig:p2} does not admit an embedding into $\mathbb{R}^3$. 
	Since the ultimate goal of this paper is to study graph coloring problems, and the dimension-raising construction illustrated in Figure~\ref{fig:p1} yields a simpler resulting structure, it is more appropriate to employ induced cycles for increasing the dimension.
	
	To further increase the dimension of a graph $G$, one may generalize induced cycles to induced $i$-spheres and then regard each induced $i$-sphere as the boundary of an $(i+1)$-ball. It should be noted, however, that chordless cycles and their higher-dimensional analogues, namely chordless $i$-spheres, play a crucial role in our proofs; these notions will be discussed in detail in subsequent sections. 
	
	\begin{definition}[higher-dimensional chord]
		Let $G$ be a $(d-1)$-dimensional CW complex and let $S^i$ be an $i$-sphere in $G$. 
		A \emph{$k$-dimensional chord} of $S^i$ ($1 \le k \le i$) is a $k$-dimensional cell $B^k$ in $G$ such that
		\begin{itemize}
			\item $\partial B^k \subseteq S^i$, and
			\item $\operatorname{int}(B^k) \cap S^i = \emptyset$.
		\end{itemize}
		%If no such $k$-dimensional chord exists for any $1 \le k \le i$, then $S^i$ is called a \emph{chordless $i$-sphere}.
	\end{definition}
	
	\begin{definition}[induced $i$-sphere]\label{induced}
		An induced $i$-sphere $S^i$ of a $(d-1)$-dimensional CW complex $G$ is a sphere that is a subcomplex of $G$ with no higher-dimensional chords such that $G \backslash S^i$ is connected. It should be noted that an $1$-dimensional CW complex is equivalent to a simple graph. 
	\end{definition}
	
	\begin{definition}[chordless $i$-sphere]\label{chordless}
		A chordless $i$-sphere $S^i$ of a $(d-1)$-dimensional CW complex $G$ is a sphere that is a subcomplex of $G$ with no higher-dimensional chords. 
	\end{definition}

	The CW complex obtained from $G$ via this dimension-raising process is referred to as a $(d-1)$-dimensional topological hypergraph, which we define as follows.

	\begin{definition}[$i$-face of an $n$-dimensional polytope \cite{hatcher2002algebraic}]
		An $i$-face of an $n$-dimensional polytope is a face of the polytope that is itself an $i$-dimensional polytope, where $0 \leq i \leq n-1$. Specifically:
		\begin{itemize}
			\item A $0$-face is a vertex.
			\item A $1$-face is an edge.
			\item A $2$-face is a polygonal face.
			\item And so on, up to the $(n-1)$-face, which is a $(n-1)$-dimensional polytope.
		\end{itemize}
	\end{definition}
	
	\begin{definition}[{\it $(d-1)$-dimensional topological hypergraph}]\label{new-topological hypergraph}
		
		Let $T_1, T_2, \dots, T_m$ be $(d-1)$-dimensional convex polytopes (which are allowed to undergo topological transformations under homeomorphism, provided that their interiors do not intersect). We say $K= \bigcup \limits_{i = 1}^m T_i$ is a {\it $(d-1)$-dimensional topological hypergraph} if the following conditions hold (each $T_i$ can be regarded as the $(d-1)$-dimensional topological hyperedge of $K$): 
		
		\begin{itemize}
			\item If $T_i$ and $T_j$ are convex polytopes in $\{T_1, T_2, ...,T_m\}$ and $T' = T_i\cap T_j \neq \emptyset$, then $T' = T_i \cap T_j$ is the union of the common $i$-faces of $T_i$ and $T_j$ (with $i \leq d-2$). In other words, $T_i$ and $T_j$ may intersect only along their faces. 
			\item The $i$-th homotopy group of $K$ is trivial for all $i\in \{1,2,3,...,d-2\}$. 
			\item  For any $(d-2)$-face $J$ of $K$, there exist $T_i, T_j \in \{T_1, T_2, \ldots, T_m\}$ such that $J = T_i \cap T_j$. 
			A face $J$ satisfying this condition is called a \emph{non-pendant $i$-topological hyperedge}; otherwise, $J$ is called a \emph{pendant $i$-topological hyperedge}.
			
		\end{itemize}
		
	\end{definition}
	
	Definition~\ref{new-topological hypergraph} can be interpreted either as a high-dimensional polytope or a CW complex. In this paper, however, we adopt a graph perspective to describe this object. In Section~\ref{geo}, we will present a procedure that transforms an arbitrary graph $G$ into such a CW complex.

	%\begin{definition}[{\it pendant $i$-topological hyperedge}]
	%	Let $T$ be a $(d-1)$-dimensional topological hypergraph of a $(d-1)$-dimensional topological hypergraph $K$
	%	and $N_{K(d-2)}(T)$ represent the set of all $(d-2)$-dimensional polytopes that are incident to $T$. 
	%	If 
	%	$\forall\, J \in N_{K(d-2)}(T)$, there exists a $(d-1)$-dimensional polytope $T' \subseteq K$ such that $J= T \cap T'$, then $T$ is called {\it non-pendant $i$-topological hyperedge}, otherwise $T$ is called {\it pendant $i$-topological hyperedge}. 
	%\end{definition}

	\begin{definition}[{\it induced topological sub-hypergraph}]
		\label{def:vertex-induced-topological-subhypergraph}
		
		Let $K = \bigcup_{i=1}^m T_i$ be a $(d-1)$-dimensional topological hypergraph as in Definition~\ref{new-topological hypergraph}.  
		Let $V(K)$ denote the set of all vertices of $K$, and let $V' \subseteq V(K)$ be a subset of vertices.
		
		Define
		\[
		\mathcal{T}(V') := \{\, T_i \in \{T_1, T_2, \dots, T_m\} \mid V(T_i) \subseteq V' \,\},
		\]
		where $V(T_i)$ denotes the set of vertices of the $(d-1)$-dimensional topological hyperedge $T_i$.
		
		The \emph{vertex-induced $(d-1)$-dimensional topological sub-hypergraph} of $K$ induced by $V'$, denoted by $K[V']$, is defined as
		\[
		K[V'] := \bigcup_{T_i \in \mathcal{T}(V')} T_i .
		\]
		
	\end{definition}
	
	According to Definitions~\ref{induced} and \ref{chordless}, we know that an induced (or chordless) $i$-sphere $S^i$ of a $(d-1)$-dimensional topological hypergraph $U$ is an induced subgraph of $U$. 
	
	If we further require that every $(d-1)$-dimensional topological hypergraph be homeomorphic to a $(d-1)$-simplex, we obtain a special class of topological hypergraphs. 
	In this setting, the term \emph{$d$-uniform} means that each $(d-1)$-dimensional topological hyperedge contains exactly $d$ vertices.
	
	\begin{definition}[$d$-uniform $(d-1)$-dimensional topological hypergraph]\label{uniform}
		Let $U$ be a $(d-1)$-dimensional topological hypergraph. If each $(d-1)$-dimensional topological hyperedge of $U$ is homeomorphic to a $(d-1)$-simplex, then $U$ is called the $d$-uniform $(d-1)$-dimensional topological hypergraph. 
	\end{definition}

	\subsection{Overview of the paper}
	
	\begin{figure}[htbp]
		%	\centering
		
		\begin{tikzpicture}[
			node distance=1.2cm and 1.5cm,
			font=\sffamily\small,
			% Define styles
			base/.style = {rectangle, rounded corners, draw=black, thick, minimum width=4cm, minimum height=1cm, align=center, drop shadow},
			def/.style = {base, fill=blue!10},
			lemma/.style = {base, fill=orange!10},
			proofstep/.style = {rectangle, draw=red!80, thick, fill=red!5, minimum width=3.5cm, minimum height=0.8cm, align=center},
			app/.style = {base, fill=green!10},
			line/.style = {draw, -Latex, thick},
			labeltext/.style = {fill=white, font=\footnotesize\itshape, text=gray}
			]
			
			% --- SECTION 1: DEFINITIONS & FRAMEWORK (Items 1-4) ---
			\node[def] (conn) {1. Sphere Connectivity in $\mathbb{R}^d$};
			\node[def, below=0.6cm of conn] (s_decomp) {2. $S$-decomposition in $\mathbb{R}^d$};
			\node[def, below=0.6cm of s_decomp] (dim_func) {3. Dimensional-raising Function $U^x(G)$ \\ (Graph $\to$ CW complex)};
			\node[def, below=0.6cm of dim_func] (hyp_def) {4. Definition: $\mathbb{R}^d$-hypergraphs \\ \& Topological Properties};
			
			% --- SECTION 2: LEMMAS & TOOLS (Items 5-8) ---
			% Placing these to the right or below to feed into the main proof
			\node[lemma, right=2cm of conn] (bridges) {5. Lemma \ref{bridges}: Overlapping Bridges \\ (Skew or Equivalent)};
			\node[lemma, below=0.6cm of bridges] (sphere) {6. Lemma \ref{sphere}: Ear Decomposition \\ \& Neighbors on Common Sphere};
			\node[lemma, below=0.6cm of sphere] (comp_lemma) {7. Lemma \ref{10.34}: Embeddability $\iff$ \\ Marked $S$-components Embeddable};
			\node[lemma, below=0.6cm of comp_lemma] (contract) {8. Lemma \ref{connected}: Contraction $G/e$ \\ Preserves 2-sphere Connectivity};
			
			% --- SECTION 3: THE MAIN PROOF FLOW (Item 9 Detailed) ---
			% Using a cluster/grouping for the proof logic
			
			\node[draw=red!80, dashed, inner sep=15pt, fit={(contract) (bridges)}, label=above:\textbf{9. Proof of Theorem \ref{anti-minor} (Sufficiency)}] (proof_box) {};
			
			% Positioning the proof steps below the definitions and tools
			\node[proofstep, below=2cm of hyp_def] (p_start) {\textbf{Start Proof}\\ Theorem \ref{anti-minor} Necessity is Straightforward};
			
			\node[proofstep, below=0.8cm of p_start] (p_decomp) {Apply $U^x(G)$ \& $S$-decomposition};
			
			\node[proofstep, right=1.5cm of p_decomp] (p_lemma1034) {Apply Lemma \ref{10.34}:\\ Focus on Non-embeddable \\ marked $S$-component $G_0$};
			
			\node[proofstep, right=1.5cm of p_lemma1034] (p_induct) {Induction on Edges};
			
			\node[proofstep, below=1cm of p_induct] (p_contract) {Contract Edge $e=xy$ \\ (2-sphere connectivity)};
			
			\node[proofstep, left=1.5cm of p_contract] (p_analyze) {Structural Analysis \\ Neighbors on Sphere (Lemma \ref{sphere}) \\ Bridges Skew/Equiv (Lemma \ref{bridges})};
			
			\node[proofstep, left=1.5cm of p_analyze] (p_minors) {Identify Minors in $G_0$ \\ $\implies$ Minors in $G$};
			
			\node[proofstep, below=0.8cm of p_minors] (p_end) {\textbf{Conclusion}\\ Theorem \ref{anti-minor} Follows};
			
			% --- SECTION 4: APPLICATIONS (Item 10) ---
			\node[app, below=1cm of p_end, xshift=4cm] (apps) {10. Applications: \\ $\bullet$ Chromatic Number \\ $\bullet$ High-dimensional Discharging Method \\ $\bullet$ The Hadwiger Conjecture};
			
			% --- CONNECTIONS ---
			
			% Flow between definitions
			\draw[line] (conn) -- (s_decomp);
			\draw[line] (s_decomp) -- (dim_func);
			\draw[line] (dim_func) -- (hyp_def);
			
			% Definitions feeding into Proof
			\draw[line] (hyp_def) -- (p_start);
			
			% Proof Internal Flow
			\draw[line] (p_start) -- (p_decomp);
			\draw[line] (p_decomp) -- (p_lemma1034);
			\draw[line] (p_lemma1034) -- (p_induct);
			\draw[line] (p_induct) -- (p_contract);
			\draw[line] (p_contract) -- (p_analyze);
			\draw[line] (p_analyze) -- (p_minors);
			\draw[line] (p_minors) -- (p_end);
			
			% Linking Lemmas to Proof Steps (Crucial part)
			\draw[line, dashed, orange!80] (comp_lemma.west) to[out=180,in=90] (p_lemma1034.north);
			\draw[line, dashed, orange!80] (contract.east) to[out=0,in=90] (p_contract.north);
			\draw[line, dashed, orange!80] (sphere.east) -| (p_analyze.north);
			\draw[line, dashed, orange!80] (bridges.east) -| ($(p_analyze.north)+(0.5,0)$);
			
			% Link to Apps
			\draw[line] (p_end) -- (apps);
			
			% Background box for proof area to make it distinct
			\begin{pgfonlayer}{background}
				\node[fill=red!5, rounded corners, fit=(p_start) (p_induct) (p_end) (p_minors), inner sep=10pt] {};
			\end{pgfonlayer}
			
		\end{tikzpicture}
		
		\caption{Proof outline.}
		\label{outline}
	\end{figure}
	
	As shown in Figure~\ref{outline}, our proof follows a strategy inspired by the proof of Wagner's theorem \cite{bondy2008graph}. 
	Within this framework, we prove Theorem~\ref{anti-minor}.

	In Section \ref{geo}, we give the definition of $U^{d-1}(G)$, and prove $U^{d-1}(G)$ is a $(d-1)$-dimensional topological hypergraph. 
	In Section \ref{section3}, we define a class of $(d-1)$-dimensional topological hypergraph that can embed into $\mathbb{R}^d$, which can be regarded as the high-dimensional analogue of planar graphs; we call such structures $\mathbb{R}^d$-hypergraphs. Subsequently, by analogy with planar graphs, we establish definitions and theorems concerning bridges (Sections \ref{section-Br}), ear decomposition (Section \ref{section-ED}), $S$-component and connectivity (Section \ref{connectivity}) in $\mathbb{R}^d$. In Section \ref{recognizing}, we establish two main results, and reformulate the Hadwiger conjecture into problems of high-dimensional embedding and high-dimensional coloring. 
	%In Section \ref{discharging}, we extend the discharging method to $\mathbb{R}^d$ and apply it to investigate coloring problems in high-dimensional spaces. 
	
	%\begin{theorem}\label{anti-minor}
	%	Let $G$ be a $(d+1)$-connected graph, $U^i(G)$ denote the dimension-raising function of $G$, then $U^{d-1}(G)$ embeds into $\mathbb{R}^d$ if $G$ contains neither $K_{d+3}$ nor $K_{3,d+1}$ as a minor. 
	%\end{theorem}

	The following theorems are main results of this paper. If a complete $p$-partite graph has exactly $q$ vertices in each part, then we denote it by $K_{q \otimes p}$. For example, the graph $K_{2,2,2,2}$ can be written as $K_{2 \otimes 4}$, and $K_{1,2,2,2}$ can be written as $K_{1,2 \otimes 3}$.
	And the symbol $\wedge$ denotes the join of graphs, which is defined as follows. 
	
	\begin{definition}[join of graphs]
		Let $G_1 = (V_1, E_1)$ and $G_2 = (V_2, E_2)$ be two vertex-disjoint graphs.  
		The \emph{join} of $G_1$ and $G_2$, denoted by $G_1 \wedge G_2$, 
		is the graph with vertex set
		\[
		V(G_1 \wedge G_2) = V_1 \cup V_2
		\]
		and edge set
		\[
		E(G_1 \wedge G_2) = E_1 \cup E_2 \cup \{ uv \mid u \in V_1,\ v \in V_2 \}.
		\]
		That is, $G_1 \wedge G_2$ is obtained from the disjoint union of $G_1$ and $G_2$
		by adding all possible edges between every vertex of $G_1$ and every vertex of $G_2$.
	\end{definition}

	\begin{definition}[union of graphs]
		Let $G_1=(V_1,E_1)$ and $G_2=(V_2,E_2)$ be graphs.  
		The \emph{union} of $G_1$ and $G_2$ is the graph
		\[
		G_1 \cup G_2 := (V_1 \cup V_2,\; E_1 \cup E_2).
		\]
	\end{definition}

	\begin{definition}[complement graph]
		Let $G=(V,E)$ be a simple graph.  
		The \emph{complement} of $G$, denoted by $\overline{G}$, is the graph with vertex set $V$ such that two distinct vertices $u,v \in V$ are adjacent in $\overline{G}$ if and only if they are not adjacent in $G$, that is,
		\[
		E(\overline{G})=\bigl\{\{u,v\}\subseteq V : u\neq v \text{ and } \{u,v\}\notin E(G)\bigr\}.
		\]
	\end{definition}

	%\begin{theorem}\label{anti-minor}
	%	Let $G$ be a $(d+1)$-connected graph, $U^i(G)$ denote the dimension-raising function of $G$, then $U^{d-1}(G)$ embeds into $\mathbb{R}^d$ if $G$ does not contain any minors from the set $\{K_{d+3}, K_{3,d+1}, K_{4,d}, \dots, K_{\lfloor \frac{1}{2}(d+4) \rfloor, \lceil \frac{1}{2}(d+4) \rceil}\}$. 
	%\end{theorem}
	
	\begin{theorem}\label{anti-minor}
		Let $G$ be a graph and $U^i(G)$ denote the dimension-raising function of $G$ (Definition~\ref{dimension-raising}). Then $U^{d-1}(G)$ embeds into $\mathbb{R}^d$ if and only if: 
		%$G$ does not contain one of the graphs in Table \ref{minor} as a minor. 
		\begin{itemize}
			\item $G$ does not contain $K_6$ and $K_{3,4}$ as a minor for $d=3$; 
			\item $G$ does not contain $K_7$, $K_{4,4}$ and $\overline{K_3} \wedge (K_1 \cup K_4)$ as a minor for $d=4$; 
			\item $G$ does not contain one of the graphs in Table \ref{minor} as a minor for $d\geq 5$. 
		\end{itemize}
		
		\begin{table}[htbp]
			\centering
			\renewcommand{\arraystretch}{1.4}
			\begin{tabular}{|c|p{11cm}|}
				\hline
				\textbf{Dimension $d\ge 5$} & \textbf{Minors} \\
				\hline
				
				$d=4k$ or $d=4k+2$
				&
				$
				\begin{aligned}
					&K_{d+3},\quad
					K_{4,4}\wedge K_{d-4},\quad
					K_{4,2\otimes \frac{d}{2}},\\
					&\overline{K_3}\wedge (K_1 \cup K_d),\quad
					\overline{K_4}\wedge (K_1 \cup K_{d-1}),\\
					&(K_1\cup K_{j-1})\wedge (K_1\cup K_{d+3-j}) \ \text{if $j\ge 5$, $d-j\geq 1$},\\
					&\overline{K_4}\wedge K_{j-4}\wedge (K_1\cup K_{d+3-j}) \ \text{if $j\ge 4$, $d-j\geq 1$},\\
					&\left\{
					\begin{aligned}
						K_{2\otimes \frac{j}{2}}\wedge (K_1\cup K_{d+3-j})
						&\ \text{if $j$ even, $j\ge 6$, $d-j \ge 1$},\\
						K_{1,2\otimes \frac{j-1}{2}}\wedge (K_1\cup K_{d+3-j})
						&\ \text{if $j$ odd, $j\ge 7$, $d-j \ge 1$}.
					\end{aligned}
					\right.
				\end{aligned}
				$
				\\
				\hline
				
				$d=4k+1$
				&
				$
				\begin{aligned}
					&K_{d+3},\quad
					K_{4,4}\wedge K_{d-4},\quad
					K_{4,1,2\otimes \frac{d-1}{2}},\\
					&K_{3,2\otimes \frac{d+1}{2}},\quad
					\overline{K_3}\wedge (K_1 \cup K_d),\\
					&\overline{K_4}\wedge (K_1 \cup K_{d-1}),\\
					&(K_1\cup K_{j-1})\wedge (K_1\cup K_{d+3-j}) \ \text{if $j\ge 5$, $d-j\geq 1$},\\
					&\overline{K_4}\wedge K_{j-4}\wedge (K_1\cup K_{d+3-j}) \ \text{if $j\ge 4$, $d-j\geq 1$},\\
					&\left\{
					\begin{aligned}
						K_{2\otimes \frac{j}{2}}\wedge (K_1\cup K_{d+3-j})
						&\ \text{if $j$ even, $j\ge 6$, $d-j \ge 1$},\\
						K_{1,2\otimes \frac{j-1}{2}}\wedge (K_1\cup K_{d+3-j})
						&\ \text{if $j$ odd, $j\ge 7$, $d-j \ge 1$}.
					\end{aligned}
					\right.
				\end{aligned}
				$
				\\
				\hline
				
				$d=4k+3$
				&
				$
				\begin{aligned}
					&K_{d+3},\quad
					K_{4,4}\wedge K_{d-4},\quad
					K_{4,1,2\otimes \frac{d-1}{2}},\\
					&K_{3,2\otimes \frac{d+1}{2}},\quad
					\overline{K_3}\wedge (K_1 \cup K_d),\\
					&\overline{K_4}\wedge (K_1 \cup K_{d-1}),\\
					&(K_1\cup K_{j-1})\wedge (K_1\cup K_{d+3-j}) \ \text{if $j\ge 5$, $d-j\geq 1$},\\
					&\overline{K_4}\wedge K_{j-4}\wedge (K_1\cup K_{d+3-j}) \ \text{if $j\ge 4$, $d-j\geq 1$},\\
					&\left\{
					\begin{aligned}
						K_{2\otimes \frac{j}{2}}\wedge (K_1\cup K_{d+3-j})
						&\ \text{if $j$ even, $j\ge 6$, $d-j \ge 1$},\\
						K_{1,2\otimes \frac{j-1}{2}}\wedge (K_1\cup K_{d+3-j})
						&\ \text{if $j$ odd, $j\ge 7$, $d-j \ge 1$}.
					\end{aligned}
					\right.
				\end{aligned}
				$
				\\
				\hline
			\end{tabular}
			\caption{Minors: either it is the complete graph $K_{d+3}$, or it is a supergraph of the complete bipartite graph 
				$K_{i,\, d+4-i}$, where $i \in \{2, 3, \dots, \lfloor \tfrac{d+4}{2} \rfloor \}$.}
			\label{minor}
		\end{table}

	\end{theorem}
	
	In the higher-dimensional setting, discussing the necessary and sufficient conditions for embeddability requires considerable exposition, since the structure of higher-dimensional spaces is far more complex than that of the plane. Fortunately, we observe that in order to establish the connection between the Hadwiger Conjecture and higher-dimensional embeddings, it suffices to use only the sufficient condition. Moreover, every $(d-1)$-dimensional topological hypergraph that is not embeddable in $\mathbb{R}^d$ necessarily contains two very simple classes of $1$-skeletons. Therefore, in what follows, Theorem~\ref{anti-minor} immediately implies Corollary~\ref{anti-minor2}. 
	
	\begin{corollary}\label{anti-minor2}
		Let $G$ be a graph and $U^i(G)$ denote the dimension-raising function of $G$ (Definition~\ref{dimension-raising}). Then $U^{d-1}(G)$ embeds into $\mathbb{R}^d$ if $G$ does not contain both $K_{d+3}$ and $K_{i, d+4-i}$ ($i \in \{2, 3, \dots, \lfloor \frac{d+4}{2} \rfloor \}$) as a minor. 
		%	Let $G$ be a graph, $U^i(G)$ denote the dimension-raising function of $G$, then $U^{d-1}(G)$ embeds into $\mathbb{R}^d$ if $G$ contains neither $K_{d+3}$ nor $K_{i,d+4-i}$ ($i\in \{2, 3, \dots, \lfloor\frac{d+4}{2}\rfloor\}$) as a minor. 
		%	
		%	does not contain one of the following graphs as a minor. 
		%	\begin{itemize}
			%		\item $K_{d+3}$. 
			%		\item $K_{i,d+4-i}$ for $i\in \{2, 3, \dots, \lfloor\frac{d+4}{2}\rfloor\}$. 
			%	\end{itemize}
		%	\begin{itemize}
			%		\item If $d=4k$: $K_{d+3}$ and $K_{\frac{d+4}{2},\frac{d+4}{2}}$. 
			%		\item If $d=4k+1$ and $d=4k+3$: $K_{d+3}$ and $K_{\frac{d+3}{2},\frac{d+5}{2}}$.  
			%		\item If $d=4k+2$: $K_{d+3}$ and $K_{\frac{d+2}{2},\frac{d+6}{2}}$.  
			%%		\item If $d=4k+3$: $K_{\frac{d+3}{2},\frac{d+5}{2}}$.  
			%	\end{itemize}
		
	\end{corollary}

	%$\{K_{d+3}, K_{3,d+1}, K_{4,d}, \dots, K_{\lfloor \frac{1}{2}d \rfloor,d+4-\lfloor \frac{1}{2}d \rfloor}\}$
	
	In this paper, we do not pursue a detailed study of the coloring problem. 
	Instead, we provide only a coarse estimate of the upper bound on the chromatic number, leading to the following theorem. 
	We conjecture that the chromatic number of $G$ satisfies $\chi(G) \leq d+2$.
	
	\begin{theorem}\label{chromatic-number}
		Let $G$ be a graph and $U^i(G)$ denote the dimension-raising function of $G$ (Definition~\ref{dimension-raising}). If $U^{d-1}(G)$ embeds into $\mathbb{R}^d$, then $\chi(G)\leq 3\cdot 2^{d-1}$. 
	\end{theorem}
	
	In Section \ref{discharging}, we extend the Discharging method, originally developed for studying coloring problems of planar graphs, to the setting of high-dimensional spaces. Using this approach, we prove that every $d$-uniform $\mathbb{R}^d$-hypergraph has $(d-2)$-dimensional chromatic number at most $d+3$ (which means the chromatic number of $(d-2)$-face). This result illustrates how classical techniques for planar graphs can be effectively generalized to higher-dimensional spaces. 
	Finally, in Section \ref{future}, we discuss our future work.

	%\begin{theorem}\label{apply-for-discharging}
	%	Let $H$ be a $d$-uniform $\mathbb{R}^d$-hypergraph with $d\geq 3$. Then $\chi^i_{\mathbb{R}^d}(H) \leq d + 3$. 
	%\end{theorem}

	%\section{Properties of $\mathbb{R}^d$-hypergraph}\label{geo}
	
	\section{The geometrization of a graph}\label{geo}
	
	In this section, we introduce a dimension-raising procedure that transforms a graph $G$ into an $x$-dimensional CW complex $U^{x}(G)$. We further prove that the CW complex $U^{x}(G)$ obtained by this procedure is an $x$-dimensional topological hypergraph.

	\subsection{Higher-dimensional connectivity: sphere connectivity}
	
	In graph theory, vertex-connectivity is a fundamental notion that measures the robustness of a graph under vertex deletions.
	From a topological perspective, this notion depends only on the $1$-skeleton of the underlying space. 
	While such a description is adequate for graphs, it becomes insufficient when one attempts to capture the connectivity properties of higher-dimensional CW complexes.
	
	For a general CW complex, connectivity phenomena are not fully determined by its $1$-skeleton.
	Higher-dimensional cells may create multiple independent ways of linking lower-dimensional faces that are invisible at the graph level.
	In particular, two faces may appear weakly connected when viewed through the $1$-skeleton, while in fact being linked by several essentially different higher-dimensional spheres.
	This indicates that classical vertex-connectivity fails to reflect the true structural complexity of higher-dimensional spaces.
	
	Motivated by this observation, we seek a higher-dimensional analogue of vertex-connectivity that is sensitive to the full cell structure of a CW complex.
	Rather than counting internally disjoint paths, our approach counts independent $(d-1)$-spheres simultaneously containing a given pair of faces.
	This leads naturally to a notion of \emph{sphere connectivity}, which extends the classical concept of connectivity from graphs to higher-dimensional CW complexes and provides a framework for measuring connectivity beyond the $1$-skeleton.

	We first reinterpret vertex-connectivity in graph theory from a topological viewpoint.
	Let $G$ be a graph and let $u,v \in V(G)$.  
	A classical result states that $G$ is $k$-connected if and only if between any two distinct vertices $u$ and $v$ there exist $k$ internally vertex-disjoint $u$--$v$ paths.
	Each pair of such paths determines a $1$-sphere (a simple cycle) containing $\{u,v\}$.
	
	Let $\mathbb{S}$ denote the collection of all $1$-spheres in $G$ that contain both $u$ and $v$.
	We say that a subcollection $\mathbb{S}' \subseteq \mathbb{S}$ is \emph{linearly independent} if it satisfies:
	\begin{itemize}
		\item For every $S \in \mathbb{S}'$,
		\[
		S \nsubseteq \bigcup_{S_i \in \mathbb{S}' \setminus \{S\}} S_i ;
		\]
		\item For any two distinct spheres $S_i, S_j \in \mathbb{S}'$,
		\[
		\{u,v\} \subseteq S_i \cap S_j .
		\]
	\end{itemize}
	A maximal linearly independent subcollection of $\mathbb{S}$ is called a \emph{basis} of $\mathbb{S}$.
	In particular, a $3$-connected graph admits a basis consisting of two $1$-spheres for any pair of vertices.
	
	\medskip
	
	We now generalize this idea to higher dimensions.
	Let $U$ be a $(d-1)$-dimensional topological hypergraph, and let $u,v$ be two $i$-topological hyperedges of $U$ with $i \le d-2$
	(in particular, when $i=0$, $u$ and $v$ are vertices).
	Let $\mathbb{S}$ be a collection of $(d-1)$-spheres in $U$ such that each sphere $S \in \mathbb{S}$ contains both $u$ and $v$.
	
	A subcollection $\mathbb{S}' \subseteq \mathbb{S}$ is said to be \emph{linearly independent} if:
	\begin{itemize}
		\item For every $S \in \mathbb{S}'$,
		\[
		S \nsubseteq \bigcup_{S_i \in \mathbb{S}' \setminus \{S\}} S_i ;
		\]
		\item For any two distinct spheres $S_i, S_j \in \mathbb{S}'$,
		$S_i \cap S_j$ is a $(d-1)$-ball and
		\[
		\{u,v\} \subseteq \partial ( S_i \cap S_j ),
		\]
		where $\partial ( S_i \cap S_j )$ denotes the boundary of the ball.
	\end{itemize}
	
	A maximal linearly independent subcollection $\mathbb{S}'$ is called a \emph{basis} of $\mathbb{S}$.
	The cardinality of such a basis is called the \emph{$(d-1)$-cosphere number} of $u$ and $v$.
	
	If for every pair of $i$-topological hyperedges $u,v$ of $U$,  the $(d-1)$-cosphere number is at least $k$,
	then $U$ is called \emph{$k$-sphere connected}.
	The largest integer $k$,  for which $U$ is $k$-sphere connected,  is called the
	\emph{$(d-1)$-sphere connectivity} of $U$, and is denoted by $\kappa^{d-1}(U)$.

	For clarity, we now separate the above construction into several basic notions. 
	We first define what it means for a collection of spheres connecting two fixed faces to be independent. 
	Based on this, we introduce the notions of a sphere basis and the associated cosphere number. 
	Finally, these local quantities are used to define a global connectivity invariant for a CW complex (topological hypergraph), called $(d-1)$-sphere connectivity.

	Let $U$ be a $(d-1)$-dimensional topological hypergraph. Let $u$ and $v$ be two distinct $i$-topological hyperedges of $U$, where $0 \le i \le d-2$. We denote by $\mathbb{S}(u,v)$ the set of all subcomplexes in $U$ that are homeomorphic to a $(d-1)$-sphere and contain both $u$ and $v$. 
	The symmetric difference appearing in Definition~\ref{independence} is taken in the sense of higher-dimensional complexes; for the relevant background we refer the reader to~\cite{hatcher2002algebraic}.

	\begin{definition}[sphere independence]\label{independence}
		A subset $\mathbb{S}' \subseteq \mathbb{S}(u,v)$ is said to be \textit{sphere independent} (or linearly independent) if it satisfies the following two conditions:
		\begin{itemize}
			\item \textbf{Independent:} No sphere in $\mathbb{S}'$ can be expressed as the symmetric difference of a finite number of other spheres in $\mathbb{S}'$. 
			%		For any sphere $S \in \mathbb{S}'$, $S$ is not contained in the union of the other spheres in $\mathbb{S}'$. That is:
			%		\[
			%		S \nsubseteq \bigcup_{S_k \in (\mathbb{S}' \setminus \{S\})} S_k
			%		\]
			\item \textbf{Geometric Intersection:} For any distinct pair $\{S_i, S_j\} \subseteq \mathbb{S}'$, their intersection $S_i \cap S_j$ is homeomorphic to a $(d-1)$-ball, and the faces $u, v$ are contained in the boundary of this intersection:
			\[
			S_i \cap S_j \cong B^{d-1} \quad \text{and} \quad \{u, v\} \subseteq \partial(S_i \cap S_j).
			\]
		\end{itemize}
	\end{definition}
	
	\begin{definition}[sphere basis and cosphere number]
		A \textit{sphere basis} of $\mathbb{S}(u,v)$ is defined as a maximal sphere independent subset $\mathbb{S}' \subseteq \mathbb{S}(u,v)$. The cardinality of such a basis, $|\mathbb{S}'|$, is defined as the \textit{$(d-1)$-cosphere number} of $u$ and $v$, denoted by $\nu^{d-1}(u,v)$.
	\end{definition}
	
	\begin{definition}[sphere connectivity]\label{sphere connectivity}
		The topological hypergraph $U$ is said to be \textit{$k$-sphere connected} if for every pair of $i$-topological hyperedges $u, v$ in $U$ (where $0 \le i \le d-2$), the $(d-1)$-cosphere number satisfies $\nu^{d-1}(u,v) \ge k$.
		
		The \textit{$(d-1)$-sphere connectivity} of $U$, denoted by $\kappa^{d-1}(U)$, is the largest integer $k$ such that $U$ is $k$-sphere connected:
		\[
		\kappa^{d-1}(U) = \min_{u,v \subset U} \{ \nu^{d-1}(u,v) \}
		\]
	\end{definition}
	
	%\begin{definition}[sphere cut]
	%	Let $U$ be a $(d-1)$-dimensional topological hypergraph with $d \ge 3$.
	%	A subset of vertices $S \subseteq V(U)$ with $|S|$ vertices is called a \emph{$|S|$-sphere cut} of $U$ if it satisfies the following conditions:
	%	\begin{itemize}
		%		\item \textbf{Disconnection:}
		%		Let $U - S$ denote the topological hypergraph obtained from $U$ by deleting all $i$-faces
		%		($0 \le i \le d-1$) that are incident to at least one vertex in $S$.
		%		Then $U - S$ is disconnected. 
		%		
		%		\item \textbf{Spherical boundary structure:}
		%		The induced sub-hypergraph $U[S]$ is a disjoint union of $(d-2)$-spheres.
		%	\end{itemize}
	%\end{definition}
	
	\begin{definition}[sphere cut]\label{sphere cut}
		Let $U$ be a $(d-1)$-dimensional topological hypergraph with $d \ge 3$.
		A subset of vertices $S \subseteq V(U)$ is called a \emph{$(d-2)$-dimensional $k$-sphere cut} of $U$ if it satisfies the following two conditions:
		\begin{itemize}
			\item \textbf{Disconnection:}
			Let $U - S$ denote the topological hypergraph obtained from $U$ by deleting all $i$-faces
			($0 \le i \le d-1$) that are incident to at least one vertex in $S$.
			Then $U - S$ is disconnected;
			
			\item \textbf{Spherical interface:}
			The induced sub-hypergraph $U[S]$ is a disjoint union of exactly $k$ $(d-2)$-spheres.
		\end{itemize}
	\end{definition}
	
	While classical notions of connectivity in topology (e.g., homotopical or homological connectivity) capture global topological properties based on the vanishing of homotopy or homology groups, they do not directly generalize the graph-theoretic concept of multiple disjoint connections between two faces. In contrast, the notion of $(d-1)$-sphere connectivity introduced above quantifies the independent geometric pathways between two faces using collections of $(d-1)$-spheres satisfying sphere independence conditions. \emph{To the best of our knowledge, this particular sphere-based connectivity invariant has not been previously formalized in the topology literature.}

	\subsection{High-dimensional $S$-component}
	
	Before introducing the higher-dimensional analogue, we briefly recall the notion of
	$S$-components and $S$-decompositions in graph theory.
	Given a vertex cut $S$ of a graph $G$, the idea is to decompose $G$ along $S$ into subgraphs,
	each consisting of $S$ together with one connected component of $G-S$.
	These subgraphs, called $S$-components, capture how $G$ is locally attached to the separating
	set $S$.
	
	When $G$ is $2$-connected and $S=\{x,y\}$ is a $2$-vertex cut, each $S$-component is augmented
	by adding an artificial edge between $x$ and $y$.
	This \emph{marker edge} records the role of the separating set and ensures that each component
	remains $2$-connected.
	The resulting marked $S$-components form the marked $S$-decomposition of $G$, from which the
	original graph can be recovered by deleting all marker edges.
	This decomposition plays a central role in structural results for planar graphs.

	\begin{definition}[$S$-component of a topological hypergraph]
		Let $U$ be a $1$-sphere connected $(d-1)$-dimensional topological hypergraph with $d \ge 3$.
		Let $S \subseteq V(U)$ be a set of $p$ vertices such that: 
		\begin{itemize}
			\item $S$ is a vertex cut of $U$, that is, $U - S$ is disconnected;
			\item the induced sub-hypergraph $U[S]$ is homeomorphic to a $(d-2)$-sphere or $(d-1)$-ball.
		\end{itemize}
		Let $X$ be the vertex set of a connected component of $U - S$.
		The sub-hypergraph $H$ of $U$ induced by $S \cup X$, namely
		\[
		H := U[S \cup X],
		\]
		is called an \emph{$S$-component} of $U$.
	\end{definition}
	
	\begin{definition}[marked $S$-component]
		Let $H$ be an $S$-component of a $(d-1)$-dimensional topological hypergraph $U$.
		If $U$ contains no $(d-1)$-topological hyperedge whose vertex set is exactly $S$, we form a new topological hypergraph $\widehat{H}$ by attaching a $(d-1)$-topological hyperedge $\sigma_S$ to $H$ such that: 
		\begin{itemize}
			\item the vertex set of $\sigma_S$ is $S$;
			\item the attaching map identifies $\partial \sigma_S$ with $U[S]$. %"attaching map"是CW 复形中的“附着映射”
		\end{itemize}
		The added $(d-1)$-topological hyperedge $\sigma_S$ is called a \emph{marker topological hyperedge}, and the resulting topological hypergraph
		$\widehat{H}$ is called a \emph{marked $S$-component} of $U$.
		If such a $(d-1)$-topological hyperedge already exists in $U$, we set $\widehat{H} := H$.
	\end{definition}
	
	\begin{definition}[marked $S$-decomposition]
		Let $U$ be a $(d-1)$-dimensional topological hypergraph and let $S \subseteq V(U)$ satisfy the conditions of an
		$S$-vertex cut. 
		The collection of all marked $S$-components $\widehat{H}$ arising from the connected components of
		$U - S$ is called the \emph{marked $S$-decomposition} of $U$. 
	\end{definition}
	
	\begin{remark}
		The original topological hypergraph $U$ can be recovered from its marked $S$-decomposition by taking the union
		of all marked $S$-components and subsequently deleting all marker cells. 
	\end{remark}

	\subsection{Dimensional-raising function}
	
	Before defining the dimension-raising function $U^x(G)$, we introduce the following notion. 
	
	For ease of exposition, we present an example in Figures~\ref{f1} - \ref{f3}, which provides a visual illustration of how the dimension of a graph $G$ can be raised via the dimensional-raising function and marked $S$-decomposition.

	\begin{definition}[$(i+1)$-sphere generating set]\label{sphere generating set}
		Let $U$ be an $i$-dimensional topological hypergraph, and let $V(S^{i+1}) \subseteq V(U)$.  
		We call $V(S^{i+1})$ an \emph{$(i+1)$-sphere generating set} of $U$ if it satisfies the following conditions.
		\begin{itemize}
			\item $V(S^{i+1})$ is minimal;
			
			\item The sub-hypergraph of $U$ induced by $V(S^{i+1})$ is an $\mathbb{R}^{i+1}$-hypergraph;
			
			\item After deleting the vertex set $V(S^{i+1})$ from $U$, the remaining part of $U$ is still connected;
			
			\item Let $\mathbb{S}^i(V(S^{i+1}))$ denote the collection of all induced $i$-spheres contained in $V(S^{i+1})$.  
			If each induced $i$-sphere in $\mathbb{S}^i(V(S^{i+1}))$ is filled with an $(i+1)$-ball, the resulting space is an $(i+1)$-sphere. (We refer to the above procedure of turning $U[V(S^{i+1})]$ into an $(i+1)$-sphere as a \emph{filling} of $U[V(S^{i+1})]$.)
		\end{itemize}
		
	\end{definition}

	\begin{definition}[dimension-raising function $U^{x}(G)$]\label{dimension-raising}
		Let $G$ be a graph, we assume that the graph $G$ has already undergone a marked $S$-decomposition with $|S|=2$, therefore, $G$ is a $3$-connected graph, and let $U^x(G)$ denote the \emph{dimension-raising function} of the graph $G$. This function transforms $G$ into a $x$-dimensional manifold as follows (note that $U^1(G)$ is equivalent to $G$): 
		
		Before performing Step~$i$, one must apply a marked $S$-decomposition to $U^i(G)$ so that every marked $S$-component is $2$-sphere connected. If some marked $S$-component is only $1$-sphere connected, then the procedure terminates for that component; in this case, the dimension of the marked $S$-component can be raised to at most $i$. 
		
		\begin{itemize}
			\item Choose any $2$-sphere generating set $V(S^2)$ in $U^1(G)$ and fill it to obtain a $2$-sphere $S^2$. Denote the resulting graph by $U^2(G)$. If the complex $U^2(G)$ contains a $1$-dimensional $1$-sphere cut $S$, we then perform a marked $S$-decomposition on $U^2(G)$. This procedure is repeated until $U^2(G)$ contains no $2$-sphere generating set. 
			
			\item Choose any $3$-sphere generating set $V(S^3)$ in $U^2(G)$ and fill it to obtain a $3$-sphere $S^3$. Denote the resulting graph by $U^3(G)$. If the complex $U^3(G)$ contains a $2$-dimensional $1$-sphere cut $S$, we then perform a marked $S$-decomposition on $U^3(G)$. This procedure is repeated until $U^3(G)$ contains no $3$-sphere generating set.
			
			\item $\dots$
			
			\item 
			\noindent
			\textbf{Step 1 (Filling).}
			Given the graph $U^{i}(G)$, choose any $(i+1)$-sphere generating set $V(S^{i+1}) \subseteq V(U^{i}(G))$. We fill $V(S^{i+1})$ to obtain an $(i+1)$-sphere $S^{i+1}$, and denote the resulting graph by $U^{i+1}(G)$.
			
			\noindent
			\textbf{Step 2 (Decomposition).}
			If the graph $U^{i+1}(G)$ contains an $i$-dimensional $1$-sphere cut $S$, then we perform a marked $S$-decomposition on $U^{i+1}(G)$.
			
			\noindent
			\textbf{Step 3 (Iteration).}
			Repeat Steps~1 and~2 on the updated graph until $U^{i+1}(G)$ contains no $(i+1)$-sphere generating set.
			
			\medskip
			\noindent
			This recursive process is applied successively for $i = 1, 2, \dots, x-1$, and terminates once the transformation from $U^{x-1}(G)$ to $U^{x}(G)$ is completed.

			\item $\dots$
			\item Choose any $x$-sphere generating set $V(S^x)$ in $U^{x-1}(G)$ and fill it to obtain a $x$-sphere $S^x$. Denote the resulting graph by $U^x(G)$. If the complex $U^x(G)$ contains a $(x-1)$-dimensional $1$-sphere cut $S$, we then perform a marked $S$-decomposition on $U^x(G)$. This procedure is repeated until $U^x(G)$ contains no $x$-sphere generating set.
		\end{itemize}
		
		Note that the above process of increasing the dimension cannot proceed indefinitely. The process terminates when there exists a unique induced $x$-sphere in $U^x(G)$. 
		The $i$-ball obtained in the above process is referred to as the $i$-dimensional topological hyperedge of $U^x(G)$. It follows that this hyperedge is homeomorphic to an $i$-dimensional convex polytope. 
		We refer to the above process as the geometrization of the graph $G$, and call $U^x(G)$ the $x$-dimensional topological hypergraph of $G$. 
		
	\end{definition}
	
	We do not claim uniqueness: Different choices of generating sets may lead to different decompositions.  
	Nevertheless, under any such decomposition, each marked $S$-component is either a $2$-sphere connected $x$-dimensional topological hypergraph or is homeomorphic to an $i$-sphere with $2 \le i \le x$.

	We provide a definition of the $d$-uniform $(d-1)$-dimensional topological hypergraph of $G$. 
	
	\begin{definition}[$d$-uniform $(d-1)$-dimensional topological hypergraph of $G$]\label{triangulated hypergraph}
		Let $G$ be a graph, and let $U^{d-1}(G)$ denote the $(d-1)$-dimensional topological hypergraph of $G$. If each $(d-1)$-dimensional topological hyperedge of $U^{d-1}(G)$ admits a triangulation such that,  after the triangulation, every $(d-1)$-dimensional topological hyperedge is homeomorphic to an $(d-1)$-simplex, then the resulting hypergraph is called the $d$-uniform $(d-1)$-dimensional topological hypergraph of $G$, denoted by $U^{d-1}_{\Delta}(G)$. 
		
		Note that the process of triangulating $U^i(G)$ to obtain a new manifold $U^i_{\Delta}(G)$ does not alter its embeddability in $(i+1)$-dimensional Euclidean space. 
	\end{definition}

	\begin{figure}[htbp]
		\centering
		
		\begin{tikzpicture}[
			scale=1.3, 
			>=Stealth,
			% 样式定义
			vertex/.style={circle, draw, fill=white, inner sep=0pt, minimum size=14pt, font=\small},
			marked_edge/.style={red, very thick}, % 标记边样式
			normal_edge/.style={thick}
			]
			
			% =========================================
			% 1. 左侧：原图 (Original Graph)
			% =========================================
			\begin{scope}[shift={(0,0)}]
				\node[font=\bfseries] at (0, -2.5) {Original graph $G$.};
				
				% Vertices
				\node[vertex] (v1) at (90:1)  {$v_1$};
				\node[vertex] (v2) at (30:1)  {$v_2$};
				\node[vertex] (v3) at (-30:1) {$v_3$};
				\node[vertex] (v4) at (-90:1) {$v_4$};
				\node[vertex] (v5) at (-150:1){$v_5$};
				\node[vertex] (v6) at (150:1) {$v_6$};
				\node[vertex] (v7) at (120:1) {$v_7$}; % subdivision point
				\node[vertex] (v8) at (55:1.9)  {$v_8$};
				\node[vertex] (v9) at (30:2)  {$v_9$};
				\node[vertex] (v0) at (5:1.9)  {$v_{0}$};
				
				% Edges
				\draw[normal_edge] (v1)--(v2)--(v3)--(v4)--(v5)--(v6);
				\draw[normal_edge] (v1)--(v3) (v1)--(v4) (v1)--(v5);
				\draw[normal_edge] (v2)--(v4) (v2)--(v5) (v2)--(v6);
				\draw[normal_edge] (v3)--(v5) (v3)--(v6);
				\draw[normal_edge] (v4)--(v6);
				\draw[normal_edge] (v1)--(v7)--(v6); % path v1-v7-v6
				
				\draw[normal_edge] (v8)--(v9)--(v0)--(v8);
				\draw[normal_edge] (v8)--(v1);
				\draw[normal_edge] (v9)--(v2);
				\draw[normal_edge] (v0)--(v3);
				\draw[normal_edge] (v8)--(v2)--(v0);
			\end{scope}
			
			% =========================================
			% 2. 中间：分解箭头
			% =========================================
			\draw[->, very thick, gray] (2.2, 0) -- (4, 0) 
			node[midway, above, text=black, align=center, font=\footnotesize] {marked $\{v_1, v_6\}$-\\decomposition};
			
			% =========================================
			% 3. 右侧：分解后的两个图
			% =========================================
			
			% --- 右上部分：3-cycle (Component 1) ---
			\begin{scope}[shift={(6, 1)}]
				\node[font=\bfseries] at (-0.5, 1.5) {Marked $\{v_1, v_6\}$-component $G_1$.};
				
				% 为了保持视觉一致性，保留原极坐标位置，但局部平移以便居中
				\node[vertex] (c1_v1) at (90:1)  {$v_1$}; % 对应v1
				\node[vertex] (c1_v6) at (150:1) {$v_2$}; % 对应v6
				\node[vertex] (c1_v7) at (120:1) {$v_7$};
				
				% Edges
				\draw[normal_edge] (c1_v1)--(c1_v7)--(c1_v6);
				% 补上的标记边 (Marked Edge)
				\draw[marked_edge] (c1_v1)--(c1_v6) node[midway, below left, red, font=\tiny] {};
			\end{scope}
			
			% --- 右下部分：删除v7并补边的图 (Component 2) ---
			\begin{scope}[shift={(6, -1.5)}]
				\node[font=\bfseries] at (0, -1.5) {Marked $\{v_1, v_6\}$-component $G_2$.};
				
				% Vertices (excluding v7)
				\node[vertex] (c2_v1) at (90:1)  {$v_1$};
				\node[vertex] (c2_v2) at (30:1)  {$v_2$};
				\node[vertex] (c2_v3) at (-30:1) {$v_3$};
				\node[vertex] (c2_v4) at (-90:1) {$v_4$};
				\node[vertex] (c2_v5) at (-150:1){$v_5$};
				\node[vertex] (c2_v6) at (150:1) {$v_6$};
				\node[vertex] (c2_v8) at (55:1.9)  {$v_8$};
				\node[vertex] (c2_v9) at (30:2)  {$v_9$};
				\node[vertex] (c2_v0) at (5:1.9)  {$v_{0}$};
				
				% Edges
				\draw[normal_edge] (c2_v1)--(c2_v2)--(c2_v3)--(c2_v4)--(c2_v5)--(c2_v6);
				\draw[normal_edge] (c2_v1)--(c2_v3) (c2_v1)--(c2_v4) (c2_v1)--(c2_v5);
				\draw[normal_edge] (c2_v2)--(c2_v4) (c2_v2)--(c2_v5) (c2_v2)--(c2_v6);
				\draw[normal_edge] (c2_v3)--(c2_v5) (c2_v3)--(c2_v6);
				\draw[normal_edge] (c2_v4)--(c2_v6);
				
				\draw[normal_edge] (c2_v8)--(c2_v9)--(c2_v0)--(c2_v8);
				\draw[normal_edge] (c2_v8)--(c2_v1);
				\draw[normal_edge] (c2_v9)--(c2_v2);
				\draw[normal_edge] (c2_v0)--(c2_v3);
				\draw[normal_edge] (c2_v8)--(c2_v2)--(c2_v0);
				
				% 关键变化：补上 v1-v6 边 (Marked Edge)
				\draw[marked_edge] (c2_v1)--(c2_v6);
			\end{scope}
			
		\end{tikzpicture}
		
		\caption{The marked $\{v_1, v_6\}$-decomposition of $G$ is such that each marked $\{v_1, v_6\}$-component of $G$ is either a 3-connected graph or a graph that can be embedded in $\mathbb{S}^1$.}
		\label{f1}
	\end{figure}

	\begin{figure}[htbp]
		\centering
		\begin{tikzpicture}[
			scale=1.3, 
			>=Stealth,
			% 样式定义
			vertex/.style={circle, draw, fill=white, inner sep=0pt, minimum size=14pt, font=\small},
			marked_edge/.style={red, very thick},
			normal_edge/.style={thick},
			% 新增面填充样式：fill opacity 控制透明度 (0-1)
			face/.style={fill=orange!40, fill opacity=0.6, draw=none},
			face_c1/.style={fill=cyan!40, fill opacity=0.6, draw=none},
			face_c2/.style={fill=magenta!40, fill opacity=0.6, draw=none},
			face_c3/.style={fill=yellow!40, fill opacity=0.6, draw=none}
			]
			
			% =========================================
			% 1. 左侧：原图 (Original Graph)
			% =========================================
			\begin{scope}[shift={(0,0)}]
				\node[font=\bfseries] at (0, -2.5) {A $2$-sphere generating set in $G_2$.};
				
				% Vertices
				\node[vertex] (v1) at (90:1)  {$v_1$};
				\node[vertex] (v2) at (30:1)  {$v_2$};
				\node[vertex] (v3) at (-30:1) {$v_3$};
				\node[vertex] (v4) at (-90:1) {$v_4$};
				\node[vertex] (v5) at (-150:1){$v_5$};
				\node[vertex] (v6) at (150:1) {$v_6$};
				\node[vertex] (v8) at (55:1.9)  {$v_8$};
				\node[vertex] (v9) at (30:2)  {$v_9$};
				\node[vertex] (v0) at (5:1.9)  {$v_{0}$};
				
				% Edges
				\draw[normal_edge] (v1)--(v2)--(v3)--(v4)--(v5)--(v6);
				\draw[normal_edge] (v1)--(v3) (v1)--(v4) (v1)--(v5);
				\draw[normal_edge] (v2)--(v4) (v2)--(v5) (v2)--(v6);
				\draw[normal_edge] (v3)--(v5) (v3)--(v6);
				\draw[normal_edge] (v4)--(v6);
				\draw[normal_edge] (v8)--(v9)--(v0)--(v8);
				\draw[normal_edge] (v8)--(v1);
				\draw[normal_edge] (v9)--(v2);
				\draw[normal_edge] (v0)--(v3);
				\draw[normal_edge] (v8)--(v2)--(v0);
				\draw[marked_edge] (v1)--(v6);
				
				% --- 新增：填充面 (使用 background layer 放在底层) ---
				\begin{scope}[on background layer]
					% v1, v2, v4 | v1, v3, v4 | v2, v3, v4
					\fill[face] (v1.center)--(v2.center)--(v4.center)--cycle;
					\fill[face] (v1.center)--(v3.center)--(v4.center)--cycle;
					\fill[face] (v2.center)--(v3.center)--(v4.center)--cycle;
					
					% v1, v3, v0, v8
					\fill[face] (v1.center)--(v3.center)--(v0.center)--(v8.center)--cycle;
					
					% v2, v3, v0 | v2, v8, v1
					\fill[face] (v2.center)--(v3.center)--(v0.center)--cycle;
					\fill[face] (v2.center)--(v8.center)--(v1.center)--cycle;
					
					% v9, v0, v2 | v9, v0, v8 | v9, v2, v8
					\fill[face] (v9.center)--(v0.center)--(v2.center)--cycle;
					\fill[face] (v9.center)--(v0.center)--(v8.center)--cycle;
					\fill[face] (v9.center)--(v2.center)--(v8.center)--cycle;
				\end{scope}
			\end{scope}
			
			% =========================================
			% 2. 中间：分解箭头
			% =========================================
			\draw[->, very thick, gray] (2.2, 0) -- (4, 0) 
			node[midway, above, text=black, align=center, font=\footnotesize] {marked $v_1v_2v_3$-\\decomposition};
			
			% =========================================
			% 3. 右侧：分解后的两个图
			% =========================================
			
			% --- 右上部分：Component 1 & 3 ---
			\begin{scope}[shift={(6, 1)}]
				\node[font=\bfseries] at (-0.5, 2.6) {Marked $v_1v_2v_3$-components $G_3$.};
				\node[font=\bfseries] at (-0.5, 2.1) {$G_3$ also admits a marked $v_0v_2v_8$-decomposition.};
				% Vertices
				\node[vertex] (c1_v1) at (90:1)  {$v_1$};
				\node[vertex] (c1_v2) at (30:1)  {$v_2$};
				\node[vertex] (c1_v3) at (-30:1) {$v_3$};
				\node[vertex] (c1_v8) at (55:1.9)  {$v_8$};
				\node[vertex] (c1_v9) at (30:3)  {$v_9$}; % 注意：这是该区域唯一的v9
				\node[vertex] (c1_v0) at (5:1.9)  {$v_{0}$};
				
				\node[vertex] (c3_v8) at (48:2.5)  {$v_8$};
				\node[vertex] (c3_v2) at (30:1.9)  {$v_2$};
				\node[vertex] (c3_v0) at (12:2.5)  {$v_{0}$};
				
				% Edges
				\draw[normal_edge] (c1_v1)--(c1_v2)--(c1_v3);
				\draw[normal_edge] (c1_v1)--(c1_v3);
				\draw[normal_edge] (c3_v8)--(c1_v9)--(c3_v0)--(c3_v8);
				\draw[normal_edge] (c1_v8)--(c1_v1);
				\draw[normal_edge] (c1_v9)--(c3_v2);
				\draw[normal_edge] (c1_v0)--(c1_v3);
				\draw[normal_edge] (c1_v8)--(c1_v2)--(c1_v0)--(c1_v8);
				\draw[normal_edge] (c3_v8)--(c3_v2)--(c3_v0);
				
				% Marked Edges
				\draw[marked_edge] (c1_v1)--(c1_v2)--(c1_v3)--(c1_v1);
				\draw[marked_edge] (c1_v8)--(c1_v2)--(c1_v0)--(c1_v8);
				\draw[marked_edge] (c3_v8)--(c3_v2)--(c3_v0)--(c3_v8);
				
				% --- 新增：填充面 ---
				\begin{scope}[on background layer]
					% Group c1:
					% c1_v1, c1_v3, c1_v0, c1_v8
					\fill[face_c1] (c1_v1.center)--(c1_v3.center)--(c1_v0.center)--(c1_v8.center)--cycle;
					% c1_v2, c1_v1, c1_v3
					\fill[face_c1] (c1_v2.center)--(c1_v1.center)--(c1_v3.center)--cycle;
					% c1_v2, c1_v3, c1_v0
					\fill[face_c1] (c1_v2.center)--(c1_v3.center)--(c1_v0.center)--cycle;
					% c1_v2, c1_v0, c1_v8
					\fill[face_c1] (c1_v2.center)--(c1_v0.center)--(c1_v8.center)--cycle;
					% c1_v2, c1_v8, c1_v1
					\fill[face_c1] (c1_v2.center)--(c1_v8.center)--(c1_v1.center)--cycle;
					
					% Group c3: 
					% (注意: 用 c1_v9 替代代码中未定义的 c3_v9)
					% c3_v0, c3_v2, c3_v8
					\fill[face_c3] (c3_v0.center)--(c3_v2.center)--(c3_v8.center)--cycle;
					% c3_v0, c3_v2, c3_v9 -> c1_v9
					\fill[face_c3] (c3_v0.center)--(c3_v2.center)--(c1_v9.center)--cycle;
					% c3_v0, c3_v8, c3_v9 -> c1_v9
					\fill[face_c3] (c3_v0.center)--(c3_v8.center)--(c1_v9.center)--cycle;
					% c3_v2, c3_v8, c3_v9 -> c1_v9
					\fill[face_c3] (c3_v2.center)--(c3_v8.center)--(c1_v9.center)--cycle;
				\end{scope}
			\end{scope}
			
			% --- 右下部分：Component 4 ---
			\begin{scope}[shift={(6, -1.5)}]
				\node[font=\bfseries] at (0, -1.5) {Marked $v_1v_2v_3$-component $G_4$.};
				
				% Vertices
				\node[vertex] (c2_v1) at (90:1)  {$v_1$};
				\node[vertex] (c2_v2) at (30:1)  {$v_2$};
				\node[vertex] (c2_v3) at (-30:1) {$v_3$};
				\node[vertex] (c2_v4) at (-90:1) {$v_4$};
				\node[vertex] (c2_v5) at (-150:1){$v_5$};
				\node[vertex] (c2_v6) at (150:1) {$v_6$};
				
				% Edges
				\draw[normal_edge] (c2_v1)--(c2_v2)--(c2_v3)--(c2_v4)--(c2_v5)--(c2_v6);
				\draw[normal_edge] (c2_v1)--(c2_v3) (c2_v1)--(c2_v4) (c2_v1)--(c2_v5);
				\draw[normal_edge] (c2_v2)--(c2_v4) (c2_v2)--(c2_v5) (c2_v2)--(c2_v6);
				\draw[normal_edge] (c2_v3)--(c2_v5) (c2_v3)--(c2_v6);
				\draw[normal_edge] (c2_v4)--(c2_v6);
				
				% Marked Edges
				\draw[marked_edge] (c2_v1)--(c2_v6);
				\draw[marked_edge] (c2_v1)--(c2_v2)--(c2_v3)--(c2_v1);
				
				% --- 新增：填充面 ---
				\begin{scope}[on background layer]
					% c2_v1, c2_v2, c2_v4
					\fill[face_c2] (c2_v1.center)--(c2_v2.center)--(c2_v4.center)--cycle;
					% c2_v1, c2_v3, c2_v4
					\fill[face_c2] (c2_v1.center)--(c2_v3.center)--(c2_v4.center)--cycle;
					% c2_v2, c2_v3, c2_v4
					\fill[face_c2] (c2_v2.center)--(c2_v3.center)--(c2_v4.center)--cycle;
					% c2_v1, c2_v2, c2_v3
					\fill[face_c2] (c2_v1.center)--(c2_v2.center)--(c2_v3.center)--cycle;
				\end{scope}
			\end{scope}
			
		\end{tikzpicture}
		\caption{We choose a $2$-sphere generating set in $G_2$ (Definition~\ref{sphere generating set}) and fill each induced sphere with a $2$-ball ($v_4v_1v_2$, $v_4v_1v_3$, $v_4v_2v_3$, $v_1v_3v_0v_8$, $v_2v_0v_3$, $v_2v_1v_8$, $v_9v_0v_2$, $v_9v_0v_8$, $v_9v_2v_8$). 
			For each $1$-dimensional $1$-sphere cut (Definition~\ref{sphere cut}), namely $v_1v_2v_3$ and $v_0v_2v_8$ in this example, 
			we apply the corresponding marked decomposition.
		}
		\label{f2}
	\end{figure}

	\begin{figure}[htbp]
		\centering
		
		%----------------------------------------------------
		% 左图：宽度调整为 0.42\textwidth
		%----------------------------------------------------
		\begin{subfigure}[b]{0.42\textwidth}
			\centering
			\begin{tikzpicture}[scale=1.5, every node/.style={circle, draw, minimum size=18pt, inner sep=1pt, fill=white}] % scale稍微放大以便观察，节点填充白色以防透色
				
				% --- 顶点定义 ---
				\node (v1) at (90:1)  {$v_1$};
				\node (v2) at (30:1)  {$v_2$};
				\node (v3) at (-30:1) {$v_3$};
				\node (v4) at (-90:1) {$v_4$};
				\node (v5) at (-150:1){$v_5$};
				\node (v6) at (150:1) {$v_6$};
				
				% --- 填充 3-Face (放在背景层，以免遮挡文字) ---
				\begin{scope}[on background layer]
					% 为了代码整洁，定义一个样式：半透明
					\tikzstyle{face}=[opacity=0.3, draw=none]
					
					% 1. 包含 v1 的三角形 (10个)
					\fill[face, red]           (v1.center)--(v2.center)--(v3.center)--cycle;
					\fill[face, orange]        (v1.center)--(v2.center)--(v4.center)--cycle;
					\fill[face, yellow]        (v1.center)--(v2.center)--(v5.center)--cycle;
					\fill[face, green]         (v1.center)--(v2.center)--(v6.center)--cycle;
					\fill[face, cyan]          (v1.center)--(v3.center)--(v4.center)--cycle;
					\fill[face, blue]          (v1.center)--(v3.center)--(v5.center)--cycle;
					\fill[face, violet]        (v1.center)--(v3.center)--(v6.center)--cycle;
					\fill[face, magenta]       (v1.center)--(v4.center)--(v5.center)--cycle;
					\fill[face, brown]         (v1.center)--(v4.center)--(v6.center)--cycle;
					\fill[face, lime]          (v1.center)--(v5.center)--(v6.center)--cycle;
					
					% 2. 包含 v2 但不含 v1 的三角形 (6个)
					\fill[face, pink]          (v2.center)--(v3.center)--(v4.center)--cycle;
					\fill[face, teal]          (v2.center)--(v3.center)--(v5.center)--cycle;
					\fill[face, olive]         (v2.center)--(v3.center)--(v6.center)--cycle;
					\fill[face, purple]        (v2.center)--(v4.center)--(v5.center)--cycle;
					\fill[face, lightgray]     (v2.center)--(v4.center)--(v6.center)--cycle;
					\fill[face, darkgray]      (v2.center)--(v5.center)--(v6.center)--cycle;
					
					% 3. 剩余的三角形 (4个)
					\fill[face, red!50!blue]   (v3.center)--(v4.center)--(v5.center)--cycle;
					\fill[face, green!50!blue] (v3.center)--(v4.center)--(v6.center)--cycle;
					\fill[face, orange!50!red] (v3.center)--(v5.center)--(v6.center)--cycle;
					\fill[face, blue!50!cyan]  (v4.center)--(v5.center)--(v6.center)--cycle;
				\end{scope}
				
				% --- 连边 (保持原样) ---
				\draw (v1)--(v2)--(v3)--(v4)--(v5)--(v6)--(v1);
				\draw (v1)--(v3) (v1)--(v4) (v1)--(v5);
				\draw (v2)--(v4) (v2)--(v5) (v2)--(v6);
				\draw (v3)--(v5) (v3)--(v6);
				\draw (v4)--(v6);
				
			\end{tikzpicture}
			\caption{$U^2(G_4)$ is isomorphic to the $2$-skeleton of the $5$-simplex.}
			\label{fig:left-graph}
		\end{subfigure}
		\hfill % 弹性间距
		%
		%----------------------------------------------------
		% 中间：箭头 (使用 raisebox 调整高度)
		%----------------------------------------------------
		\raisebox{2cm}{ % <--- 如果箭头偏上或偏下，修改这里的 2cm
			\begin{tikzpicture}
				% 画一个灰色粗箭头，长度为 1cm
				\draw[->, thin, gray] (0,0) -- (0.5,0);
				% 如果需要在箭头上加字，可以取消下面这行的注释
				% \node[midway, above, font=\scriptsize] {embedding};
			\end{tikzpicture}
		}
		\hfill % 弹性间距
		%
		%----------------------------------------------------
		% 右图：宽度调整为 0.50\textwidth
		%----------------------------------------------------
		\begin{subfigure}[b]{0.50\textwidth}
			\centering
			\tdplotsetmaincoords{70}{120}
			\begin{tikzpicture}[tdplot_main_coords, scale=2.8]
				
				% 坐标轴
				\draw[->, gray] (0,0,0) -- (1.2,0,0) node[below left] {$x$};
				\draw[->, gray] (0,0,0) -- (0,1.2,0) node[below right] {$y$};
				\draw[->, gray] (0,0,0) -- (0,0,1.2) node[above] {$z$};
				
				% 顶点定义
				\coordinate (v5) at (0,0,0);
				\coordinate (v1) at (1,0,0);
				\coordinate (v2) at (0,1,0);
				\coordinate (v3) at (0,0,1);
				\coordinate (v4) at (0.3,0.3,0.3);
				\coordinate (v6) at (0,1,0.5);
				
				% 10 个面
				\filldraw[fill=red!60, opacity=0.45, draw=black] (v5)--(v1)--(v2)--cycle;
				\filldraw[fill=blue!60, opacity=0.45, draw=black] (v5)--(v1)--(v3)--cycle;
				\filldraw[fill=green!60, opacity=0.45, draw=black] (v5)--(v1)--(v4)--cycle;
				\filldraw[fill=orange!70, opacity=0.45, draw=black] (v5)--(v2)--(v3)--cycle;
				\filldraw[fill=purple!60, opacity=0.45, draw=black] (v5)--(v2)--(v4)--cycle;
				\filldraw[fill=cyan!60, opacity=0.45, draw=black] (v5)--(v3)--(v4)--cycle;
				\filldraw[fill=yellow!70!brown, opacity=0.45, draw=black] (v1)--(v2)--(v3)--cycle;
				\filldraw[fill=teal!60, opacity=0.45, draw=black] (v1)--(v2)--(v4)--cycle;
				\filldraw[fill=magenta!60, opacity=0.45, draw=black] (v1)--(v3)--(v4)--cycle;
				\filldraw[fill=lime!60, opacity=0.45, draw=black] (v2)--(v3)--(v4)--cycle;
				
				% v6 连线
				\foreach \target in {v5, v1, v2, v3, v4} {
					\draw[thick, red, dashed] (v6) -- (\target);
				}
				
				% 顶点标注 (带 $)
				\foreach \v/\name/\x/\y/\z in {
					v5/$v_5$/0/0/0,
					v1/$v_1$/1/0/0,
					v2/$v_2$/0/1/0,
					v3/$v_3$/0/0/1,
					v4/$v_4$/\frac{1}{3}/\frac{1}{3}/\frac{1}{3},
					v6/$v_6$/0/1/0.5
				}
				{
					\filldraw[black] (\v) circle (0.6pt)
					node[above right] {\name\ {\scriptsize $(\x,\y,\z)$}};
				}
			\end{tikzpicture}
			\caption{$U^2(G_4)$ is not embeddable in $\mathbb{R}^3$.}
			\label{fig:right-embedding}
		\end{subfigure}
		
		\caption{Applying the procedure in Figure~\ref{f2} to the graph $G_4$ yields $U^2(G_4)$, which is isomorphic to the $2$-skeleton of the $5$-simplex, and whose $2$-sphere connectivity (Definiton \ref{sphere connectivity}) equals $2$. Since $U^2(G_4)$ is not embeddable in $\mathbb{R}^3$ but is embeddable in $\mathbb{R}^4$ (Lemma \ref{lemma:flo}), Theorem~\ref{chromatic-number} implies that the chromatic number of $G$ is at most $24$.}
		\label{f3}
	\end{figure}
	
	%\FloatBarrier

	It remains to verify that the complex constructed in Definition~\ref{dimension-raising} indeed forms a topological hypergraph in the sense of Definition~\ref{new-topological hypergraph}. Clearly, it suffices to show that the $i$-th homotopy group of $U^{d-1}(G)$ is trivial for every $i \in \{1,2,\dots,d-2\}$ (Lemma \ref{c1}).

	Similar to the edges in planar graphs, the complexes in the above definition are allowed to undergo continuous topological transformations such as contraction and stretching, provided that no internal intersections occur.

	%We now prove that $U^{d-1}(G)$ is a $(d-1)$-dimensional topological hypergraph.

	\begin{lemma}\label{c1}
		The $i$-th homotopy group of $U^{d-1}(G)$ is trivial for all $i\in \{1,2,3,...,d-2\}$. 
	\end{lemma}
	
	\begin{proof}

		We prove the lemma by contradiction. Suppose there exists $i \in \{1, 2, 3, \dots, d-2\}$ such that the $i$-th homotopy group of $U^{d-1}(G)$ is nontrivial. Then there must exist a non-contractible $i$-sphere $S_1$ in $U^{d-1}(G)$ that satisfies the following conditions: 
		\begin{itemize}
			\item $S_1$ is composed of $i$-dimensional topological hyperedges;
			\item $S_1$ cannot be continuously contracted to a point;
			\item $U^{d-1}(G) \backslash S_1$ is disconnected;
			\item $S_1$ is minimal.
		\end{itemize}
		
		As shown in Figure \ref{huan}, $S_1$ must lie on the surface of a multiply connected region $M_1$, otherwise $S_1$ would not be minimal. Note that $U^{d-1}(G) \backslash S_1$ must be disconnected; otherwise, $S_1$ would be an induced sphere, which contradicts the assumption that $S_1$ is a non-contractible $i$-sphere.
		Let $M_2$ be a connected component of $U^{d-1}(G) \backslash S_1$ that is distinct from $M_1$. It follows that $M_2$ must also be a multiply connected region, otherwise $S_1$ could be contracted to a point through $M_2$.
		
		In $M_2$, take an $i$-sphere $S_2$ with the same properties as $S_1$. Let $M_3$ be a connected component of $U^{d-1}(G) \backslash S_2$ that is distinct from $M_2$. Then $M_3$ cannot contain any points from $M_1$, otherwise it would imply that $U^{d-1}(G) \backslash S_1$ is connected, leading to a contradiction since $S_1$ would be an induced $i$-sphere.
		
		In $M_3$, take an $i$-sphere $S_3$ with the same properties as $S_1$. Let $M_4$ be a connected component of $U^{d-1}(G) \backslash S_3$ that is distinct from $M_3$. Similarly, $M_4$ cannot contain any points from $M_1$ or $M_2$.
		
		Since this process can be repeated indefinitely, it follows that the graph $G$ must be an infinite graph, leading to a contradiction.
		
		Therefore, the assumption is false, and the lemma is proved.
		
		\begin{figure}
			\centering     
			\includegraphics[width=0.6\linewidth]{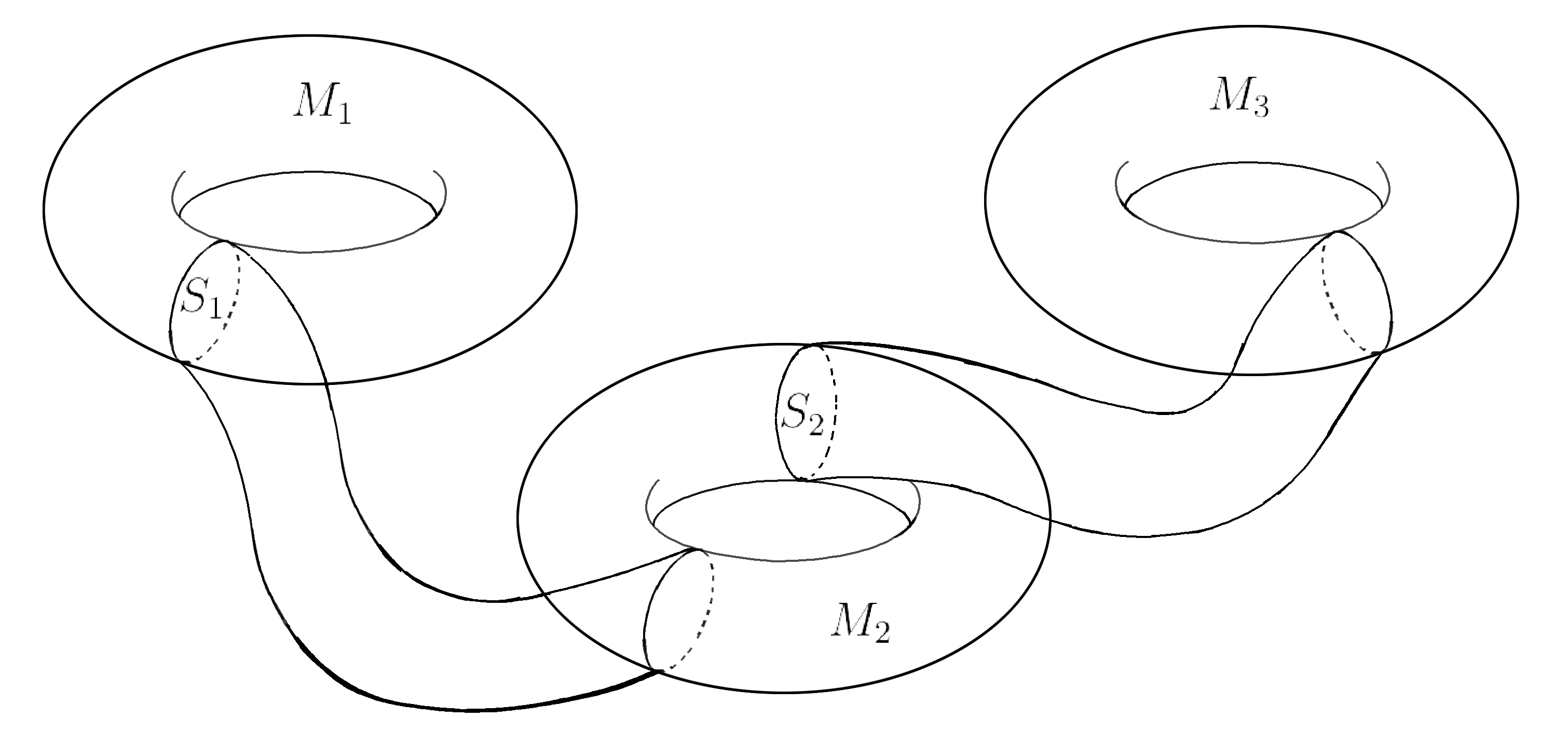}
			\caption{Infinite graph $G$.}
			\label{huan}
		\end{figure}
		
	\end{proof}
	
	\begin{corollary}
		$U^{d-1}(G)$ is a $(d-1)$-dimensional topological hypergraph. 
	\end{corollary}

	Next, our goal is to study the embeddability of $U^{d-1}(G)$ in $\mathbb{R}^{d}$.

	\section{$\mathbb{R}^d$-hypergraph: Definition and property}\label{section3}

	In this section, we generalize the notions of planar graphs, minors, incidence, adjacency, edge contraction, and related concepts to the setting of $\mathbb{R}^d$.

	\begin{definition}[{\it $\mathbb{R}^d$-hypergraph} and {\it non-$\mathbb{R}^d$-hypergraph}] \label{special0}
		If a {\it $(d-1)$-dimensional topological hypergraph} $K$ can be embedded in $\mathbb{R}^d$, then $K$ is called an {\it $\mathbb{R}^d$-hypergraph}; otherwise, $K$ is called a {\it non-$\mathbb{R}^d$-hypergraph}. 
	\end{definition}
	
	\begin{definition}[{\it $d$-uniform $\mathbb{R}^d$-hypergraph}] \label{special1}
		Let $K$ be an $\mathbb{R}^d$-hypergraph. If every $(d-1)$-dimensional topological hyperedge of $K$ is homeomorphic to a $(d-1)$-simplex, then $K$ is called a {\it $d$-uniform $\mathbb{R}^d$-hypergraph}.
	\end{definition}
	
	Note that a $d$-uniform $(d-1)$-dimensional topological hypergraph is not necessarily embeddable in $\mathbb{R}^d$, whereas a $d$-uniform $\mathbb{R}^d$-hypergraph can be embedded in $\mathbb{R}^d$. 
	Note that the $\mathbb{R}^d$-hypergraph can always be embedded in the $d$-sphere. 
	Since each $\mathbb{R}^d$-hypergraph is $(d-1)$-dimensional topological hypergraph, every definition concerning $(d-1)$-dimensional topological hypergraphs introduced in this paper also applies to $\mathbb{R}^d$-hypergraphs.

	%According to Definition \ref{topological-hypergraph} and Definition \ref{special0}, we know that an $\mathbb{R}^d$-hypergraph is a special type of $d$-uniform-topological hypergraph, a $d$-uniform-topological hypergraph may not necessarily be embeddable in $\mathbb{R}^d$. 
	%Note that the generalized Poincar\'{e} conjecture ensures that the $\mathbb{R}^d$-hypergraph can always be embedded in the $d$-sphere. 

	%Note that in the case where the $i$-th homotopy group is trivial for all $i\in \{1,2,3,...,d-2\}$, there will be no knot or other complex structures in the manifold. 
	%According to Definition \ref{special0}, it is obvious that the {\it $\mathbb{R}^d$-hypergraph} is a special case of the {\it general $\mathbb{R}^d$-hypergraph}. 
	
	%Since the $i$-th homotopy group of an $\mathbb{R}^d$-hypergraph is trivial, we will prove it can be regarded, up to isomorphism, as a union of internally disjoint $(d-1)$-spheres and $(d-1)$-balls in the following part. Note that {\it internally disjoint} means that their interiors do not overlap and they intersect only at their boundaries. (An {\it $i$-sphere} $S^i$ is a topological space that is homeomorphic to a standard $i$-sphere. The space enclosed by an $i$-sphere is called an {\it $(i+1)$-ball} or {\it $(i+1)$-disc} $D^{i+1}$.)

	In the process of triangulating a planar graph, it is necessary to add some edges. However, in higher dimensions, such operations become much less intuitive. It is easy to see that contracting simplices in an $\mathbb{R}^d$-hypergraph does not alter its homotopy groups. Nevertheless, when adding or removing simplices from an $\mathbb{R}^d$-hypergraph, one must impose certain constraints to ensure that the homotopy type is preserved. Therefore, we need to establish the following lemma.

	%\begin{lemma} [{\it Construction of $\mathbb{R}^d$-hypergraph}] \label{special}
	%	Every $\mathbb{R}^d$-hypergraph can be constructed by the following procedure. 
	%	
	%	{\bf Procedure \uppercase\expandafter{\romannumeral10}:} $T_1, T_2, ...,T_m$ are $(d-1)$-simplices in $\mathbb{R}^d$. Starting from $T_0$, add $T_1, T_2, ..., T_m$ in $\mathbb{R}^d$ one by one, and this procedure satisfies the following conditions: 
	%	
	%	\begin{itemize}
		%		\item 1. Let $T_i$ and $T_j$ be arbitrary simplex in $\{T_1, T_2, ...,T_m\}$ and $T' = T_i\cap T_j\neq \emptyset$, then $T'$ is a face of both $T_i$ and $T_j$.
		%		
		%		\item 2. Suppose the procedure is at step $i$ ($T_i$ has been added in $\mathbb{R}^d$). 
		%		Let $K_i= \bigcup \limits_{j = 0}^i T_j$, 
		%		then $T_{i+1}$ satisfies the following condition when adding $T_{i+1}$ to $\mathbb{R}^d$: $T_{i+1} \cap K_i \cong D^{d-1}$ or $T_{i+1} \cap K_i \cong S^{d-1}$. 
		%		
		%	\end{itemize}
	%	
	%\end{lemma}
	%
	%\begin{proof}
	%	
	%	Lemma \ref{attach}-\ref{attach2} directly implies Lemma \ref{special}.
	%	
	%\end{proof}

	\begin{lemma}\label{attach}
		Let $X$ be a $d$-dimensional CW complex satisfying $\pi_k(X) = 0$ for all $k \leq d-1$. Let Y be a $d$-dimensional simplex and $X \cap Y \cong S^{d-1}$, i.e., it is homeomorphic to the $(d-1)$-dimensional sphere. 
		Then: $\pi_k(X \cup Y) = 0$ for all $k \leq d-1.$
		
	\end{lemma}
	
	\begin{proof}
		Since $Y \cong D^d$ (i.e., the $d$-dimensional disk) and $X \cap Y \cong S^{d-1}$ is the boundary of $Y$, the pair $(X \cup Y, X)$ deformation retracts onto the pair $(Y, S^{d-1})$. Therefore,
		\[
		\pi_k(X \cup Y, X) \cong \pi_k(D^d, S^{d-1}).
		\]
		
		It is a standard fact that the relative homotopy groups satisfy
		\[
		\pi_k(D^d, S^{d-1}) = 0 \quad \text{for all } k \leq d - 1.
		\]
		
		Consider the long exact sequence \cite{switzer2017algebraic} of homotopy groups for the pair $(X \cup Y, X)$:
		\[
		\cdots \to \pi_k(X) \xrightarrow{i_*} \pi_k(X \cup Y) \xrightarrow{\partial} \pi_k(X \cup Y, X) \xrightarrow{\delta} \pi_{k-1}(X) \to \cdots
		\]
		Since $\pi_k(X \cup Y, X) = 0$ for $k \leq d - 1$, and $\pi_k(X) = \pi_{k-1}(X) = 0$ for $k \leq d - 1$ by assumption, the sequence reduces to:
		\[
		0 \to \pi_k(X \cup Y) \to 0,
		\]
		which implies:
		\[
		\pi_k(X \cup Y) = 0 \quad \text{for all } k \leq d - 1.
		\]
		
	\end{proof}
	
	Using a similar approach, we can also prove the following lemma.
	
	\begin{lemma}\label{attach2}
		Let $X$ be a $d$-dimensional CW complex satisfying $\pi_k(X) = 0$ for all $k \leq d-1$. Let Y be a $d$-dimensional simplex and $X \cap Y \cong D^{d-1}$, i.e., it is homeomorphic to the $(d-1)$-dimensional disk. 
		Then: $\pi_k(X \cup Y) = 0$ for all $k \leq d-1.$
	\end{lemma}

	In high-dimensional spaces, some common definitions can be generalized as follows.

	\begin{definition}[{\it multiple topological hyperedges}]\label{mul}
		Let $T_1$ and $T_2$ be $(d-1)$-dimensional topological hyperedges of a $(d-1)$-dimensional topological hypergraph $G$. If $V(T_1)= V(T_2)$, then $T_1$ and $T_2$ are called {\it multiple topological hyperedges}. 
	\end{definition}

	\begin{definition}[{\it $\mathbb{R}^d$-loops}]\label{loop}
		Let $T_1$ be a $(d-1)$-dimensional topological hyperedge of a $(d-1)$-dimensional topological hypergraph. $V(T_1)=\{u_1, u_2, ..., u_k\}$. If there exists $u_i, u_j \in V(T_1)$ ($i\neq j$) such that $u_i= u_j$, then $T_1$ is called an {\it $\mathbb{R}^d$-loop}. 
	\end{definition}
	
	Since higher-dimensional topological hyperedges have more than two vertices, the definition of the loops in higher-dimensional manifolds differs slightly from that in planar graphs. As long as two vertices of a topological hyperedge overlap, we consider it as an $\mathbb{R}^d$-loop.

	\begin{definition}[{\it simple $(d-1)$-dimensional topological hypergraph}]\label{simple} 
		If a $(d-1)$-dimensional topological hypergraph $G$ does not contain multiple topological hyperedges and $\mathbb{R}^d$-loop, then $G$ is called a {\it simple $(d-1)$-dimensional topological hypergraph}. 
	\end{definition}

	Unless otherwise specified, all $(d-1)$-dimensional topological hypergraphs mentioned hereafter will be simple. Analogous to the definition of incident and adjacent in graph theory, we can define the notions of incident and adjacent in $(d-1)$-dimensional topological hypergraph. 
	
	\begin{definition}[{\it $i$-face}]\label{def:Ai}
		Let $G$ be a $(d-1)$-dimensional topological hypergraph, and an $i$-face of $G$ is denoted by $a_i$, and the set of $i$-face of $G$ is denoted by 
		$$A_i(G)= \{a_i| a_i \ is \ \emph{an} \ i\emph{-face of} \ G \}.$$ 
		For convenience, we still use $V(G)$ to denote the vertex set of $G$, use $V(a_i)$ to denote the vertex set of $a_i$. 		
	\end{definition}

	\begin{definition}[{\it incident, adjacent and neighbour}] \label{def-incident}
		Let $u$ and $v$ be vertices, we say $u$ is adjacent to $v$ if there exists $a_{d-1} \in A_{d-1}(G)$ such that $u, v \in V(a_{d-1})$. 
		The set of all vertices that adjacent to $u$ is denoted by $N_G(u)$, the degree of $u$ is denoted by $d_G(u)= |N_G(u)|$. 
		
		Let $a_i$ be an $i$-face, and $a_j$ be a $j$-face ($i>0$ and $i\leq j$). We say $a_i$ is incident to $a_j$ (or $a_j$ is incident to $a_i$) if $i< j$ and $a_i\cap a_j= a_i$; we say $a_i$ is adjacent to $a_j$ if $i= j$ and $a_i\cap a_j$ is an $(i-1)$-face. 
		The set of all $j$-face incident (adjacent) to $a_i$ is denoted by $N_{G,j}(a_i)$. We say $d_{G,j}(a_i)= |N_{G,j}(a_i)|$ is the {\it $j$-dimensional degree} of $a_i$. When no ambiguity arises, $N_{G,j}(a_i)$ and $d_{G,j}(a_i)$ will be abbreviated as 
		$N_{j}(a_i)$ and $d_{j}(a_i)$, respectively.

	\end{definition}

	We now introduce the definition of \emph{minor} in the framework of $(d-1)$-dimensional topological hypergraphs,
	and start with defining the operations of \emph{deletion}, \emph{merging}, as well as \emph{contraction} of simplices.
	
	\vspace{0.5em}
	
	\begin{definition}[{\it topological hyperedge deletion}]\label{deletion}
		Given a $(d-1)$-dimensional topological hypergraph $G$, there are two natural ways of deriving smaller hypergraphs from $G$. 
		If $e$ is a $(d-1)$-dimensional topological hyperedge of $G$, we may obtain a hypergraph with $m-1$ $(d-1)$-dimensional topological hyperedges by deleting $e$ from $G$ but leaving the vertices and the remaining simplices intact. The resulting hypergraph is denoted by $G\backslash e$. Similarly, if $v$ is an $i$-simplex of $G$ with $i<d-1$, we may obtain a hypergraph by deleting from $G$ the topological hyperedge $v$ together with all the $(d-1)$-dimensional topological hyperedge incident with $v$. The resulting hypergraph is denoted by $G - v$ or $G\backslash v$. 
		
	\end{definition}

	Intuitively, a \emph{simplex contraction} can be understood as continuously shrinking a simplex into a single point.  
	We shall give a rigorous definition of this process by means of the concept of quotient space.
	
	%\begin{definition}[simplex contraction]
	%	Let $P$ be a $d$-uniform topological hypergraph (or a regular CW complex), and let $e \subset P$ be an $i$-simplex.  
	%	The \emph{contraction of $e$ in $P$} is the quotient space
	%	\[
	%	P/e := P/{\sim_e},
	%	\]
	%	where the equivalence relation $\sim_e$ is defined by
	%	\[
	%	x \sim_e y \iff 
	%	\begin{cases}
		%		x = y, & \text{or} \\
		%		x, y \in e.
		%	\end{cases}
	%	\]
	%	That is, all points of $e$ are identified to a single vertex, while all other points of $P$ remain distinct.  
	%	The resulting space $P/e$ is called the polytope obtained from $P$ by \emph{contracting the simplex $e$ to a vertex}. 
	%\end{definition}
	
	\begin{definition}[topological hyperedge contraction]
		Let $P$ be a $(d-1)$-dimensional topological hypergraph, and let $e \subset P$ be an $i$-dimensional topological hyperedge.  
		The \emph{contraction of $e$ in $P$} is the quotient space
		\[
		P/e := P/{\sim_e},
		\]
		where the equivalence relation $\sim_e$ is defined by
		\[
		x \sim_e y \iff x = y \text{ or } x, y \in e.
		\]
		%	\[
		%	x \sim_e y \iff x = y, & \text{or} 
		%	\begin{cases}
			%		x = y, & \text{or} \\
			%		x, y \in e.
			%	\end{cases}
		%	\]
		That is, all points of $e$ are identified to a single vertex, while all other points of $P$ remain distinct.  
		The resulting space $P/e$ is called the $(d-1)$-dimensional topological hypergraph obtained from $P$ by \emph{contracting the topological hyperedge $e$ to a vertex}. 
	\end{definition}

	%\begin{definition}[$(i,j)$-simplex contraction]
	%	Let $P$ be a $d$-uniform topological hypergraph (or a regular CW complex), 
	%	let $e \subset P$ be an $i$-simplex, and let $f \subset P$ be a $j$-simplex with $i>j$.  
	%	The \emph{contraction of $e$ onto $f$} is the quotient space
	%	\[
	%	P/(e \to f) := P/{\sim_{e,f}},
	%	\]
	%	where the equivalence relation $\sim_{e,f}$ is defined by
	%	\[
	%	x \sim_{e,f} y \iff
	%	\begin{cases}
		%		x = y, & \text{or} \\
		%		x, y \in e \text{ and } \pi_{ef}(x) = \pi_{ef}(y),
		%	\end{cases}
	%	\]
	%	with $\pi_{ef}: e \to f$ a continuous surjection that collapses $e$ onto $f$.
	%	In particular, all points of $e$ are identified along the fibers of $\pi_{ef}$, 
	%	and every point of $f$ remains fixed.  
	%	The resulting space $P/(e \to f)$ is called the $d$-uniform topological hypergraph obtained from $P$ by 
	%	\emph{contracting the $i$-simplex $e$ onto the $j$-simplex $f$}.
	%\end{definition}

	It should be noted that a topological hyperedge contraction may produce multiple topological hyperedges (see Definition~\ref{mul}).  
	To ensure that the resulting $(d-1)$-dimensional topological hypergraph remains simple, these multiple simplices must be removed.  
	In the case of graphs, multiple edges can simply be reduced by deleting the redundant ones.  
	However, in higher dimensions, two distinct cases must be considered:
	\begin{itemize}
		\item For multiple $(d-1)$-simplices, it suffices to delete one of them;
		\item For multiple $i$-simplices with $i < d-1$, according to Definition~\ref{deletion}, deleting one of them directly would also remove all $(d-1)$-simplices adjacent to it.  
		Therefore, we need to define the \emph{merging of multiple simplices}.
	\end{itemize}
	
	Intuitively, \emph{merging} means combining two multiple topological hyperedges into a single topological hyperedge.  
	A rigorous definition of this process will also be given using the notion of a quotient space.
	
	\begin{definition}[merging of two $i$-topological hyperedges]
		Let $P$ be a $(d-1)$-dimensional topological hypergraph. 
		Let $f_1,f_2 \subset P$ be two $i$-topological hyperedges which are homeomorphic, and suppose their sets of vertices coincide:
		\[
		V(f_1) = V(f_2).
		\]
		Assume there exists an $(i+1)$-face $F$ of $P$ such that 
		\[
		\partial F = f_1 \cup f_2,
		\]
		and a homeomorphism 
		\[
		\varphi: f_1 \xrightarrow{\;\cong\;} f_2
		\]
		that restricts to the identity on vertices, i.e. $\varphi|_{V(f_1)} = \mathrm{id}_{V(f_1)}$.
		
		Define an equivalence relation $\sim_{\varphi}$ on $P$ by
		\[
		x \sim_{\varphi} \varphi(x) \quad \text{for all } x \in f_1.
		\]
		%and $x \sim_{\varphi} y$ only if $x=y$ otherwise.  
		The \emph{merging} of $f_1$ and $f_2$ is the quotient space
		\[
		P/(f_1 \sim_{\varphi} f_2) := P/{\sim_{\varphi}},
		\]
		%and the image $\pi_{\varphi}(f_1) = \pi_{\varphi}(f_2)$ is called the \emph{merged} $i$-simplex.
		the resulting space $P/(f_1 \sim_{\varphi} f_2)$ is called the $(d-1)$-dimensional topological hypergraph obtained from $P$ by \emph{merging $f_1$ and $f_2$}. 
	\end{definition}

	\begin{definition}[minor]
		Let $G$ be a finite, simple $(d-1)$-dimensional topological hypergraph. 
		$H$ is called a minor of $G$ if $H$ can be formed from $G$ by deleting, merging and contracting simplices.  
		
		%We say that \emph{$H$ is a minor of $G$} (and write $H \leq G$) if $H$ can be obtained from a subgraph of $G$ by a sequence of edge contractions.  
	\end{definition}

	\begin{definition}[{\it $\mathbb{R}^d$-embedding}]
		Let $G$ be a $(d-1)$-dimensional topological hypergraph, an $\mathbb{R}^d$-embedding $G'$ of $G$ can be regarded as a hypergraph isomorphic to $G$ and is embedded in $\mathbb{R}^d$. 
	\end{definition}

	We conclude this section with a brief summary. The $\mathbb{R}^d$-hypergraph can be considered as an generalization of the definition of the planar graph into higher-dimensional space or as a special type of CW complex. Whether it is an $\mathbb{R}^d$-hypergraph, or a non-$\mathbb{R}^d$-hypergraph, they are essentially special cases of CW complex. A thorough understanding of these definitions lays a solid foundation for subsequent proofs.

	\begin{lemma}\label{structure}
		Let $K$ be a $(d-1)$-dimensional topological hypergraph. 
		Then every connected component of the complement \(S^d\setminus K\) is homeomorphic to the open $d$-ball.
	\end{lemma}
	
	\begin{proof}
		Note that the generalized Schoenflies theorem (Lemma \ref{Schoenflies}) guarantees that no pathological cases like \emph{Alexander's horned sphere} can arise in this proof.
		
		Since $K$ can be viewed as a CW complex, the complement \(S^d\setminus K\) is partitioned into $d$-dimensional components $D_1,D_2,...,D_x$. If there exists a $D_i\in \{D_1,D_2,...,D_x\}$ such that $D_i$ is not homeomorphic to the open $d$-ball, then the boundary of $D_i$ (denoted by $\partial(D_i)$) is not $S^{d-1}$ by Lemma \ref{Schoenflies}. Furthermore, $\partial(D_i)$ is not $(d-2)$-connected.  Therefore, there exists a $j$-dimensional sphere $S^j\subseteq \partial(D_i)$ ($j\leq d-2$) such that $S^j$ cannot continuously deform into a base point. On the other hand, it is clear that $S^j \subseteq K$, which implies that $K$ is not $(d-2)$-connected, a contradiction. 
		
	\end{proof}
	
	Lemma \ref{structure} implies that a $(d-1)$-dimensional topological hypergraph can be regarded as a union of internally disjoint $(d-1)$-dimensional spheres.

	\begin{lemma}[generalized Schoenflies theorem \cite{brown1961proof}]\label{Schoenflies}
		Let $\varphi\colon S^{n-1} \hookrightarrow S^n$ be a topological embedding in a locally flat way (that is, the embedding extends to that of a thickened sphere) with $n \geq 2$, and let $A$ be the closure of a component of $S^n \setminus \varphi(S^{n-1})$, then $A$ is homeomorphic to the closed $n$-dimensional ball $D^n$.
	\end{lemma}

	\section{Bridges}\label{section-Br}
	In this section, we aim to establish some lemmas of {\it bridge} in higher-dimensional spaces. 
	Let $H$ be a proper subgraph of a connected $(d-1)$-dimensional topological hypergraph $G$. The set $A_{d-1}(G) \backslash A_{d-1}(H)$ may be partitioned into classes as follows. 
	For each component $F$ of $G[V(G)-V(H)]$, there is a class consisting of the $(d-1)$-dimensional topological hyperedges of $F$ together with the $(d-1)$-dimensional topological hyperedges linking $F$ to $H$. 
	Each remaining $(d-1)$-dimensional topological hyperedges $e$ defines a singleton class $\{e\}$. 
	The subgraphs of $G$ induced by these classes are the bridges of $H$ in $G$. 
	It follows immediately from this definition that bridges of $H$ can intersect only in $i$-dimensional topological hyperedges of $H$ with $i\leq d-2$, and that any two vertices of a bridge of $H$ are connected by a path in the bridge that is internally disjoint from $H$. 
	
	For a bridge $B$ of $H$, the {\it projection} of $B$ is denoted by $p(B)= B \cap H$; 
	the elements of $V(B \cap H)$ are called its {\it vertices of attachment} to $H$, and the remaining vertices of $B$ are its internal vertices. 
	A bridge is trivial if it has no internal vertices. 
	A bridge with $k$ vertices of attachment is called a $k$-bridge. Two bridges with the same projection are equivalent bridges. 
	
	%kkk A bridge with $k$ vertices of attachment is called a $k$-bridge. Two bridges with the same vertices of attachment and same projection are equivalent bridges. 

	We are concerned here with bridges of $(d-1)$-sphere, and all bridges are understood to be bridges of a given $(d-1)$-sphere $S^{d-1}$. Thus, to avoid repetition, we abbreviate `bridge of $S^{d-1}$' to `bridge' in the coming discussion. 
	
	\begin{lemma}\label{projection-connect}
		Let $G$ be a $(d-1)$-dimensional topological hypergraph and $B$  a bridge of $(d-1)$-sphere $S'$, then the projection of $B$, denoted by $p(B)= B \cap S'$,  is a connected $\mathbb{R}^{d-1}$-hypergraph. 
	\end{lemma}
	
	\begin{proof}
		
		Note that we only need to prove that the $i$-th homotopy group of $p(B)$ is trivial for all $i\in \{1,2,3,...,d-3\}$. 
		
		Lemma \ref{structure} implies that every $(d-1)$-dimensional topological hypergraph can be regarded as a union of $(d-1)$-dimensional spheres $S^{d-1}$. Let $S' \cup B = S_1 \cup S_2 \cup \dots \cup S_x$, where each $S_i$ $(i \in \{1, 2, \dots, x\})$ is a $(d-1)$-dimensional sphere. Without loss of generality, we assume that $S' = S_1$. 
		
		It is easy to observe that for each $S_i \in \{ S_2, S_3, \dots, S_x \}$, the intersection $D_i = S_i \cap S'$ is either empty or an $i$-dimensional ball ($i\leq d-1$). We only need to consider the case when $i= d-1$, otherwise there exists $\{S_{x_1}, S_{x_2}, ...,S_{x_s}\}\subseteq \{ S_2, S_3, \dots, S_x \}$ such that $S_j \cap S'$ is a $(d-1)$-dimensional ball for all $S_j\in \{S_{x_1}, S_{x_2}, ...,S_{x_s}\}$, and $S_i \cap S' \subseteq \bigcup_{S_j\in \{S_{x_1}, S_{x_2}, ...,S_{x_s}\}}(S_j \cap S')$. 
		
		In case when $D_i = S_i \cap S'$ is a $(d-1)$-ball, the boundary of $D_i$, denoted $\partial(D_i)$, is a $(d-2)$-sphere, note that $p(B) = \bigcup_{i \in \{2, 3, \dots, x\}} \partial(D_i)$.
		
		Since $p(B)$ can be viewed as a union of finitely many $(d-2)$-dimensional spheres, and each pair of spheres intersects only along their boundaries (i.e., for all $i, j \in \{2, 3, \dots, x\}$, the intersection $D_i \cap D_j$ is either empty or a contractible disk), it follows from Lemmas \ref{attach} and \ref{attach2} that $\pi_k(p(B)) \cong \pi_k(S^{d-2})$ for all $k\leq d-3$. Therefore, the $i$-th homotopy group of $p(B)$ is trivial for all $i\in \{1,2,3,...,d-3\}$. 
		
	\end{proof}

	The projection of a $k$-bridge $B$ with $k\geq d-1$ effects a partition of $S^{d-1}$ into $r$ disjoint segments, called the segments of $B$. Two bridges avoid each other if all the vertices of attachment of one bridge lie in a single segment of the other bridge; otherwise, they overlap.

	Two bridges $B$ and $B'$ are skew if $p(B)$ contains a $(d-2)$-sphere $C(B)$ as a subgraph which effects a partition of $S^{d-1}$ into two disjoint segments $\{R_1, R_2\}$, and there are distinct vertices $u,v$ in vertices of attachment of $B'$ such that $u$ and $v$ are in different segment of $\{R_1, R_2\}$. Note that there is a $uv$-path $P(u,v)$ in $S^{d-1}$ such that $P(u,v)\cap C(B) \neq \phi$ by Lemma \ref{projection-connect} and Jordan-Brouwer Separation Theorem (Lemma \ref{jordan-Brouwer}). 
	
	\begin{lemma}[Jordan-Brouwer Separation Theorem \cite{hatcher2002algebraic}]\label{jordan-Brouwer}
		Let $X$ be a $d$-dimensional topological sphere in the $(d+1)$-dimensional Euclidean space $\mathbb{R}^{d+1}$ ($d > 0$), i.e. the image of an injective continuous mapping of the $d$-sphere $S^d$ into $\mathbb{R}^{d+1}$, then the complement $Y$ of $X$ in $\mathbb{R}^{d+1}$ consists of exactly two connected components. One of these components is bounded (the interior) and the other is unbounded (the exterior). The set $X$ is their common boundary.
	\end{lemma}
	
	%We give an example of skew for $S^2$. As shown in Figure \ref{new10}, both $u$ and $v$ are in the inner region of $S^2$, the bridge induced by $\{u,u_1, u_2,...,u_5\}$ is denoted by $B_1$, the bridge induced by $\{v,v_1, v_2\}$ is denoted by $B_2$. It is obvious that $B_1$ effects a partition of $S^2$ into two disjoint segments (two hemispheres), and $v_1$ and $v_2$ lie in different segments. 
	%
	%\begin{figure}
	%	\centering     
	%	\includegraphics[width=0.65\linewidth]{new11}
	%	\caption{$B_1$ and $B_2$ are skew of $S^2$.}
	%	\label{new10}
	%\end{figure}

	\begin{lemma}\label{bridges}
		Overlapping bridges of a $(d-1)$-dimensional topological hypergraph are either skew or else equivalent $(d+1)$-bridges.
	\end{lemma}
	
	\begin{proof}
		Suppose that bridges $B$ and $B'$ overlap. Clearly, each of them must have at least $d$ vertices of attachment. If either $B$ or $B'$ is a $d$-bridge, it is easily verified that they must be skew (two equivalent $d$-bridges cannot overlap). We may therefore assume that both $B$ and $B'$ have at least $(d+1)$ vertices of attachment. 
		
		If $B$ and $B'$ are not equivalent bridges, then all the vertices of attachment of one bridge cannot lie in a single segment of the other bridge. Without loss of generality, let $C(B)\subseteq p(B)$ be a $(d-2)$-sphere which effects a partition of $S^{d-1}$ into $2$ disjoint segments $\{R_1, R_2\}$, then there exist distinct vertices $u',v'$ in vertices of attachment of $B'$ such that $u'$ and $v'$ are in different segment of $\{R_1, R_2\}$.  It follows that $B$ and $B'$ are skew. 
		
		If $B$ and $B'$ are equivalent $k$-bridges, then $k\geq d+1$. If $k\geq d+2$, then there exists a subgraph $D\subseteq B$ which effects a partition of $S^{d-1}$ into $2$ disjoint segments $\{R_1, R_2\}$ and all the vertices of attachment of $B'$ cannot lie in a single segment of $B$ since they are overlapping bridges, therefore, $B$ and $B'$ are skew; if $k=d+1$, they are equivalent $(d+1)$-bridges.

	\end{proof}

	\section{Hyper ear decomposition}
	\label{section-ED}
	
	In this section, we establish several lemmas on {\it ear decomposition} in higher-dimensional spaces.
	Let $G$ be a $(d-1)$-dimensional topological hypergraph which is $1$-sphere connected. Note that the $1$-sphere connected $(d-1)$-dimensional topological hypergraph contains a subgraph $G_0$ which is homeomorphic to $S^{d-1}$. We describe here a simple recursive procedure for generating any such hypergraph starting with an arbitrary $(d-1)$-sphere in the $(d-1)$-dimensional topological hypergraph. 
	
	\begin{definition}[{\it hyper ear}]\label{ear}
		Let $F$ be a subgraph of a $(d-1)$-dimensional topological hypergraph $G$. 
		A {\it hyper ear} of $F$ in $G$ is a nontrivial $(d-1)$-ball in $G$ whose boundary lies in $F$ but whose interior vertices are not contained in $F$.
	\end{definition}

	\begin{definition}[{\it hyper ear decomposition}]\label{hyper-ear-decomposition}
		A nested sequence of a $(d-1)$-dimensional topological hypergraph is a sequence $(G_0,G_1,\ldots,G_k)$ of $(d-1)$-dimensional topological hypergraphs such that $G_i\subseteq G_{i+1}$, $0\leq i\leq k-1$. A {\it hyper ear decomposition} of a $2$-sphere connected $(d-1)$-dimensional topological hypergraph $G$ is a nested sequence $(G_0,G_1,\ldots,G_k)$ of $G$ such that: 
		
		\begin{itemize}
			\item $G_0$ is homeomorphic to $S^{d-1}$;
			\item $G_{i+1}= G_i \cup P_i$ where $P_i$ is a hyper ear of $G_i$ in $G$ for $0\leq i\leq k-1$;
			\item $G_k=G$.
		\end{itemize}
	\end{definition}
	
	\begin{lemma}\label{ear-exist}
		In a $(d-1)$-dimensional topological hypergraph $G$ with $|V(G)|\geq d+1$, if $G$ is $1$-sphere connected, then $G$ has a hyper ear decomposition. 
	\end{lemma}
	
	\begin{proof}
		On the one hand, $G$ is $1$-sphere connected. On the other hand, since $G$ contains at least $(d+1)$ vertices, it must contain a subgraph homeomorphic to $S^{d-1}$. 
		
		Since $G$ is homeomorphic to the union of a finite collection of $(d-1)$-spheres by Definition \ref{special0} and Lemma \ref{structure}, it is obvious that $G$ has a hyper ear decomposition. 
	\end{proof}

	%\begin{lemma}\label{ear-exist}
	%	In a $(d-1)$-dimensional topological hypergraph $G$ with $|V(G)|\geq d+1$, if the $1$-skeleton of $G$ is $d$-connected,  then $G$ has a hyper ear decomposition. 
	%\end{lemma}
	%
	%%\begin{lemma}\label{ear-exist}
	%%	If the $1$-skeleton of a $(d-1)$-dimensional topological hypergraph $G$ is $d$-connected and $|V(G)|\geq d+1$, then $G$ has a hyper ear decomposition. 
	%%\end{lemma}
	%
	%%\begin{lemma}\label{ear-exist}
	%%	The $(d-1)$-dimensional topological hypergraph $G$ with $|V(G)|\geq d+1$ has a hyper ear decomposition. 
	%%	
	%%\end{lemma}
	%
	%\begin{proof}
	%	On the one hand, the $1$-skeleton of $G$ is $d$-connected. On the other hand, since $G$ contains at least $(d+1)$ vertices, it must contain at least one $(d-1)$-dimensional sphere $S^{d-1}$.
	%	
	%	Since $G$ is homeomorphic to the union of a finite collection of $(d-1)$-spheres by Definition \ref{special0} and Lemma \ref{structure}, it is obvious that $G$ has a hyper ear decomposition. 
	%\end{proof}
	%
	%%\begin{proof}
	%%	On the one hand, the $1$-skeleton of $G$ is $d$-connected by Lemma \ref{sphere3}. On the other hand, since $G$ contains at least $(d+1)$ vertices, it must contain at least one $(d-1)$-dimensional sphere $S^{d-1}$.
	%%	
	%%	Since $G$ is homeomorphic to the union of a finite collection of $(d-1)$-spheres by Definition \ref{special0} and Lemma \ref{structure}, it is obvious that $G$ has a hyper ear decomposition. 
	%%\end{proof}

	\begin{lemma}\label{E1}
		In a $(d-1)$-dimensional topological hypergraph $G$ with $|V(G)|\geq d+1$, if $G$ is $1$-sphere connected, then each maximal connected region in $\mathbb{R}^d\backslash G$ is bounded by a $(d-1)$-sphere. 
		
	\end{lemma}
	
	\begin{proof}
		Note that $G$ has a hyper ear decomposition by Lemma \ref{ear-exist}. 
		Consider a hyper ear decomposition $(G_0,G_1,...,G_k)$ of $G$, where $G_0$ is homeomorphic to $S^{d-1}$, $G_k= G$, and, for $0\leq i \leq k-2$, $G_{i+1}= G_i \cup P_i$ is a $1$-sphere connected subgraph of $G$, where $P_i$ is an ear of $G_i$ in $G$. Since $G_0$ is homeomorphic to $S^{d-1}$, the two maximal connected regions of $G_0$ are clearly bounded by a $(d-1)$-sphere. Assume, inductively, that all maximal connected regions of $G_i$ are bounded by $(d-1)$-spheres, where $i\geq 0$. Because $G_{i+1}$ is a $1$-sphere connected $(d-1)$-dimensional topological hypergraph, the ear $P_i$ of $G_i$ is contained in some maximal connected region $f$ of $G_i$. Each region of $G_i$ other than $f$ is a region of $G_{i+1}$ as well, and so, by the  induction hypothesis, is bounded by a $(d-1)$-sphere. On the other hand, the region $f$ of $G_i$ is divided by $P_i$ into two regions of $G_{i+1}$, and it is easy to see that these regions are also bounded by $(d-1)$-spheres. 
	\end{proof}
	
	%\begin{lemma}\label{E1}
	%	In a $(d-1)$-dimensional topological hypergraph $G$ with $|V(G)|\geq d+1$, if the $1$-skeleton of $G$ is $d$-connected, then each maximal connected region or $\mathbb{R}^d\backslash G$ is bounded by a $(d-1)$-sphere. 
	%	
	%\end{lemma}
	%
	%\begin{proof}
	%	Note that $G$ has a hyper ear decomposition by Lemma \ref{ear-exist}. 
	%	Consider an ear decomposition $(G_0,G_1,...,G_k)$ of $G$, where $G_0$ is homeomorphic to $S^{d-1}$, $G_k= G$, and, for $0\leq i \leq k-2$, $G_{i+1}= G_i \cup P_i$ is a $d$-connected subgraph of $G$, where $P_i$ is an ear of $G_i$ in $G$. Since $G_0$ is homeomorphic to $S^{d-1}$, the two maximal connected regions of $G_0$ are clearly bounded by a $(d-1)$-sphere. Assume, inductively, that all maximal connected regions of $G_i$ are bounded by $(d-1)$-spheres, where $i\geq 0$. Because $G_{i+1}$ is a $d$-connected $(d-1)$-dimensional topological hypergraph, the ear $P_i$ of $G_i$ is contained in some maximal connected region $f$ of $G_i$. Each region of $G_i$ other than $f$ is a region of $G_{i+1}$ as well, and so, by the  induction hypothesis, is bounded by a $(d-1)$-sphere. On the other hand, the region $f$ of $G_i$ is divided by $P_i$ into two regions of $G_{i+1}$, and it is easy to see that these regions are also bounded by $(d-1)$-spheres. 
	%\end{proof}

	\begin{lemma}\label{sphere}
		In a $(d-1)$-dimensional topological hypergraph $G$, if $G$ is $2$-sphere connected, then the neighbors of any vertex lie on a common $(d-1)$-sphere. 
	\end{lemma}
	
	\begin{proof}
		Let $v$ be a vertex of $G$, then $G-v$ is $1$-sphere connected, so each maximal connected region of $G-v$ is bounded by a sphere according to Lemma \ref{E1}. If $f$ is the region of $G-v$ in which the vertex $v$ was situated, the neighbors of $v$ lie on its bounding sphere $\partial(f)$. 
		
	\end{proof}

	\section{$S$-component and Connectivity}\label{connectivity}
	
	%We need to establish some lemmas before proving our main results.
	
	\begin{lemma}\label{10.33}
		Let $G$ be a $(d-1)$-dimensional topological hypergraph with a $(d-2)$-dimensional $1$-sphere cut $\{x_1,x_2, ..., x_p\}$, then each marked $\{x_1,x_2, ..., x_p\}$-component of $G$ is isomorphic to a minor of $G$. 
	\end{lemma}
	
	\begin{proof}
		Let $H$ be an $\{x_1,x_2, ..., x_p\}$-component of $G$, with marker topological hyperedge $e$. Let $H'$ be another $\{x_1,x_2, ..., x_p\}$-component of $G$, with marker topological hyperedge $e$, then there is a $(d-1)$-ball $B$ such that $B\subseteq H'$ and $B\cup e$ is a $(d-1)$-sphere by Lemma \ref{ear-exist}. It is easy to verify that $H$ is isomorphic to a minor of $G$ by contract $H'$ into a single topological hyperedge $e$. 
		
	\end{proof}

	\begin{lemma}\label{Exercise 10.4.1}
		Let $G_1$ and $G_2$ be $(d-1)$-dimensional topological hypergraphs whose intersection is isomorphic to a $(d-1)$-dimensional topological hyperedge $\sigma_S$ with $V(\sigma_S)= \{x_1,x_2, ..., x_p\}$, 
		%	\[
		%	E(K_d)= \{a_i|a_i \text{ is a (d-2)-dimensional topological hyperedge with vertex set }V(K_d)\backslash \{x_i \} \text{,  }(i\in \{1,2,...,d\}) \},
		%	\]
		%	\[
		%	\begin{aligned}
			%		E(K_d) = \{ a_i \mid &\, a_i \text{ is a $(d-2)$-dimensional topological hyperedge with vertex set }\\
			%		& V(K_d)\backslash \{x_i\},\ (i\in \{1,2,\dots,d\}) \},
			%	\end{aligned}
		%	\]
		then $G_1\cup G_2$ is an $\mathbb{R}^d$-hypergraph. 
	\end{lemma}
	
	\begin{proof}
		Let $H$ be a $(d-1)$-hyperplane, $\sigma_S\subseteq H$. At this point, the hyperplane $H$ divides $\mathbb{R}^d$ into two disconnected regions, denoted by $R_1$ and $R_2$, respectively. We embed $G_1$ into $R_1$ and $G_2$ into $R_2$ in such a way that $G_1$ and $G_2$ intersect only at $\sigma_S$. 
		
		By contradiction, suppose the $i$-th homotopy group of $G_1 \cup G_2$ is nontrivial for some $i \in \{1, 2, \ldots, d-2\}$, then there must exist an $i$-sphere $S^i$ that cannot be continuously contracted to the base point. If $S^i$ belongs to either $G_1$ or $G_2$, then it can be continuously contracted to the base point, which leads to a contradiction. Therefore, $S^i$ must intersect both $G_1$ and $G_2$. Let $S^i\cap G_1= L_1$, $S^i\cap G_2= L_2$. We first transform $L_1$ into $L_3$ by homotopy, such that $L_3$ belongs to $H$. It is easy to verify that $L_2$ and $L_3$ belong to $G_2$, thus they can be continuously contracted to the base point. By combining the two homotopy transformations, we obtain that $S^i$ can be continuously contracted to the base point, a contradiction. In conclusion, the assumption is invalid, and the theorem is proven.
		
	\end{proof}

	\begin{lemma}\label{10.34}
		Let $G$ be a $(d-1)$-dimensional topological hypergraph with a $(d-2)$-dimensional $1$-sphere cut $\{x_1,x_2, ..., x_p\}$, then $G$ is an $\mathbb{R}^d$-hypergraph if and only if each of its marked $\{x_1,x_2, ..., x_p\}$-components is an $\mathbb{R}^d$-hypergraph. 
	\end{lemma}
	
	\begin{proof}
		Suppose, first, that $G$ is an $\mathbb{R}^d$-hypergraph. By Lemma \ref{10.33}, each marked $\{x_1,x_2, ..., x_p\}$-component of $G$ is isomorphic to a minor of $G$, hence is $\mathbb{R}^d$-hypergraph. 
		
		Conversely, suppose that $G$ has $k$ marked $\{x_1,x_2, ..., x_p\}$-components each of which is an $\mathbb{R}^d$-hypergraph. Let $e$ denote their common marker topological hyperedge. Applying Lemma \ref{Exercise 10.4.1} and induction on $k$, it follows that $G + e$ is an $\mathbb{R}^d$-hypergraph, hence so is $G$. 
	\end{proof}
	
	By Lemma \ref{10.34}, we know that to prove a $(d-1)$-dimensional topological hypergraph can be embedded in $\mathbb{R}^d$, it is sufficient to show that all of its marked $\{x_1,x_2, ..., x_p\}$-components can be embedded in $\mathbb{R}^d$.

	%\subsection{Connectivity}

	Before proving Theorem \ref{anti-minor}, we need a lemma regarding sphere connectivity. 
	
	\begin{lemma}
		\label{connected}
		Let $G$ be a $2$-sphere connected topological hypergraph on at least $(d+2)$ vertices, then $G$ contains an edge $e$ such that $G/e$ is $2$-sphere connected. 
	\end{lemma}
	
	\begin{proof}
		We now distinguish two cases.
		
		\noindent\textbf{Case 1.} There exists an edge $e = xy$ in $G$ such that the vertex connectivity of the $1$-skeleton of $G/e$ is the same as that of $G$.
		
		\noindent\textbf{Case 2.} For every edge $e = xy$ in $G$, the vertex connectivity of $G/e$ is strictly smaller than that of $G$.
		
		In what follows, we only prove Case~1, since Case~2 can be established by analogous arguments.

		Suppose the theorem is false. Then, for any edge $e = xy$ of $G$, the contraction $G/e$ is not $2$-sphere connected. Let $w$ be the vertex resulting from the contraction of $e$. 
		By Lemma \ref{not-c}, there exists vertex set $\{x, y, z_1, z_2, ..., z_{\kappa -1}\}$ such that $\{x, y, z_1, z_2, ..., z_{\kappa -1}\}$ is a $2$-sphere cut of $G$. 
		
		Choose $e=xy$ and $\{z_1, z_2, ..., z_{\kappa -1}\}$ in such a way that $G - \{x,y,z_1, z_2, ..., z_{\kappa -1}\}$ has a component $F$ with as many vertices as possible. 
		Consider the graph $G - \{z_{\kappa -1}\}$. Because $G$ is $2$-sphere connected, $G- \{z_{\kappa -1}\}$ is $1$-sphere connected. Moreover $G-\{z_{\kappa -1}\}$ has the $1$-sphere cut $\{x,y, z_1, ..., z_{\kappa -2}\}$. It follows that the $\{x,y, z_1, ..., z_{\kappa -2}\}$-component $H = G[V (F) \cup \{x,y, z_1, ..., z_{\kappa -2}\}]$ is $1$-sphere connected. 
		
		Let $u$ be a neighbour of $z_{\kappa -1}$ in a component of $G - \{x,y,z_1, z_2, ..., z_{\kappa -1}\}$ different from $F$. Since $f = z_{\kappa -1}u$ is an edge of $G$, and $G$ is a counterexample to Lemma \ref{connected}, there is a vertex set $\{v_1, v_2, ..., v_{\kappa -1}\}$ such that $\{z_{\kappa -1},u, v_1, v_2, ..., v_{\kappa-1}\}$ is a $2$-sphere cut of $G$, too. Here we assume that the vertex connectivity of the $1$-skeleton of $G/f$ is the same as that of $G/e$. If the vertex connectivity of the $1$-skeleton of $G/f$ is smaller than that of $G/e$, the situation reduces to an easier case. The vertices $\{v_1, v_2, ..., v_{\kappa -1}\}$ might or might not lie in $H$. 
		
		Let $v$ denote the vertex obtained by contracting the edge $f = z_{\kappa-1}u$. Observe that the induced subgraph $G/f[\{v_1, v_2, \dots, v_{\kappa-1}, v\}]$ is a $(d-2)$-sphere. Consequently, the induced subgraph $G[\{v_1, v_2, \dots, v_{\kappa-1}\}]$ is a $(d-2)$-ball. 
		
		Moreover, because $H$ is $1$-sphere connected and $G[\{v_1, v_2, ..., v_{\kappa -1}\}]$ is not a $(d-2)$-sphere, $H-\{v_1, v_2, ..., v_{\kappa -1}\}$ is connected (where, if there exists $v_i\in \{v_1, v_2, ..., v_{\kappa -1}\}$ such that $v_i \in V (H)$, we set $H- v_i = H$), and thus is contained in a component of $G - \{z,u, v_1, v_2, ..., v_{\kappa -1}\}$. But this component has more vertices than $F$ (because $H$ has $\kappa$ more vertices than $F$), contradicting the choice of the edge $e$ and the vertex $v$. 
		
	\end{proof}
	
	\begin{lemma}\label{not-c}
		Let $G$ be a $2$-sphere connected topological hypergraph on at least $(d+2)$ vertices, and let $e = xy$ be an edge of $G$ such that $G/e$ is not $2$-sphere connected. Then there exist some vertices such that $G[\{x,y, z_1, z_2, ..., z_{\kappa -1}\}]$ is a $2$-sphere cut of $G$ ($\kappa$ is the vertex connectivity of the $1$-skeleton of $G/e$). 
	\end{lemma}
	
	\begin{proof}
		Let $\{z_1, z_2, ..., z_{\kappa -1}, w\}$ be an $1$-sphere cut of $G/e$. At least $\kappa -1$ of these $\kappa$ vertices, say $\{z_1, z_2, ..., z_{\kappa -1}\}$, are not the vertices resulting from the contraction of $e$. Set $F = G- \{z_1, z_2, ..., z_{\kappa-1}\}$, then $F/e  = (G/e) - \{z_1, z_2, ..., z_{\kappa-1}\}$ has a cut vertex, namely $w$. 
		
		If $w$ is not the vertex resulting from the contraction of $e$, then $\{z_1, z_2, ..., z_{\kappa-1}, w\}$ must be a $2$-sphere cut of $G$, a contradiction. Hence $w$ must be the vertex resulting from the contraction of $e$. Therefore $G - \{x,y,z_1, z_2, ..., z_{\kappa-1}\} = (G/e) - \{z_1, z_2, ..., z_{\kappa-1},w\}$ is disconnected, in other words, $\{x,y,z_1, z_2, ..., z_{\kappa-1}\}$ is a $2$-sphere cut of $G$.
	\end{proof}

	\section{Proof of Theorem \ref{anti-minor} and \ref{chromatic-number}}\label{recognizing}
	
	Before proceeding to the proof of Theorem \ref{anti-minor}, we need to define a specific class of structures based on bipartite graphs that cannot be embedded into $\mathbb{R}^d$.

	\subsection{Complete bipartite hypergraph}
	
	%\begin{definition}[Chordless Cycle]
	%	Let $G=(V,E)$ be a graph. A cycle $C$ in $G$ is called a \emph{chordless cycle} if the subgraph of $G$ induced by the vertex set of $C$ is exactly $C$ itself. 
	%\end{definition}
	
	% --- Definition of Link ---
	\begin{definition}[Link of a Vertex \cite{hatcher2002algebraic}]\label{link}
		Let $K$ be a simplicial complex and let $v \in V(K)$ be a vertex of $K$. The \textit{link} of $v$ in $K$, denoted by $\mathrm{Lk}(v, K)$ (or simply $\mathrm{Lk}(v)$), is the subcomplex of $K$ consisting of all simplices $\sigma \in K$ such that $v \notin \sigma$ but $\{v\} \cup \sigma \in K$. Formally:
		\[
		\mathrm{Lk}(v, K) = \{ \sigma \in K \mid v \notin \sigma \text{ and } \{v\} \cup \sigma \in K \}.
		\]
		Geometrically, the link represents the boundary of the neighborhood of $v$ in the complex.
	\end{definition}
	
	% --- Definition of Suspension ---
	\begin{definition}[Suspension \cite{hatcher2002algebraic}]\label{suspension}
		Let $X$ be a topological space. The \textit{suspension} of $X$, denoted by $\Sigma X$, is the quotient space formed from the product $X \times [0, 1]$ by collapsing $X \times \{0\}$ to a single point (the south pole) and $X \times \{1\}$ to a single point (the north pole). Formally:
		\[
		\Sigma X = (X \times [0, 1]) / \sim,
		\]
		where the equivalence relation $\sim$ identifies all points $(x, 0)$ to a point $v_S$ and all points $(x, 1)$ to a point $v_N$. 
		
		%	In the context of simplicial complexes, the suspension of a complex $K$ is isomorphic to the join of $K$ with a 0-sphere (two distinct points), i.e., $\Sigma K \cong K * S^0$.
	\end{definition}

	%\begin{definition}[chordless $i$-sphere]
	%	An chordless $i$-sphere $S^i$ of a $(d-1)$-dimensional topological hypergraph $G$ is a sphere that is an induced subgraph of $G$. 
	%\end{definition}

	\begin{definition}[complete dimension-raising function]\label{complete-dimension-raising}
		Let $G$ be a graph, and let $U^x_{comp}(G)$ denote the \emph{complete dimension-raising function} of the graph $G$. 
		This definition can be obtained from Definition~\ref{dimension-raising} by replacing the phrase ``induced $i$-sphere'' with ``chordless $i$-sphere'', while leaving all other parts unchanged. If $G$ is a bipartite graph, then $U^x_{\mathrm{comp}}(G)$ is called a \emph{complete bipartite hypergraph}. 
	\end{definition}
	
	For example, $U^2_{comp}(K_{2,3})$ denotes the complex obtained by filling every $4$-cycle (chordless $1$-sphere) of the complete bipartite graph $K_{2,3}$ with a $2$-ball. It is easy to see that $U^2_{comp}(K_{2,3})$ is homeomorphic to a triangular bipyramid.

	We get the following theorem by Corollary \ref{a+b-3=d} and Lemma \ref{a+b-4=d}. 
	
	\begin{theorem}\label{min-emb}
		Let $K_{a,b}$ be a complete bipartite graph with $a\geq 2$, $b\geq 2$, $a+b\geq 7$ and $d = a + b - 4$. 
		Then $U^{a+b-5}_{comp}(K_{a,b})$ is minimally non-embeddable in $\mathbb{R}^d$. 
		That is, for any $i$-face $f$ is removed from $U^{a+b-5}_{comp}(K_{a,b})$ with $i \geq 2$, 
		the resulting $U^{a+b-5}_{comp}(K_{a,b}) \setminus f$ can be embeded in $\mathbb{R}^d$. 
	\end{theorem}

	Before proving Theorem~\ref{min-emb}, we first establish several auxiliary lemmas.

	By Definition~\ref{suspension}, we obtain the following corollary.
	
	\begin{corollary}
		\label{lemma:suspension}
		Let $\Delta^{d+1}$ be the standard simplex on $d+2$ vertices, and let $(\Delta^{d+1})^{(k)}$ denote its $k$-skeleton. Then $U_{comp}^{d-1}(K_{2,d+2})$ is homeomorphic to the suspension of the $(d-2)$-skeleton of $\Delta^{d+1}$:
		\[
		U_{comp}^{d-1}(K_{2,d+2}) \cong \Sigma \left( (\Delta^{d+1})^{(d-2)} \right).
		\]
	\end{corollary}
	
	%\begin{proof}
	%	Let the partition $A = \{a_1, a_2\}$ act as the suspension points (poles). The suspension of a space $X$, $\Sigma X$, can be viewed as the union of two cones over $X$, $C_+(X)$ and $C_-(X)$, glued along $X$. Here, $a_1$ is the apex of $C_+$ and $a_2$ is the apex of $C_-$.
	%	
	%	Consider the link of $a_1$ in the constructed complex.
	%	\begin{itemize}
		%		\item At step 1 ($U^2$), we fill every induced 4-cycle. These cycles are of the form $a_1-b_i-a_2-b_j-a_1$. In the neighborhood of $a_1$, this filling connects every pair $\{b_i, b_j\}$. Thus, the link of $a_1$ contains the 1-skeleton of the simplex on $B$ (which is $K_{d+2}$).
		%		\item At general step $k$, we fill induced $k$-spheres formed by $A$ and a subset $B' \subset B$ of size $k+1$. Filling such a sphere with a $(k+1)$-disk implies adding a $k$-simplex to the link of $a_1$ formed by vertices $B'$.
		%		\item The process stops after filling induced $(d-2)$-spheres, which correspond to subsets of $B$ of size $d-1$.
		%	\end{itemize}
	%	Therefore, the link of $a_1$ consists of all simplices formed by vertices in $B$ up to dimension $d-2$. This is precisely $(\Delta^{d+1})^{(d-2)}$, where the vertices of $\Delta^{d+1}$ are identified with $B$. The entire complex is generated by the cones from $a_1$ and $a_2$ to this structure, forming the suspension.
	%\end{proof}

	\begin{lemma}[Local Embedding Condition \cite{rourke2012introduction}]
		\label{lemma:link}
		If a simplicial complex $K$ embeds piecewise linearly into $\mathbb{R}^d$, then for every vertex $v \in K$, the link $\text{Lk}(v, K)$ must embed into the sphere $S^{d-1}$ (and consequently into $\mathbb{R}^{d-1}$ via stereographic projection, assuming the link is not the entire sphere). 
	\end{lemma}

	%\begin{proposition}[Local Embedding Condition]
	%	\label{prop:local_embedding}
	%	A necessary condition for a simplicial complex $K$ to embed piecewise linearly into the Euclidean space $\mathbb{R}^d$ is that for every vertex $v \in K$, the link $\mathrm{Lk}(v, K)$ must embed into the sphere $S^{d-1}$ (or equivalently $\mathbb{R}^{d-1}$ via stereographic projection, provided the link is a proper subset of the sphere). This follows from the fact that the star of a vertex in a PL embedding must be homeomorphic to a cone over its link embedded in a $d$-ball.
	%\end{proposition}
	%
	%\begin{proof}
	%	This is a standard result in piecewise linear topology derived from the preservation of local structures. For a rigorous treatment, see Rourke and Sanderson \cite{rourke1972}.
	%\end{proof}

	The following lemma is a consequence of the Van Kampen-Flores theorem and the theory of piecewise linear embeddings. Instead of appealing to these results abstractly, we provide explicit coordinates for each vertex, yielding a more intuitive construction of an embedding and thereby proving that the $(d-1)$-skeleton $(\Delta^{d+1})^{(d-1)}$ can be embedded into $\mathbb{R}^d$. 
	Moreover, by adding one additional vertex to $(\Delta^{d+1})^{(d-1)}$ and applying Theorem~\ref{jordan-Brouwer}, one can show that the $(d-1)$-skeleton $(\Delta^{d+2})^{(d-1)}$ cannot be embedded into $\mathbb{R}^d$.

	\begin{lemma}\label{lemma:flo}
		%	Let $\Delta^{n}$ denote the standard simplex with $n+1$ vertices, and let $(\Delta^{n})^{(k)}$ denote its $k$-skeleton (the subcomplex consisting of all faces of dimension at most $k$).
		%	\begin{itemize}
			%		\item $(\Delta^{d+1})^{(d-1)}$ can be embedded into $\mathbb{R}^d$. 
			%		\item For $d \geq 2$, $(\Delta^{d+2})^{(d-1)}$ cannot be embedded into $\mathbb{R}^d$. 
			%	\end{itemize}
		$(\Delta^{d+1})^{(d-1)}$ can be embedded into $\mathbb{R}^d$; $(\Delta^{d+2})^{(d-1)}$ cannot be embedded into $\mathbb{R}^d$ for $d \geq 2$.
	\end{lemma}
	
	\begin{proof}
		As shown in Figure \ref{five-vertex-emb}, we assign a coordinate in $\mathbb{R}^d$ to each vertex of the simplex $\Delta^{d+1}$. Let the $d+2$ vertices of $\Delta^{d+1}$ be labeled $v_0, v_1, \dots, v_d, v_{d+1}$. Their coordinates are given respectively as $v_0 = (0, 0, \dots, 0)$, $v_1 = (1, 0, \dots, 0)$, $v_2 = (0, 1, \dots, 0)$, ..., $\dots$, $v_d = (0, \dots, 0, 1)$, and $v_{d+1} = (\frac{1}{d+1}, \frac{1}{d+1}, \dots, \frac{1}{d+1})$ (The vertex $v_{d+1}$ is located at the geometric center of the simplex with vertices $v_0, v_1, \dots, v_d$). Taking any $d$ vertices from $\{v_0, v_1, \dots, v_{d+1}\}$ will form a $(d-1)$-simplex. In this case, it is easy to observe that the final complex is the $(d-1)$-skeleton of $\Delta^{d+1}$, and it can be embedded into $\mathbb{R}^d$. 
		
		Furthermore, by the Jordan-Brouwer Separation Theorem (Theorem~\ref{jordan-Brouwer}), $(\Delta^{d+2})^{(d-1)}$ cannot be embedded into $\mathbb{R}^d$. 
		
	\end{proof}
	
	\begin{figure}[htbp]
		\centering
		\tdplotsetmaincoords{70}{120}
		
		\begin{tikzpicture}[tdplot_main_coords, scale=4]
			
			%------------------------------------------------
			% 坐标轴
			%------------------------------------------------
			\draw[->, gray] (0,0,0) -- (1.2,0,0) node[below left] {$x$};
			\draw[->, gray] (0,0,0) -- (0,1.2,0) node[below right] {$y$};
			\draw[->, gray] (0,0,0) -- (0,0,1.2) node[above] {$z$};
			
			%------------------------------------------------
			% 顶点
			%------------------------------------------------
			\coordinate (v0) at (0,0,0);
			\coordinate (v1) at (1,0,0);
			\coordinate (v2) at (0,1,0);
			\coordinate (v3) at (0,0,1);
			\coordinate (v4) at (0.3,0.3,0.3);
			
			%------------------------------------------------
			% 所有 10 个三角形
			%------------------------------------------------
			\filldraw[fill=red!60, opacity=0.45, draw=black]     (v0)--(v1)--(v2)--cycle;
			\filldraw[fill=blue!60, opacity=0.45, draw=black]    (v0)--(v1)--(v3)--cycle;
			\filldraw[fill=green!60, opacity=0.45, draw=black]   (v0)--(v1)--(v4)--cycle;
			\filldraw[fill=orange!70, opacity=0.45, draw=black]  (v0)--(v2)--(v3)--cycle;
			\filldraw[fill=purple!60, opacity=0.45, draw=black]  (v0)--(v2)--(v4)--cycle;
			\filldraw[fill=cyan!60, opacity=0.45, draw=black]    (v0)--(v3)--(v4)--cycle;
			\filldraw[fill=yellow!70!brown, opacity=0.45, draw=black] (v1)--(v2)--(v3)--cycle;
			\filldraw[fill=teal!60, opacity=0.45, draw=black]    (v1)--(v2)--(v4)--cycle;
			\filldraw[fill=magenta!60, opacity=0.45, draw=black] (v1)--(v3)--(v4)--cycle;
			\filldraw[fill=lime!60, opacity=0.45, draw=black]    (v2)--(v3)--(v4)--cycle;
			
			%------------------------------------------------
			% 顶点标注（名字 + 坐标，绝对安全）
			%------------------------------------------------
			\foreach \v/\name/\x/\y/\z in {
				v0/$v_0$/0/0/0,
				v1/$v_1$/1/0/0,
				v2/$v_2$/0/1/0,
				v3/$v_3$/0/0/1,
				v4/$v_4$/\frac{1}{3}/\frac{1}{3}/\frac{1}{3}
			}
			{
				\filldraw[black] (\v) circle (0.6pt)
				node[above right] {\name\ {\scriptsize $(\x,\y,\z)$}};
			}
			
		\end{tikzpicture}
		
		\caption{A geometric realization of the $2$-skeleton of $\Delta^{4}$ in $\mathbb{R}^3$.
			All $\binom{5}{3}=10$ triangular faces are shown in distinct colors. An illustration showing that the $2$-skeleton of $\Delta^{5}$ cannot be embedded into $\mathbb{R}^3$ is provided in Figure~\ref{fig:right-embedding}.}
		
		\label{five-vertex-emb}
	\end{figure}

	\begin{lemma}\label{a=2,b=d+1}
		$U_{comp}^{d-1}(K_{2,d+1})$ admits an embedding into $\mathbb{R}^d$.
	\end{lemma}
	
	\begin{proof}
		By Lemma~\ref{lemma:suspension}, the complex $U_{comp}^{d-1}(K_{2,d+1})$ is homeomorphic to the suspension of the $(d-2)$-skeleton of $\Delta^{d}$. 
		By Lemma~\ref{lemma:flo}, the $(d-2)$-skeleton $(\Delta^{d})^{(d-2)}$ can be embedded into $\mathbb{R}^{d-1}$. 
		It follows that $U_{comp}^{d-1}(K_{2,d+1})$ admits an embedding into $\mathbb{R}^{d}$.
		
	\end{proof}

	Consider the complete bipartite graph $K_{2,d+1}$ with bipartition $A = \{u_1, u_2\} \quad \text{and} \quad B = \{v_1, v_2, \dots, v_{d+1}\}$. 
	It is straightforward to observe that adding edges between vertices within the set $B$ does not affect the embeddability of the complex $U_{comp}^{d-1}(K_{2,d+1})$. This observation leads to the following corollary.
	
	\begin{corollary}\label{a+b-3=d}
		Let $K_{a,b}$ be a complete bipartite graph and let $d = a + b - 3$. Then $U_{comp}^{d-1}(K_{a,b})$ admits an embedding into $\mathbb{R}^d$.
	\end{corollary}

	\begin{lemma}\label{a=2,b=d+2}
		$U_{comp}^{d-1}(K_{2,d+2})$ cannot be embedded into $\mathbb{R}^d$. 
	\end{lemma}
	
	\begin{proof}
		Let $U = U_{comp}^{d-1}(K_{2,d+2})$. Assume the contrary that there exists an embedding $f: U \to \mathbb{R}^d$.
		Consider the vertex $a_1 \in A$. By Lemma \ref{lemma:link}, the existence of an embedding for $U$ implies that the link of $a_1$, denoted $\text{Lk}(a_1, U)$, must embed into $S^{d-1}$.
		From Lemma \ref{lemma:suspension}, we established that:
		\[ 
		\text{Lk}(a_1, U) \cong (\Delta^{d+1})^{(d-2)},
		\]
		implying that $\text{Lk}(a_1, U)$ is the $(d-2)$-skeleton of a simplex with $d+2$ vertices.
		We attempt to embed this skeleton into $S^{d-1}$ (or equivalently $\mathbb{R}^{d-1}$ since the skeleton is not the whole sphere).
		However, by Lemma \ref{lemma:flo}, $(\Delta^{d+1})^{(d-2)}$ cannot be embedded into $\mathbb{R}^{d-1}$.
		This contradiction implies that $U$ cannot be embedded into $\mathbb{R}^d$.
	\end{proof}
	
	It is now known that $U_{comp}^{d-1}(K_{2,d+2})$ is not embeddable in $\mathbb{R}^d$, whereas $U_{comp}^{d-1}(K_{2,d+1})$ admits an embedding in $\mathbb{R}^d$.
	
	For $U_{comp}^{d-1}(K_{2,d+1})$, let $A=\{a_1,a_2\}$ and $B=\{b_1,b_2,\dots,b_{d+1}\}$. 
	Place the vertices of $B$ on a $(d-1)$-dimensional hyperplane $\mathbb{H}\subset \mathbb{R}^d$, and place $a_1$ and $a_2$ in two distinct connected open regions of $\mathbb{R}^d\setminus \mathbb{H}$. 
	This yields an embedding of $U_{comp}^{d-1}(K_{2,d+1})$ in $\mathbb{R}^d$. 
	In this configuration, there always exists a vertex $b_i\in B$ that is enclosed by a $(d-1)$-sphere. 
	Consequently, if one adds an additional vertex $b_{d+2}$ to the original graph, the resulting complex $U_{comp}^{d-1}(K_{2,d+2})$ is no longer embeddable in $\mathbb{R}^d$.
	
	However, if we delete any $(d-1)$-face $f$ from $U_{comp}^{d-1}(K_{2,d+1})$, then there exists a connected open set $S'$ in $\mathbb{R}^d\setminus U_{comp}^{d-1}(K_{2,d+1})$ such that all vertices in $A=\{a_1,a_2\}$ and $B=\{b_1,b_2,\dots,b_{d+1}\}$ lie on the boundary $\partial S'$. 
	By placing the vertex $b_{d+2}$ in the interior of $S'$, we obtain an embedding of $U_{comp}^{d-1}(K_{2,d+2})\setminus f$ in $\mathbb{R}^d$.
	
	Combining these observations with Lemmas~\ref{a=2,b=d+1} and~\ref{a=2,b=d+2}, we obtain the following corollary.

	\begin{corollary}
		$U = U_{comp}^{d-1}(K_{2,d+2})$ is minimally non-embeddable in $\mathbb{R}^d$. That is, if any $(d-1)$-face is removed from $U$, the resulting $U'$ embeds into $\mathbb{R}^d$.
	\end{corollary}

	We get the following corollary by Definition~\ref{link}. 
	
	\begin{corollary}\label{lemma:reduction}
		Let $v \in A$ be a vertex in the partition $A$ of $K_{a,b}$. The link of $v$ in the complex $U_{comp}^{d-1}(K_{a,b})$, denoted by $\text{Lk}(v)$, is combinatorially isomorphic to the complex $U_{comp}^{d-2}(K_{a-1, b})$.
	\end{corollary}

	\begin{lemma}\label{a+b-4=d}
		The complex $U = U_{comp}^{d-1}(K_{a, d-a+4})$ cannot be embedded into $\mathbb{R}^d$ for any $a \geq 2$.
	\end{lemma}
	
	\begin{proof}
		We proceed by induction on the pair $(a,d)$ with respect to the lexicographic order.
		
		%	\textbf{Base Case ($a=2$):}
		Consider the base case $a=2$. In Lemma \ref{a=2,b=d+2}, we proved that $U_{comp}^{d-1}(K_{2, d+2})$ does not embed in $\mathbb{R}^d$ for $d\geq 3$; and it is ready to verify that  $U_{comp}^{2}(K_{3, 4})$ does not embed in $\mathbb{R}^3$. The base case holds.
		
		%	\textbf{Inductive Step:}
		Assume that for any integer $2\leq a' < a$ and $3\leq d'< d$, the complex $U_{comp}^{d'-1}(K_{a', d'-a'+4})$ cannot embed into $\mathbb{R}^{d'}$.
		Consider the case for $a$ and $d$. Suppose, for the sake of contradiction, that there exists an embedding:
		\[ 
		f: U_{comp}^{d-1}(K_{a, d-a+4}) \rightarrow \mathbb{R}^d. 
		\]
		
		Let $v \in A$ be a vertex. If the complex embeds in $\mathbb{R}^d$, then the link $\text{Lk}(v)$ must embed in the sphere $S^{d-1}$, and consequently into $\mathbb{R}^{d-1}$. 
		%(since the link is not the entire sphere).
		
		By Lemma \ref{lemma:reduction}, we have:
		\[ 
		\text{Lk}(v) \cong U_{comp}^{(d-1)-1}(K_{a-1, (d-1)-(a-1)+4}). 
		\]
		
		%	Let us check the parameters for this new complex $\text{Lk}(v)$ relative to the target space $\mathbb{R}^{d-1}$:
		%	\begin{itemize}
			%		\item New partition size: $a' = a-1$.
			%		\item New target dimension: $D = d-1$.
			%		\item Partition size: $b = d-a+4$.
			%	\end{itemize}
		%	We verify if $b$ satisfies the non-embeddability condition for the new parameters $(a', D)$:
		%	\[ D - a' + 4 = (d-1) - (a-1) + 4 = d - a + 4 = b \]
		%	The parameters match exactly. 
		
		Thus, the embedding of $\text{Lk}(v)$ into $\mathbb{R}^{d-1}$ would contradict the inductive hypothesis.
		
		Therefore, $U_{comp}^{d-1}(K_{a, d-a+4})$ cannot embed in $\mathbb{R}^d$.
	\end{proof}

	%We get the following theorem by Corollary \ref{a+b-3=d} and Lemma \ref{a+b-4=d}. 
	
	%\begin{corollary}\label{erbutujixiao}
	%	$U = U_{comp}^{d-1}(K_{a, d-a+4})$ is minimally non-embeddable in $\mathbb{R}^d$. That is, if any $(d-1)$-face is removed from $U$, the resulting $U'$ embeds into $\mathbb{R}^d$. 
	%\end{corollary}

	%\begin{theorem}\label{min-emb}
	%	Let $K_{a,b}$ be a complete bipartite graph with $a\geq 2$, $b\geq 5$ and $d = a + b - 4$. Then $U^{a+b-5}_{comp}(K_{a,b})$ cannot be embeded in $\mathbb{R}^d$. 
	%	Moreover, for any $i$-face $f$ of $U^{a+b-5}_{comp}(K_{a,b})$ with $i \geq 2$, 
	%	$U^{a+b-5}_{comp}(K_{a,b}) \setminus f$ can be embeded in $\mathbb{R}^d$. 
	%\end{theorem}

	\subsection{Proof of Theorem \ref{anti-minor} (necessity)}\label{section9}
	
	%We use the Jordan-Brouwer Separation Theorem to prove Lemma \ref{comp}. 
	
	We get the following lemma by Lemma \ref{lemma:flo}. 
	
	\begin{lemma}\label{comp}
		$U^{d-1}(K_{d+3})$ is a non-$\mathbb{R}^d$-hypergraph. 
	\end{lemma}

	We observe that, except for complete graphs, every graph listed in Table~\ref{minor} contains a bipartite graph $K_{a,d-a+4}$ as a subgraph. 
	Therefore, by Theorem~\ref{min-emb} and Lemma~\ref{comp}, we obtain the following lemma.

	\begin{lemma}\label{bipa}
		If $G$ is a graph in Table \ref{minor}, then $U^{d-1}(G)$ is a non-$\mathbb{R}^d$-hypergraph. 
	\end{lemma}

	%\begin{lemma}\label{bipa}
	%	$U^{d-1}(K_{3,d+1})$ is a non-$\mathbb{R}^d$-hypergraph. 
	%\end{lemma}
	%
	%\begin{proof}
	%	The proof of Lemma \ref{bipa} is similar to that of Lemma \ref{comp}. 
	%	(Figure \ref{new8} is an example of $K_{3,4}^3$ in $\mathbb{R}^3$. )
	%	
	%	\begin{figure}
		%		\centering    
		%		\includegraphics[width=0.75\linewidth]{new8}
		%		\caption{$U^2(K_{3,4})$ in $\mathbb{R}^3$.}
		%		\label{new8}
		%	\end{figure}
	%	
	%	%		\begin{figure}
		%		%			\centering    
		%		%			\includegraphics[width=0.5\linewidth]{new8}
		%		%			\caption{$K_{3,d+1}^d$ in $\mathbb{R}^d$.}
		%		%			\label{new8}
		%		%		\end{figure}
	%\end{proof}
	
	\begin{definition}[{\it anti-$d$-dimension minor}]
		If a graph $G$ has a minor in Table \ref{minor}, then we call $G$ an {\it anti-$d$-dimension minor}. 
	\end{definition}

	\subsection{Proof of Theorem \ref{anti-minor} (sufficiency)}\label{last}
	
	\begin{proof}
		%kkk
		In view of Lemmas \ref{10.33} and \ref{10.34}, it suffices to prove Theorem \ref{anti-minor} for $(d-1)$-dimensional topological hypergraph which is $2$-sphere connected. 
		
		Note that $U = U^{d-1}(G)$ is a $(d-1)$-dimensional topological hypergraph. 
		Now we assume that $U$ is a non-$\mathbb{R}^d$-hypergraph, $U$ is $2$-sphere connected, and $U$ is simple. Because all hypergraphs on $(d+2)$ or fewer vertices can be embedded in $\mathbb{R}^d$, we have $|V(U)|\geq d+3$. We proceed by induction on $|V(U)|$. By Lemma \ref{connected}, $G$ contains an $1$-dimensional topological hyperedge $e = xy$ such that $H= U/e$ is $2$-sphere connected. 
		If $H$ is a non-$\mathbb{R}^d$-hypergraph, the $1$-skeleton of $H$ contains a minor in Table \ref{minor}, by induction. Since every minor of $H$ is also a minor of $U$, we deduce that the $1$-skeleton of $U$ contains a minor in Table \ref{minor}, too. So we may assume that $H$ is an $\mathbb{R}^d$-hypergraph. 
		
		Consider an $\mathbb{R}^d$-embedding $H'$ of $H$. Denote by $z$ the vertex of $H$ formed by contracting $e$. Because $H$ is $2$-sphere connected, by Lemma \ref{sphere}, the neighbors of $z$ lie on a $(d-1)$-sphere $S^{d-1}$, the boundary of some polytope $W$ of $H'- z$. Denote by $B_x$ and $B_y$ the bridges of $W$ in $U \backslash e$ that contain the vertices $x$ and $y$, respectively. 
		
		%Suppose, first, that $B_x$ and $B_y$ avoid each other. In this case, $B_x$ and $B_y$ can be embedded in the polytope $W$ of $H'- z$ in such a way that the vertices $x$ and $y$ belong to the same polytope of the resulting $\mathbb{R}^d$-hypergraph $(H'-z) \cup B_x \cup B_y$. The edge $xy$ and the topological hyperedges that incident with $xy$ can now be drawn in that polytope so as to obtain an $\mathbb{R}^d$-embedding of $G$ itself, contradicting the hypothesis that $G$ is a non-$\mathbb{R}^d$-hypergraph. 

		\begin{figure}
			\centering     
			\begin{tikzpicture}[scale=0.9]
				% ================= 定义坐标点 =================
				\coordinate (x2) at (0, 4);   % 北极
				\coordinate (x1) at (0, -4);  % 南极
				\coordinate (y1) at (-4, 0);  % 左赤道点
				\coordinate (y2) at (4, 0);   % 右赤道点
				
				\coordinate (x) at (-1.1, 2.5);     % x
				\coordinate (yp1) at (-1.8, -1.2);  % y'_1
				\coordinate (yp2) at (1.8, 0.8);    % y'_2
				
				% ================= 1. 绘制红色虚线 (内部构造) =================
				\draw[red, dashed, thick] (y1) -- (y2) node[midway, left=-0.3, above=0.2, text=black] {$B^{d-2}$};
				\draw[red, dashed, thick] (x) -- (x1) node[midway, above=2cm, text=black] {$e_x$};
				\draw[red, dashed, thick] (yp1) -- (yp2) node[midway, left=-2.2, text=black] {$B'$};
				
				% ================= 2. 绘制黑色虚线 (背景/后方曲线) =================
				% 原有的虚线
				\draw[dashed, thick] (y1) to[bend left=20] (x);
				\draw[dashed, thick] (x) to[bend left=20] (y2);
				\draw[dashed, thick] (x) to[bend right=10] (yp2);
				\draw[dashed, thick] (x) to[bend right=15] (yp1);
				\draw[dashed, thick] (y2) to[bend right=15] (yp2);
				\draw[dashed, thick] (yp2) to[bend left=10] (x1);
				
				% 【新增】 y1 到 y'2 的虚线 (对角线)
				\draw[dashed, thick] (y1) to[bend left=12] (yp2);
				
				% ================= 3. 绘制实线 (前景/轮廓) =================
				% 外圆
				\draw[thick] (0,0) circle (4cm);
				
				% 原有的实线
				\draw[thick] (y1) to[bend right=15] (yp1);
				\draw[thick] (yp1) -- (x1); 
				
				% 【新增】 y'1 到 y2 的实线 (对角线，前景)
				\draw[thick] (yp1) to[bend right=12] (y2);
				
				% ================= 4. 绘制顶点 (黑点) =================
				\foreach \p in {x, x1, x2, y1, y2, yp1, yp2} {
					\filldraw[black] (\p) circle (4pt);
				}
				
				% ================= 5. 添加标签 =================
				\node[above=0.2cm] at (x2) {\Large $x_2$};
				\node[below=0.2cm] at (x1) {\Large $x_1$};
				\node[left=0.2cm] at (y1) {\Large $y_1$};
				\node[right=0.2cm] at (y2) {\Large $y_2$};
				\node[above=0.2cm] at (x) {\Large $x$};
				\node[below left=0.1cm] at (yp1) {\Large $y'_1$};
				\node[above right=0.1cm] at (yp2) {\Large $y'_2$};
				
				\node[above right] at (3, 2.5) {\Large $S_2^{d-1}$};
				\node[below right] at (3.5, -2) {\Large $S_1^{d-1}$};
				
			\end{tikzpicture}
			\caption{$e = xx_1$ is not an edge of $G$, and $d=3$.}
			\label{erwei}
		\end{figure}

		Note that $B_x$ and $B_y$ cannot avoid each other. It follows that $B_x$ and $B_y$ overlap. By Lemma \ref{bridges}, they are therefore either skew or else equivalent $(d+1)$-bridges. In the latter case, it easy to verify that $G$ has a $K_{d+3}$-minor; In the former case, we need to discuss the following subcases. 
		
		Let a $(d-1)$-ball $B' \subset B_y$ divide $S^{d-1}$ into two spheres, denoted $S^{d-1}_1$ and $S^{d-1}_2$. Without loss of generality, assume that the vertex $x$ lies in the interior of $S^{d-1}_2$. Since $B_x$ and $B_y$ are skew, there exists a $(d-1)$-dimensional topological hyperedge $e_x \subset B_x$ that contains both $x$ and $x_1$ (see Figure~\ref{erwei} for an example with $d=3$); and there exists a $(d-1)$-dimensional topological hyperedge $e_x' \subset B_x$ that contains both $x$ and $x_2$. In this case, $e_x$ must intersect $B'$.  		
		%	\begin{figure}
			%		\centering     
			%		\includegraphics[width=0.4\linewidth]{new13}
			%		\caption{$K_{3,4}$-minor.}
			%		\label{new13}
			%	\end{figure}

		We proceed to prove that $S^{d-1}$ contains at least $d+2$ vertices.
		
		First, the intersection $B_y \cap S^{d-1}$ must contain at least $d$ vertices (If the number of vertices in the intersection were strictly less than $d$, it would result in multiple topological hyperedges). Then, there must exist vertices $x_1 \in S_1^{d-1}$ and $x_2 \in S_2^{d-1}$ such that neither $x_1$ nor $x_2$ belongs to $B_y \cap S^{d-1}$. Combining these observations, the total number of vertices on $S^{d-1}$ is at least $d + 2$.
		
		For the case $d=2$, the edges $xx_1$ and $xx_2$ necessarily exist. This immediately implies that $G$ contains a $K_{3,3}$-minor. However, for $d \geq 3$, the existence of the edges $xx_1$ and $xx_2$ is not guaranteed.
		
		We now provide a counterexample where the edge $xx_1$ is absent. As shown in Figure \ref{erwei} for the case $d=3$, the edge $xx_1$ does not exist; instead, there exists a $2$-face $e_x = x y_1' x_1 y_2'$. It is straightforward to verify that $e_x$ intersects $B' \cup S^{d-1}$ by Theorem \ref{jordan-Brouwer}. This confirms that for $d \geq 3$, the edges $xx_1$ and $xx_2$ need not exist.
		
		Having established that $S^{d-1}$ contains at least $d+2$ vertices, we may assume without loss of generality that $|V(S^{d-1})| = d+2$ (any additional vertices can be removed via contraction). Additionally, considering the two internal vertices $x$ and $y$, a minimal skew must contain $d+4$ vertices. Let $M$ denote this set of vertices. 
		%	\[
		%	M = \{x, x_1, x_2, y, y_1, y_2, \dots, y_d\}.
		%	\]
		We now proceed to analyze the independent sets within $M$.
		
		\medskip
		\noindent \textbf{Case 1.} Every independent set in $M$ consists of a single vertex. In this case, $G$ contains a $K_{d+3}$-minor.
		
		\medskip
		\noindent \textbf{Case 2.} There exists an independent set $A = \{a_1, a_2\} \subset M$. Let $B = M \setminus A = \{b_1, b_2, \dots, b_{d+2}\}$.

		By Theorem~\ref{min-emb}, we know that if a $(d-1)$-dimensional topological hypergraph on $d+4$ vertices cannot embedded into $\mathbb{R}^d$, then every chordless $i$-sphere must be present. However, in our procedure for increasing the dimension of the graph $G$, we fill each induced $i$-sphere with an $(i+1)$-ball. In other words, in order to ensure that the $(d-1)$-dimensional topological hypergraph formed by these $d+4$ vertices cannot embedded into $\mathbb{R}^d$, every chordless $i$-sphere appearing in Theorem~\ref{min-emb} must in fact be an induced $i$-sphere. Consequently, the vertices in the set $B = \{b_1, b_2, \dots, b_{d+2}\}$ must satisfy the following condition:
		
		\begin{itemize}
			\item For any subset $B' \subseteq B$ with $|B'| = d-1$, if the induced subgraph $G[B']$ is not the complete graph $K_{d-1}$, then the induced subgraph $G[V(G)\setminus B']$ is connected.
		\end{itemize}
		
		By Lemma~\ref{jitu}, we obtain the following classification. Note that when $|A| = 2$, \textbf{Case III} in the proof of Lemma~\ref{jitu} cannot occur (otherwise, a $2$-vertex cut would arise; however, by the definition of the marked $S$-decomposition, all $2$-cuts are eliminated during this process). Thus we only consider the remaining \textbf{Case I} and \textbf{Case II}: 
		\begin{itemize}
			\item If $d = 4k$, then $G \cong K_{2 \otimes 2(k+1)}$ or $G \cong \overline{K_2} \wedge \overline{K_4} \wedge K_{4k-2}$;
			\item If $d = 4k+1$, then $G \cong K_{1, 2 \otimes 2(k+1)}$ or $G \cong \overline{K_2} \wedge \overline{K_4} \wedge K_{4k-1}$;
			\item If $d = 4k+2$, then $G \cong K_{2 \otimes (2k+3)}$ or $G \cong \overline{K_2} \wedge \overline{K_4} \wedge K_{4k}$;
			\item If $d = 4k+3$, then $G \cong K_{1, 2 \otimes (2k+3)}$ or $G \cong \overline{K_2} \wedge \overline{K_4} \wedge K_{4k+1}$.
		\end{itemize}

		By Lemma~\ref{jitu}, we proceed to analyze the size of the independent set $A$ case by case; suppose this procedure has reached \textbf{Case $j$} (namely, there exists an independent set $A = \{a_1, a_2, \dots, a_j\} \subset M$).
		
		It should be noted that when the number of vertices in either $A$ or $B=M\setminus A$ exceeds $4$, the set $C \in \{A,B\}$ is required to satisfy the following condition:
		\begin{itemize}
			\item For any subset $C' \subseteq C$ with $|C'| = |C|-3$, if the induced subgraph $G[C']$ is not the complete graph $K_{|C|-3}$, then the induced subgraph $G[V(G)\setminus C']$ is connected.
		\end{itemize}
		
		\medskip
		
		Following the approaches of solving \textbf{Case 1} and \textbf{Case 2}, we classify the graph according to the size of the independent set $A$.  
		For each $j \in \{1, 2, \dots, \lfloor \tfrac{d+4}{2} \rfloor \}$, when $|A| = j$, we obtain the following cases.
		
		%		There exists an independent set $A=\{a_1,a_2,\dots,a_j\}\subset M$.
		%		Let $B=M\setminus A=\{b_1,b_2,\dots,b_{d+4-j}\}$.
		
		\noindent \textbf{Case $j$.}
		According to Lemma~\ref{jitu}, it is worth noting that, now for the sets $A$ and $B$,  each admits three possible cases. 
		In principle, this yields nine possible graph classes. 
		However, three of them are redundant and should be excluded. 
		Consequently, only six distinct graph classes arise.
		Therefore, we obtain the following classification.
		
		\begin{itemize}
			\item \textbf{If $d=4k$.}
			\begin{itemize}
				\item If $j=2i$, then $G$ contains one of the following four graphs as minor:
				\[
				K_{2\otimes i}\wedge K_{2\otimes(2k+2-i)},\quad
				K_{2\otimes i}\wedge \overline{K_4}\wedge K_{4k-2i},
				\]
				\[
				\overline{K_4}\wedge K_{2i-4}\wedge \overline{K_4}\wedge K_{4k-2i},\quad
				\overline{K_4}\wedge K_{2i-4}\wedge (K_1\cup K_{4k+3-2i}).
				\]
				\[
				(K_1\cup K_{2i-1})\wedge (K_1\cup K_{4k+3-2i}),\quad
				(K_1\cup K_{2i-1})\wedge K_{2\otimes(2k+2-i)}.
				\]
				\item If $j=2i+1$, then $G$ is isomorphic to one of the following four graphs:
				\[
				K_{1,2\otimes i}\wedge K_{1,2\otimes(2k+1-i)},\quad
				K_{1,2\otimes i}\wedge \overline{K_4}\wedge K_{4k-1-2i},
				\]
				\[
				\overline{K_4}\wedge K_{2i-3}\wedge \overline{K_4}\wedge K_{4k-1-2i},\quad
				\overline{K_4}\wedge K_{2i-3}\wedge (K_1\cup K_{4k+2-2i}).
				\]
				\[
				(K_1\cup K_{2i})\wedge (K_1\cup K_{4k+2-2i}),\quad
				(K_1\cup K_{2i})\wedge K_{1,2\otimes(2k+1-i)}.
				\]
			\end{itemize}
			
			\item \textbf{If $d=4k+1$.}
			\begin{itemize}
				\item If $j=2i$, then $G$ contains one of the following four graphs as minor:
				\[
				K_{2\otimes i}\wedge K_{1,2\otimes(2k+2-i)},\quad
				K_{2\otimes i}\wedge \overline{K_4}\wedge K_{4k+1-2i},
				\]
				\[
				\overline{K_4}\wedge K_{2i-4}\wedge \overline{K_4}\wedge K_{4k+1-2i},\quad
				\overline{K_4}\wedge K_{2i-4}\wedge (K_1\cup K_{4k+4-2i}).
				\]
				\[
				(K_1\cup K_{2i-1})\wedge (K_1\cup K_{4k+4-2i}),\quad
				(K_1\cup K_{2i-1})\wedge K_{1,2\otimes(2k+2-i)}.
				\]
				\item If $j=2i+1$, then $G$ is isomorphic to one of the following four graphs:
				\[
				K_{1,2\otimes i}\wedge K_{2\otimes(2k+2-i)},\quad
				K_{1,2\otimes i}\wedge \overline{K_4}\wedge K_{4k-2i},
				\]
				\[
				\overline{K_4}\wedge K_{2i-3}\wedge \overline{K_4}\wedge K_{4k-2i},\quad
				\overline{K_4}\wedge K_{2i-3}\wedge (K_1\cup K_{4k+3-2i}).
				\]
				\[
				(K_1\cup K_{2i})\wedge (K_1\cup K_{4k+3-2i}),\quad
				(K_1\cup K_{2i})\wedge K_{2\otimes(2k+2-i)}.
				\]
			\end{itemize}
			
			\item \textbf{If $d=4k+2$.}
			\begin{itemize}
				\item If $j=2i$, then $G$ contains one of the following four graphs as minor:
				\[
				K_{2\otimes i}\wedge K_{2\otimes(2k+3-i)},\quad
				K_{2\otimes i}\wedge \overline{K_4}\wedge K_{4k+2-2i},
				\]
				\[
				\overline{K_4}\wedge K_{2i-4}\wedge \overline{K_4}\wedge K_{4k+2-2i},\quad
				\overline{K_4}\wedge K_{2i-4}\wedge (K_1\cup K_{4k+5-2i}).
				\]
				\[
				(K_1\cup K_{2i-1})\wedge (K_1\cup K_{4k+5-2i}),\quad
				(K_1\cup K_{2i-1})\wedge K_{2\otimes(2k+3-i)}.
				\]
				\item If $j=2i+1$, then $G$ is isomorphic to one of the following four graphs:
				\[
				K_{1,2\otimes i}\wedge K_{1,2\otimes(2k+2-i)},\quad
				K_{1,2\otimes i}\wedge \overline{K_4}\wedge K_{4k+1-2i},
				\]
				\[
				\overline{K_4}\wedge K_{2i-3}\wedge \overline{K_4}\wedge K_{4k+1-2i},\quad
				\overline{K_4}\wedge K_{2i-3}\wedge (K_1\cup K_{4k+4-2i}).
				\]
				\[
				(K_1\cup K_{2i})\wedge (K_1\cup K_{4k+4-2i}),\quad
				(K_1\cup K_{2i})\wedge K_{1,2\otimes(2k+2-i)}.
				\]
			\end{itemize}
			
			\item \textbf{If $d=4k+3$.}
			\begin{itemize}
				\item If $j=2i$, then $G$ contains one of the following four graphs as minor:
				\[
				K_{2\otimes i}\wedge K_{1,2\otimes(2k+3-i)},\quad
				K_{2\otimes i}\wedge \overline{K_4}\wedge K_{4k+3-2i},
				\]
				\[
				\overline{K_4}\wedge K_{2i-4}\wedge \overline{K_4}\wedge K_{4k+3-2i},\quad
				\overline{K_4}\wedge K_{2i-4}\wedge (K_1\cup K_{4k+6-2i}).
				\]
				\[
				(K_1\cup K_{2i-1})\wedge (K_1\cup K_{4k+6-2i}),\quad
				(K_1\cup K_{2i-1})\wedge K_{1,2\otimes(2k+3-i)}.
				\]
				\item If $j=2i+1$, then $G$ is isomorphic to one of the following four graphs:
				\[
				K_{1,2\otimes i}\wedge K_{2\otimes(2k+3-i)},\quad
				K_{1,2\otimes i}\wedge \overline{K_4}\wedge K_{4k+2-2i},
				\]
				\[
				\overline{K_4}\wedge K_{2i-3}\wedge \overline{K_4}\wedge K_{4k+2-2i},\quad
				\overline{K_4}\wedge K_{2i-3}\wedge (K_1\cup K_{4k+5-2i}).
				\]
				\[
				(K_1\cup K_{2i})\wedge (K_1\cup K_{4k+5-2i}),\quad
				(K_1\cup K_{2i})\wedge K_{2\otimes(2k+3-i)}.
				\]
			\end{itemize}
		\end{itemize}
		
		Although the above classes of graphs may appear complicated, it suffices to identify the minimal ones among them. 
		%In summary, by exhausting all of the cases listed above and eliminating redundancies, we conclude that $G$ must contain one of the graphs in Table~\ref{minor} as a minor. 
		
		It should be noted that when the dimension $d$ is small, certain degenerate cases may occur. 
		We list these degenerate situations below.
		
		When $d=3$, the above procedure yields four minors, namely
		\[
		K_6,\quad K_{3,4},\quad K_{1,2\otimes 2},\quad \overline{K_2} \wedge (K_1 \cup K_4).
		\]
		We observe that $K_{3,4} \subseteq K_{1,2\otimes 2}$, and that the graph $\overline{K_2} \wedge (K_1 \cup K_4)$ contains a $2$-cut. 
		Consequently, when $d=3$, the only minors satisfying the required conditions are $K_6$ and $K_{3,4}$.
		
		Similarly, when $d=4$, the minors satisfying the conditions are
		\[
		K_7,\quad K_{4,4},\quad \overline{K_3} \wedge (K_1 \cup K_4).
		\]
		
		In summary, for $d \geq 5$, by exhausting all the cases listed above and eliminating redundancies, 
		we conclude that $G$ must contain one of the graphs in Table~\ref{minor} as a minor.

	\end{proof}

	\subsection{Proof of Lemma \ref{jitu}}
	
	If a complete $p$-partite graph has exactly $q$ vertices in each part, we denote it by $K_{q \otimes p}$. For example, the graph $K_{2,2,2,2}$ can be written as $K_{2 \otimes 4}$, and $K_{1,2,2,2}$ can be written as $K_{1,2 \otimes 3}$.

	\begin{lemma}\label{jitu}
		Let $G$ be a graph with $|V(G)| = n \geq 5$. Suppose that for every subset $V' \subseteq V(G)$ with $|V'| = n-3$, if $G - V'$ is disconnected, then $G[V'] \cong K_{n-3}$. 
		If the graph $G$ is minimal with respect to the above conditions (that is, removing any edge from $G$ causes the conditions to fail), then the following conclusions hold:
		\begin{itemize}
			%			\item If $n = 2k$, then $G$ is isomorphic to either $K_{2 \otimes \frac{n}{2}}$ or $\overline{K_4} \wedge K_{n-4}$. 
			%			\item If $n = 2k+1$, then $G$ is isomorphic to either $K_{1, 2 \otimes \frac{n-1}{2}}$ or $\overline{K_4} \wedge K_{n-4}$.
			\item If $n = 2k$, then $G$ is isomorphic to one of the following three classes of graphs:
			\[
			K_{2 \otimes \frac{n}{2}}, \qquad \overline{K_4} \wedge K_{n-4}, \qquad \text{or } K_1 \cup K_{n-1};
			\]
			
			\item If $n = 2k+1$, then $G$ is isomorphic to one of the following three classes of graphs:
			\[
			K_{1, 2 \otimes \frac{n-1}{2}}, \qquad \overline{K_4} \wedge K_{n-4}, \qquad \text{or } K_1 \cup K_{n-1}.
			\]
		\end{itemize}
	\end{lemma}
	
	%	\begin{lemma}\label{jitu}
		%		Let $G$ be a connected graph with $|V(G)| = n \geq 5$. Suppose that for every subset $V'$ satisfying the following three conditions, the induced subgraph $G[V(G)\setminus V']$ is connected: 
		%		\begin{itemize}
			%			\item $|V'| = n-3$;
			%			\item $V' \subseteq V(G)$;
			%			\item $G[V']$ is not the complete graph $K_{n-3}$.
			%		\end{itemize}
		%		If the graph $G$ is minimal with respect to the above conditions (that is, removing any edge from $G$ causes the conditions to fail), then the following conclusions hold:
		%		\begin{itemize}
			%			\item If $n = 2k$, then $G$ is isomorphic to either $K_{2 \otimes \frac{n}{2}}$ or $\overline{K_4} \wedge K_{n-4}$.
			%			\item If $n = 2k+1$, then $G$ is isomorphic to either $K_{1, 2 \otimes \frac{n-1}{2}}$ or $\overline{K_4} \wedge K_{n-4}$.
			%		\end{itemize}
		%	\end{lemma}
	
	\begin{proof}
		
		\noindent \textbf{Case I.}
		$G$ is a connected graph. 
		For any three vertices of $G$, the induced subgraph on these vertices is connected.
		It is straightforward to see that, in this situation, $G$ is isomorphic to a complete
		graph with a maximum matching removed. More precisely, when $n = 2k$, the graph $G$
		is isomorphic to $K_{2 \otimes \frac{n}{2}}$, and when $n = 2k+1$, the graph $G$ is
		isomorphic to $K_{1, 2 \otimes \frac{n-1}{2}}$. 
		It is straightforward to verify that deleting any edge from a graph in this class causes it to violate the conditions of the theorem. Hence, graphs of this type are minimal.

		\noindent \textbf{Case II.}
		$G$ is a connected graph. 
		There exist three vertices $\{u, v, w\}$ in $G$ such that the induced subgraph $G[\{u, v, w\}]$ is disconnected. 
		
		Without loss of generality, we may assume that there are no edges among the three vertices $\{u,v,w\}$. 
		(The case where some edges are present can be treated similarly, and one can ultimately show that such edges may be deleted.)
		
		In this situation, the induced subgraph $G[V(G)\setminus\{u,v,w\}] \cong K_{n-3}$. 
		If, for every $a' \in \{u,v,w\}$, the vertex $a'$ is adjacent to all vertices in $V(G)\setminus\{u,v,w\}$, then the graph $G$ clearly satisfies the required conditions.
		
		To obtain a graph with the minimum number of edges, we may delete an edge $up$ from the above graph $G$. 
		It is easy to verify that $G\setminus up$ still satisfies the conditions.
		
		If we further delete another edge incident with $u$, say $uq$, then the resulting graph $G' = G \setminus \{up, uq\}$ satisfy that $G' \setminus u$ is a complete graph. 
		This case can be reduced to \textbf{Case III}, and hence we may only delete edges that are not incident with $u$.
		
		Next, consider deleting an edge $vr$ incident with $v$. 
		If $r=p$, then the graph $G' = G \setminus \{up, vr\}$
		still satisfies the required conditions. 
		However, if $r \neq p$, then $G' = G \setminus \{up, vr\}$ does not satisfy the conditions, since graph $G[\{u,w,p\}]$ is disconnected and graph $G' \setminus \{u,w,p\}$ is not complete.
		
		In summary, at most three edges, namely $up$, $vp$, and $wp$, can be deleted from $G$. 
		It is easy to see that the resulting graph is isomorphic to $\overline{K_4} \wedge K_{n-4}$.
		One can verify that the deletion of any edge from $\overline{K_4} \wedge K_{n-4}$ yields a graph that fails to satisfy the conditions of the lemma.

		\noindent \textbf{Case III.}
		$G$ is not a connected graph. It is easy to verify that $G$ is isomorphic to $K_1 \cup K_{n-1}$. 
		
		\medskip
		\begin{remark}
			When $n = 5$, a degenerate situation arises: in this case,
			$\overline{K_4} \wedge K_{1}$ is a proper subgraph of $K_{1,2,2}$, and hence one can
			directly conclude that $G$ must be isomorphic to $\overline{K_4} \wedge K_{1}$.
			When $n \ge 6$, the graph $G$ falls into one of the three cases described above. 
		\end{remark}
	\end{proof}

	\subsection{Proof of Theorem \ref{chromatic-number}}\label{last2}
	
	At present, we have not found an effective method to further improve the upper bound for the chromatic number. 
	Here we only derive a rough upper bound based on the average degree. 
	Theorem \ref{chromatic-number} is the corollary of Theorem \ref{chromatic-number1} and Corollary \ref{nmv}.

	\begin{theorem}\label{chromatic-number1}
		Let $U^{d-1}(G)$ be an $R^d$-hypergraph, $G$ be the $1$-skeleton of $U^{d-1}(G)$, and $E(G)$ be the edge set of $G$, then $|E(G)| \leq (3\cdot 2^{d-2})|V(G)| - 3\cdot 2^{d-2}\cdot(d+1) + \frac{d(d+1)}{2}$ if $|V(G)|\geq d+1$. 
		
	\end{theorem}
	
	\begin{proof}
		Since the purpose of this theorem is to derive an upper bound for the number of edges, 
		we may assume that the closure of each connected open set in 
		$\mathbb{R}^d \setminus U^{d-1}(G)$ is homeomorphic to a $d$-simplex.
		
		By induction on $d$ (denoted by {\bf Induction $\alpha$}). The theorem holds trivially for $d=2$. Assume that it
		holds for all integers less tha $d$, and let $G$ be the $1$-skeleton of  an $R^{d-1}$-hypergraph with 
		$$|E(G)| \leq (3\cdot 2^{d-3})|V(G)| - 3\cdot 2^{d-3}\cdot d + \frac{d(d-1)}{2},$$
		then 
		$$\frac{2|E(G)|}{|V(G)|} \leq 3\cdot 2^{d-2}-\frac{3\cdot 2^{d-3}\cdot d - \frac{d(d-1)}{2}}{|V(G)|}< 3\cdot 2^{d-2},$$
		which implies that there exists a vertex $v\in V(G)$ such that $d_{G}(v)\leq 3\cdot 2^{d-2}-1$. 
		
		Now we prove that the theorem holds for $d$. 
		
		\begin{claim}
			Theorem \ref{chromatic-number1} holds for $d$. 
		\end{claim} 
		
		\begin{proof}
			By induction on $|V(G)|$ (denoted by {\bf Induction $\beta$}). The theorem holds trivially for $|V(G)|=d+1$. Let $m$ be an integer with $m>d+1$. Assume that the theorem holds for $|V(G)|< m$. 
			
			Now we prove the theorem holds for $|V(G)|= m$. 
			
			$\forall u \in V(G)$, let $d_{G}(u)$ be the number of vertices which are adjacent to $u$ and $G_1$ be the $1$-skeleton of $U_1 = U^{d-1}(G)\backslash u$ with $V(G_1)=V(G)\backslash \{u\}$. 
			Observe that after removing $u$, there exists an open set $W$ in $\mathbb{R}^d \setminus U_1$ such that, prior to the deletion of $u$, $u$ lies in the interior of $W$. In this situation, the boundary $\partial W$ forms an $\mathbb{R}^{d-1}$-hypergraph, which we denote by $G_0$. 
			
			%It is obvious that there is only one open set (denoted by $W$) in $\mathbb{R}^d \setminus U_1$ is not a $d$-simplex, and $\partial W$ is an $R^{d-1}$-hypergraph (denoted by $G_0$). 
			
			Observe that $G_0$ is homeomorphic to $S_{d-1}$, and $|V(G_1)|\geq d+1$ since $|V(G)|=m > d+1$, then there exists a vertex $v\in V(G_0)$ such that $d_{G_0}(v)\leq 3\cdot 2^{d-2}-1$ by the induction hypothesis ({\bf Induction $\alpha$}). 
			
			We triangulate $U_1$ into $U_1'$, and let $G_1'$ be the $1$-skeleton of $U_1'$. 
			At this stage, $W$ is the only region that is not a $d$-simplex. Thus it suffices to triangulate $W$ into a collection of $d$-simplices. 
			Without loss of generality, we may assume that the triangulation is obtained by projecting from $v$ to all remaining vertices, edges, and faces. 
			This yields the following equality. 
			\begin{align*}
				|E(G)|-d_{G}(u) &= |E(G_1)|=|E(G_1')|-(d_{G}(u)-d_{G_0}(v)-1),\\
				|E(G_1')|&=|E(G)|-d_{G_0}(v)-1.
			\end{align*}
			
			Observe that $U_1'$ is also an $R^d$-hypergraph, then $|E(G_1')| \leq 3\cdot 2^{d-2}|V(G_1')| - 3\cdot 2^{d-2}\cdot(d+1) + \frac{d(d+1)}{2}$ by the induction hypothesis ({\bf Induction $\beta$}). 
			
			%		Note that the triangulation is a $v$-triangulation, then $$|E(G)|-d_{G}(u) = |E(G_1)|=|E(G_1')|-(d_{G}(u)-d_{G_0}(v)-1),$$
			%		$$|E(G_1')|=|E(G)|-d_{G_0}(v)-1.$$
			
			On the other hand, $|E(G_1')|=|E(G)|-d_{G_0}(v)-1 \geq |E(G)| -3\cdot 2^{d-2}$ since $d_{G_0}(v)\leq 3\cdot 2^{d-2}-1$, and we also know that $|V(G_1)|=|V(G)|-1$. 
			Therefore, 
			\begin{align*}
				|E(G)|-3\cdot 2^{d-2}&\leq |E(G_1')| \leq  3\cdot 2^{d-2}(|V(G)|-1) - 3\cdot 2^{d-2}\cdot(d+1) + \frac{d(d+1)}{2}, \\
				|E(G)| &\leq 3\cdot 2^{d-2}|V(G)| - 3\cdot 2^{d-2}\cdot(d+1) + \frac{d(d+1)}{2},
			\end{align*}
			and the claim follows by induction. 
			
		\end{proof}
		
		We infer that Theorem~\ref{chromatic-number1} holds for $d$, and Theorem~\ref{chromatic-number} follows by induction. 
		
	\end{proof}

	\begin{corollary}\label{nmv}
		Let $U^{d-1}(G)$ be an $R^d$-hypergraph and $G$ be the $1$-skeleton of $U^{d-1}(G)$, then there exists a $(3\cdot 2^{d-1}-1)^-$-vertex in $G$, where an $i^-$-vertex $v$ represent that $d(v)\leq i$.
		
	\end{corollary}
	
	\begin{proof}
		By Theorem \ref{chromatic-number1}, we have
		$$\frac{2|E(G)|}{|V(G)|} \leq 3\cdot 2^{d-1}-\frac{3\cdot 2^{d-2}\cdot (d+1) - \frac{d(d+1)}{2}}{|V(G)|}< 3\cdot 2^{d-1},$$
		which implies that there exists a $(3\cdot 2^{d-1}-1)^-$-vertex in $G$. 
		
	\end{proof}

	\section{Discharging in $\mathbb{R}^d$}\label{discharging}
	
	In graph theory, the Discharging method is a powerful proof technique primarily used in structural and coloring problems, especially for planar graphs. It gained prominence through its central role in the proof of the Four Color Theorem. 
	Based on the theoretical framework proposed in this paper, we can extend the Discharging method, originally used to address problems on planar graphs, to $\mathbb{R}^d$.
	Although we are currently unable to use the Discharging method to improve Theorem~\ref{chromatic-number}, we have found that this approach can be applied to the coloring problem of $(d-2)$-faces in $d$-uniform $\mathbb{R}^d$-hypergraphs. For this purpose, we first introduce the following definitions.
	
	%We now demonstrate how the Discharging method can be applied to address the coloring problem of $d$-uniform $\mathbb{R}^d$-hypergraph. For this purpose, we first introduce the following definitions.
	
	%\begin{definition}[$d$-uniform $\mathbb{R}^d$-hypergraph]
	%	Let $H$ be an $\mathbb{R}^d$-hypergraph. If every $(d-1)$-dimensional topological hyperedge of $H$ is homeomorphic to a $(d-1)$-simplex (that is, a simplex of dimension $d-1$, in which each hyperedge contains exactly $d$ vertices), then $H$ is called a $d$-uniform $\mathbb{R}^d$-hypergraph.
	%\end{definition}
	
	\begin{definition}[$i$-dimensional $k$-coloring]
		Let $H$ be an $\mathbb{R}^d$-hypergraph, $A_i$ be the set of $i$-face of $G$ (Definition~\ref{def:Ai}),  		
		and $\phi : A_i \rightarrow \{1,2, \dots, k\}$ be a coloring. For each $a_i \in A_i$, denote by $c(a_i)$ the color assigned to $a_i$ under this mapping. If, for every $a_{i+1} \in A_{i+1}$, the set $\{c(a_i) \mid a_i \subseteq a_{i+1}\}$ contains at least two distinct colors, then $\phi$ is called an $i$-dimensional $k$-coloring of $H$. The smallest integer $k$ for which such a coloring exists is referred to as the $i$-dimensional chromatic number of $H$, denoted by $\chi^i_{\mathbb{R}^d}(H)$.
	\end{definition}

	At the beginning, it is necessary to demonstrate that the sum of the initial weights assigned to all topological hyperedges is less than or equal to zero. 
	
	\begin{lemma}\label{discharging-lemma}
		Let $G$ be an $R^d$-hypergraph, and $\{a,b,c\}$ be three integers such that $c=2\cdot a+d\cdot b$. 
		
		We give the {\bf initial weight assignment} of the $i$-dimensional topological hyperedges $a_i$ of $G$ for $a_i \in A_i$:
		\begin{align*}
			\omega(a_d) &= a\cdot d_{d-1}(a_d) - c, \\
			\omega(a_{d-1}) &= 0, \\
			\omega(a_{d-2}) &= b\cdot {d}_{d-1}(a_{d-2}) - c, \\
			\omega(a_i) &=
			\begin{cases} 
				(-1)^i\cdot c, &\text{if  } d \text{ is odd} \\
				(-1)^{i+1}\cdot c, &\text{if } d \text{  is even}
			\end{cases} \;\;\; \text{for } \;\; i\in \{0,1,2,...,d-3\},
		\end{align*}
		where $d_{d-1}(a_i)$ denotes the {\it $(d-1)$-dimensional degree} of $a_i$ (Definition~\ref{def-incident}).
		Then, the sum of the initial weights is 
		\begin{equation}\label{eq:eu}
			\sum_{i=0}^{d} \sum_{a_i\in A_i} \omega(a_i) = c\cdot [-1+ (-1)^{d+1}] \leq 0.
		\end{equation}
	\end{lemma}

	%	\begin{equation}[\text{Euler's~formula \cite{hatcher2002algebraic}}]
		%		\sum_{i=0}^{d} (-1)^i \cdot |A_i|=1+ (-1)^d
		%		\label{Euler}
		%	\end{equation}
	
	%\begin{equation}
	%	\sum_{a\in A_{d-2}} d_{G(d-1)}(a) = d\cdot |A_{d-1}| = \frac{d(d+1)}{2}\cdot |A_d|
	%	\label{relation} 
	%\end{equation}
	
	\begin{proof}
		First, it is easy to verify:
		\begin{align*}
			\sum_{a_d\in A_d} d_{d-1}(a_d) &= 2\cdot |A_{d-1}|, \\
			\sum_{a_{d-2}\in A_{d-2}} d_{d-1}(a_{d-2}) &= d\cdot |A_{d-1}|, 
		\end{align*}
		then we have
		\[
		\sum_{a_d\in A_d} \omega(a_d)+ \sum_{a_{d-2}\in A_{d-2}} \omega(a_{d-2}) = (2\cdot a+d\cdot b) \cdot |A_{d-1}| - c \cdot |A_{d-2}| - c \cdot |A_d|= c \cdot (-|A_{d-2}|+|A_{d-1}|-|A_d|).
		\]
		Thus, using Euler's formula~\cite{armstrong2013basic}
		\begin{equation}\label{eq:euler}
			\sum_{i=0}^{d} (-1)^i \cdot |A_i|=1+ (-1)^d 
		\end{equation}
		leads to Eq.~\eqref{eq:eu}.
	\end{proof}
	
	%			\[
	%			\sum_{i=0}^{d} \omega(a_i) = c\cdot [-1+ (-1)^{d+1}]
	%		\]

	%\begin{theorem}\label{apply-for-discharging}
	%	Let $H$ be a $d$-uniform $\mathbb{R}^d$-hypergraph with $d\geq 3$. Then
	%	\[
	%	\chi^i_{\mathbb{R}^d}(H) \leq
	%	\begin{cases}
		%		d + 3, & \text{if $d$ is odd}, \\[6pt]
		%		d + 4, & \text{if $d$ is even}.
		%	\end{cases}
	%	\]
	%\end{theorem}
	
	\begin{theorem}\label{apply-for-discharging}
		Let $H$ be a $d$-uniform $\mathbb{R}^d$-hypergraph with $d\geq 3$. Then $\chi^{d-2}_{\mathbb{R}^d}(H) \leq d + 3$. 		
		%	\[
		%	\chi^i_{\mathbb{R}^d}(H) \leq
		%	\begin{cases}
			%		d + 3, & \text{if $d$ is odd}, \\[6pt]
			%		d + 4, & \text{if $d$ is even}.
			%	\end{cases}
		%	\]
	\end{theorem}
	
	\begin{proof}
		%	\textbf{Case 1}. Suppose $d$ is odd.
		
		We proceed by contradiction. Assume that Theorem \ref{apply-for-discharging} does not hold, and let $H$ be a minimal counterexample with respect to $|A_{d-2}|$.
		
		We first establish the following lemma about reducible configuration.
		
		\begin{lemma}\label{jinyong1}
			For any $a_{d-2}\in A_{d-2}$, the number of $(d-1)$-dimensional topological hyperedges which are incident with $a_{d-2}$ is at least $d+3$.
		\end{lemma}
		
		\begin{proof}
			Since $H$ is a minimal counterexample, $H\backslash a_{d-2}$ admits a $(d-2)$-dimensional $(d+3)$-coloring $\phi$.
			
			Suppose to the contrary that the number of $(d-1)$-dimensional topological hyperedges which are incident with $a_{d-2}$ is at most $d+2$. 
			%$a_{d-2}\in A_{d-2}$ is incident with at most $d+2$ $(d-1)$-dimensional topological hyperedges. 
			Then the coloring $\phi$ of $H\backslash a_{d-2}$ can be extended to $H$, a contradiction. Hence the assumption is false, and the lemma follows.
		\end{proof}

		%		(i). If $d$ is odd: 
		%		
		%		\begin{itemize}
			%			\item $\omega(a_d) = d\cdot d_{d-1}(a_d) - d(d+1)$. 
			%			\item $\omega(a_{d-1}) = 0$. 
			%			\item $\omega(a_{d-2}) = (d-1)\cdot {d}_{d-1}(a_{d-2}) - d(d+1)$. 
			%			\item $\omega(a_i) = (-1)^i\cdot d(d+1)$ for $i\in \{0,1,2,...,d-3\}$
			%		\end{itemize}
		%		
		%		(ii). If $d$ is even: 
		%		
		%		\begin{itemize}
			%			\item $\omega(a_d) = d\cdot d_{d-1}(a_d) - d(d+1)$,
			%			\item $\omega(a_{d-1}) = 0$,
			%			\item $\omega(a_{d-2}) = (d-1)\cdot d_{d-1}(a_{d-2}) - d(d+1)$,
			%			\item $\omega(a_i) = (-1)^{i+1}\cdot d(d+1)$ for $i\in \{0,1,2,\dots,d-3\}$.
			%		\end{itemize}

		By Lemma \ref{discharging-lemma}, we assign weights to the hyperedges
		\begin{align*}
			\omega(a_d) &= d\cdot d_{d-1}(a_d) - d(d+1), \\
			\omega(a_{d-1}) &= 0, \\
			\omega(a_{d-2}) &= (d-1)\cdot {d}_{d-1}(a_{d-2}) - d(d+1), \\
			\omega(a_i) &=
			\begin{cases} 
				(-1)^i\cdot d(d+1), &\text{if  } d \text{ is odd} \\
				(-1)^{i+1}\cdot d(d+1), &\text{if } d \text{  is even}
			\end{cases} \;\;\; \text{for } \;\; i\in \{0,1,2,...,d-3\}.
		\end{align*}
		Then the sum of the initial weights is
		\[
		\sum_{i=0}^{d} \sum_{a_i\in A_i}\omega(a_i) = d(d+1) \cdot [-1+ (-1)^{d+1}] =
		\begin{cases}
			0, & \text{if $d$ is odd}, \\[6pt]
			-2d(d+1), & \text{if $d$ is even}.
		\end{cases}
		\]
		If we can define some Discharging rules and prove that: 
		\begin{itemize}
			\item If $d$ is odd, then $\omega'(a_i)\geq 0$ for every $i\in\{1,2,\dots,d\}$, and moreover, there exists some $a_j$ such that $\omega'(a_j)>0$; 
			
			\item If $d$ is even, the total weight of each hyperedge after Discharging is greater than $-2d(d+1)$; 
		\end{itemize}
		then a contradiction follows. 
		
		We now define the Discharging rules.
		
		\textbf{(R1)}. For $i\leq d-3$, collect all weights of $a_i\in A_i$, and then redistribute them equally so that every $a_i$ has the same weight.

		%	$\sum_{i=0}^{d} \omega(a_i) = c\cdot [-1+ (-1)^{d+1}] \leq 0$. 
		%	
		%	It is easy to verify that
		%	$\sum_{i\in \{1,2,\cdots,d\}}\sum_{a_i\in A_i}\omega(a_i)= c \dot [-1+ (-1)^{d+1}]$.

		%	It suffices to show that for every $i\in\{1,2,\dots,d\}$ we have $\omega'(a_i)\geq 0$, and moreover, that there exists some $a_j$ such that $\omega'(a_j)>0$.
		
		%		For $i=d$, we obtain
		%		\[
		%		\omega'(a_d)= d\cdot d_{d-1}(a_d) - d(d+1)\geq d(d+1)-d(d+1)=0.
		%		\]
		%		
		%		For $i=d-1$, we have $\omega'(a_{d-1})=0$.
		%		
		%		For $i=d-2$, by Lemma \ref{jinyong1},
		%		\[
		%		\omega'(a_{d-2})\geq (d-1)\cdot d_{d-1}(a_{d-2}) - d(d+1) \geq (d-1)(d+3)-d(d+1)=d-3.
		%		\]
		%		
		%		For $i\leq d-3$, by Euler's formula, we notice that 
		%		
		%		%$\sum_{i=0}^{d-3} \sum_{a_i\in A_i}\omega'(a_i) = d(d+1)\cdot \sum_{i=0}^{d-3}|A_{i}|$
		%		
		%		\[
		%		\sum_{i=0}^{d-3} \sum_{a_i\in A_i}\omega'(a_i) =
		%		\begin{cases}
			%			d(d+1)\cdot \sum_{i=0}^{d-3}(-1)^i|A_{i}| = d(d+1)\cdot (|A_{d}|-|A_{d-1}|+|A_{d-2}|), & \text{if $d$ is odd}, \\[6pt]
			%			d(d+1)\cdot \sum_{i=0}^{d-3}(-1)^{i+1}|A_{i}| = d(d+1)\cdot (|A_{d}|-|A_{d-1}|+|A_{d-2}| -2), & \text{if $d$ is even}.
			%		\end{cases}
		%		\]
		%		
		%		
		%		By Lemma \ref{quanhe}, 
		%		
		%		%	\[
		%		%	\sum_{i=0}^{d-3} \sum_{a_i\in A_i}\omega'(a_i) = d(d+1) - [1+(-1)^d]\cdot d(d+1) = d(d+1).
		%		%	\]
		%		
		%		\[
		%		\sum_{i=0}^{d-3} \sum_{a_i\in A_i}\omega'(a_i) \geq (-1)^{d+1} \cdot d(d+1) = 
		%		\begin{cases}
			%			d(d+1), & \text{if $d$ is odd}, \\[6pt]
			%			-d(d+1), & \text{if $d$ is even}.
			%		\end{cases}
		%		\]

		Denoting the new weights by $\omega'(a_i)$ after \textbf{(R1)}, we have:
		\begin{itemize}
			\item $i=d$: $\omega'(a_d)= d\cdot d_{d-1}(a_d) - d(d+1)\geq d(d+1)-d(d+1)=0$;
			\item $i=d-1$: $\omega'(a_{d-1})=0$;
			\item $i=d-2$: by Lemma \ref{jinyong1}, 
			\[
			\omega'(a_{d-2})\geq (d-1)\cdot d_{d-1}(a_{d-2}) - d(d+1) \geq (d-1)(d+3)-d(d+1)=d-3;
			\]
			\item $i\leq d-3$: by Euler's formula~\eqref{eq:euler}, 
			\begin{align*}
				&\sum_{i=0}^{d-3} \sum_{a_i\in A_i}\omega'(a_i) \\
				=&
				\begin{cases}
					d(d+1)\cdot \sum_{i=0}^{d-3}(-1)^i|A_{i}| = d(d+1)\cdot (|A_{d}|-|A_{d-1}|+|A_{d-2}|), & \text{if $d$ is odd}, \\[6pt]
					d(d+1)\cdot \sum_{i=0}^{d-3}(-1)^{i+1}|A_{i}| = d(d+1)\cdot (|A_{d}|-|A_{d-1}|+|A_{d-2}| -2), & \text{if $d$ is even},
				\end{cases}
			\end{align*}
			then by Lemma~\ref{quanhe},
			\[
			\sum_{i=0}^{d-3} \sum_{a_i\in A_i}\omega'(a_i) \geq (-1)^{d+1} \cdot d(d+1) = 
			\begin{cases}
				d(d+1), & \text{if $d$ is odd}, \\[6pt]
				-d(d+1), & \text{if $d$ is even}.
			\end{cases}
			\]
		\end{itemize}
		That is, after applying \textbf{(R1)}, when $d$ is odd, it is straightforward to verify that $\omega'(a_i)\geq 0$ for every $i\in \{1,2,\dots,d\}$, and there exists some $a_j$ such that $\omega'(a_j)>0$. This contradiction completes the proof of Theorem~\ref{apply-for-discharging} in the odd case. 
		
		We now turn to the case when $d$ is even which requires \textbf{(R2)}:
		
		%	\textbf{(R2)}. After (R1), if $\sum_{i\leq d-3}\sum_{a_i\in A_i}\omega'(a_i) < 0$, then this negative sum is redistributed equally among all hyperedges in $A_{d-2}$.
		
		\textbf{(R2)}. After \textbf{(R1)}, for all hyperedges of dimension at most $d-2$, we sum their weights and then redistribute this total weight equally among all such hyperedges. 
		
		Denoting the new weights by $\omega''(a_i)$ after \textbf{(R2)}, at this point, it suffices to verify that 
		\begin{equation}\label{eq:even}
			\sum_{i=0}^{d-2}\sum_{a_i\in A_i}\omega''(a_i) > -2d(d+1), 
		\end{equation}
		which immediately yields a contradiction. Since $d$ is even, and Theorem \ref{apply-for-discharging} is trivial if $|A_{d-2}|\leq d+3$, 
		%(the minimum $\mathbb{R}^d$ graph has $\frac{1}{2}d(d+1)$ $(d-2)$-dimensional topological hypergraph), 
		we may assume $|A_{d-2}|\geq d+4$. In consequence, we are able to verify Eq.~\eqref{eq:even} as follows
		\begin{align*}
			\sum_{i=0}^{d-2}\sum_{a_i\in A_i}\omega''(a_i) 
			&= \sum_{i=0}^{d-3}\sum_{a_i\in A_i}\omega'(a_i) + \sum_{a_{d-2}\in A_{d-2}}\omega'(a_{d-2}) \\
			&\geq -d(d+1) + (d-3)\cdot |A_{d-2}| \\
			&\geq -d(d+1)+(d-3)(d+4) = -12 >-2d(d+1).
		\end{align*}

	\end{proof}
	
	\begin{remark}
		Theorem~\ref{apply-for-discharging} merely serves to illustrate how the discharging method can be applied in higher dimensions. We have not attempted to optimize the upper bound for the $(d-2)$-dimensional chromatic number. We conjecture that the true upper bound for the $(d-2)$-dimensional chromatic number of $d$-uniform $\mathbb{R}^d$-hypergraphs is likely to be $d+1$.
	\end{remark}

	\begin{conjecture}
		Let $H$ be a $d$-uniform $\mathbb{R}^d$-hypergraph. Then $\chi^{d-2}_{\mathbb{R}^d}(H) \leq d+1$.
	\end{conjecture}

	%\begin{lemma}\label{quanhe0}
	%	Let $P$ be an $\mathbb{R}^d$-hypergraph with $d \geq 3$, and let $|A_{i}|$ denote the number of its $i$-dimensional topological hypergraph (we call it $i$-face for short). Then
	%	\[
	%	|A_{d}|-|A_{d-1}|+|A_{d-2}| \geq 1.
	%	\]
	%\end{lemma}
	%
	%\begin{proof}
	%	
	%	Note that each $(d-1)$-face of an $\mathbb{R}^d$-hypergraph lies in exactly two $d$-faces. 
	%	Indeed, every $(d-1)$-face $R$ is the intersection of exactly two $d$-faces $F_1,F_2$ of $P$. 
	%	Hence we may form the \emph{dual graph} $G$: its vertices are the $|A_d|$ $d$-faces of $P$, and its edges are the $|A_{d-1}|$ $(d-1)$-faces, each joining the two $d$-faces containing that $(d-1)$-face. 
	%	
	%	This graph $G$ is connected because the boundary of $P$ is connected. 
	%	Therefore $G$ is a connected graph on $|A_d|$ vertices and $|A_{d-1}|$ edges. 
	%	By the theory about cycle space \cite{bondy2008graph}, such a graph has $(|A_{d-1}| - |A_{d}| + 1)$ independent cycles.
	%	
	%	
	%	Consider a cycle $C = (F_1,F_2,\dots,F_k)$ in $G$, where $F_i$ denotes a $d$-face of $P$. 
	%	Since $\{F_i,F_{i+1}\}$ (indices mod $k$) is an edge of $G$, the intersection $R_i = F_i \cap F_{i+1}$ is a $(d-1)$-face of $P$. We notice that each independent cycle in $G$ yields at least one $(d-2)$-face of $P$. 
	%	
	%	Since $G$ has $|A_{d-1}| - |A_{d}| + 1$ independent cycles, we obtain the inequality
	%	\[
	%	|A_{d-2}| \geq |A_{d-1}| - |A_{d}| + 1.
	%	\]
	%	Rearranging gives
	%	\[
	%	|A_{d}|-|A_{d-1}|+|A_{d-2}| \geq 1.
	%	\]
	%	
	%\end{proof}

	\begin{lemma}\label{quanhe}
		Let $P$ be an $\mathbb{R}^d$-hypergraph with $d \geq 3$, and let $f_i$ denote the number of its $i$-dimensional topological hyperedges ($i$-faces). Then
		\[
		f_d - f_{d-1} + f_{d-2} \geq 1.
		\]
	\end{lemma}
	
	\begin{proof}
		We prove this inequality by considering the dual $\mathbb{R}^d$-hypergraph (Definition \ref{dual graph}) of $P$, denoted by $P^*$. Since $P$ is an $\mathbb{R}^d$-hypergraph, its dual $P^*$ is also an $\mathbb{R}^d$-hypergraph by Lemma \ref{dual-rd}. By the duality principle, there is a bijection between the $i$-faces of $P$ and the $(d-i)$-faces of $P^*$.
		
		Let $g_k$ denote the number of $k$-faces of the dual $\mathbb{R}^d$-hypergraph $P^*$. The correspondence implies:
		\begin{align*}
			f_d(P) &= g_0(P^*) \quad (\text{number of vertices}), \\
			f_{d-1}(P) &= g_1(P^*) \quad (\text{number of edges}), \\
			f_{d-2}(P) &= g_2(P^*) \quad (\text{number of 2-faces}).
		\end{align*}
		Substituting these into the original expression, the inequality to be proved becomes:
		\[
		g_0 - g_1 + g_2 \geq 1.
		\]
		Consider the 1-skeleton of $P^*$, which forms a connected graph $G = (V, E)$ with $|V| = g_0$ and $|E| = g_1$. The number of linearly independent cycles (the cyclomatic number or the first Betti number of the graph) is given by:
		\[
		\beta_1 = g_1 - g_0 + 1.
		\]
		Since the $1$-th homotopy group of $P^*$ is trivial. This implies that every cycle in the $1$-skeleton of $P^*$ is the boundary of a $2$-chain formed by the $2$-faces of $P^*$. 
		Therefore, the boundaries of the $g_2$ 2-faces must span the cycle space of the graph $G$, which has dimension $\beta_1$. For a set of $g_2$ vectors to span a vector space of dimension $g_1 - g_0 + 1$, we must have:
		\[
		g_2 \geq g_1 - g_0 + 1.
		\]
		Rearranging this inequality yields:
		\[
		g_0 - g_1 + g_2 \geq 1.
		\]
		Substituting the face counts of $P$ back into the equation, we obtain:
		\[
		f_d - f_{d-1} + f_{d-2} \geq 1.
		\]
	\end{proof}

	\begin{definition}[dual $\mathbb{R}^d$-hypergraph]\label{dual graph}
		Let $P$ be an $\mathbb{R}^d$-hypergraph which is embedded in $\mathbb{R}^d$, and let $A_i$ denote  the set of $i$-face of $P$ (Definition~\ref{def:Ai}). 
		The \textit{dual $\mathbb{R}^d$-hypergraph} $P^*$ of $P$ is constructed as follows (let $A^*_i$ denote the set of $i$-faces of $P^*$):
		
		\begin{itemize}
			\item \textbf{$0$-faces (vertices):} For every $d$-face $a_d \in A_d$ of $P$, there corresponds a unique vertex $a^*_{a_d} \in A^*_0$. Thus, $|A^*_0| = |A_d|$. 
			
			After this step, we get an empty graph $P_0$.

			\item \textbf{$1$-faces (edges):} Whenever a $(d-1)$-face $a_{d-1}$ of $P$ lies on the common boundary of two $d$-faces $a_d$ and $a_d'$, we add in $P_0$ a $1$-face $a^*_{a_{d-1}}$ that connects the $0$-faces $a^*_{a_d}$ and $a^*_{a_d'}$.
			
			Therefore, for every $(d-1)$-face $a_{d-1} \in A_{d-1}$ of $P$, there corresponds a unique $1$-face $a^*_{a_{d-1}} \in A^*_1$. Thus, $|A^*_1| = |A_{d-1}|$. 
			
			After this step, we get a graph $P_1$. 
			
			\item \textbf{$2$-faces:} Whenever a $(d-2)$-face $a_{d-2}$ of $P$ is the intersection of all $(d-1)$-faces in a set
			$J=\{a_{d-1}^1, a_{d-1}^2, \ldots, a_{d-1}^x\} \subseteq A_{d-1}$, 
			where $J$ is maximal with respect to this property, it is easy to see that the $1$-faces in $P_1$ corresponding to the elements of $J$ form a cycle $C$. 
			By filling the cycle $C$ with a $2$-ball, we obtain a $2$-face $a^*_{a_{d-2}}$ of $P^*$; equivalently, we add a $2$-face $a^*_{a_{d-2}}$ in $P_1$.

			Therefore, for every $(d-2)$-face $a_{d-2} \in A_{d-2}$ of $P$, there corresponds a unique $2$-face $a^*_{a_{d-2}} \in A^*_2$. Thus, $|A^*_2| = |A_{d-2}|$. 
			
			After this step, we get $P_2$. 
			
			\item $\dots$ Repeating the above procedure yields the dual graph $P^*$ $\dots$

			\item \textbf{$d$-faces:} For every $0$-face $a_{0} \in A_{0}$ of $P$, there corresponds a unique $d$-face $a^*_{a_{0}} \in A^*_d$. Thus, $|A^*_d| = |A_{0}|$. 
			
			After this step, we get $P_d =P^*$. 
			
			%		\item \textbf{Incidence:} If an edge $e$ in $P$ lies on the common boundary of two faces $f_1$ and $f_2$, the corresponding edge $e^*$ in $P^*$ connects the vertices $v^*_{f_1}$ and $v^*_{f_2}$. 
		\end{itemize}
		%Since the above construction is performed in $\mathbb{R}^d$ and all added faces intersect only along their boundaries, the $i$-th homotopy group of $P^*$ is trivial for $i\leq d-2$ by Lemma \ref{attach} - \ref{attach2}, and $P^*$ admits an embedding in $\mathbb{R}^d$. Therefore, $P^*$ is also an $\mathbb{R}^d$-hypergraph. 
		
	\end{definition}
	
	\begin{lemma}\label{dual-rd}
		Let $P$ be an $\mathbb{R}^d$-hypergraph, then the dual $\mathbb{R}^d$-hypergraph $P^*$ of $P$  is also an $\mathbb{R}^d$-hypergraph. 
	\end{lemma}
	
	\begin{proof}
		Since the above construction is performed in $\mathbb{R}^d$ and all added faces intersect only along their boundaries, the $i$-th homotopy group of $P^*$ is trivial for $i\leq d-2$ by Lemmas \ref{attach}  and \ref{attach2}, and $P^*$ admits an embedding in $\mathbb{R}^d$. Therefore, $P^*$ is an $\mathbb{R}^d$-hypergraph. 
	\end{proof}
	
	%\begin{definition}[dual polytope]\label{dual polytope}
	%	Let $P \subset \mathbb{R}^d$ be a convex polytope such that the origin is contained in its interior (i.e., $0 \in \operatorname{int}(P)$). The \textit{dual polytope} $P^*$, often referred to as the \textit{polar set} of $P$, is defined as:
	%	\[
	%	P^* = \left\{ y \in \mathbb{R}^d \mid \langle x, y \rangle \leq 1 \text{ for all } x \in P \right\},
	%	\]
	%	where $\langle x, y \rangle$ denotes the standard Euclidean inner product in $\mathbb{R}^d$.
	%	
	%	\vspace{0.5em}
	%	\noindent\textbf{Key Property:} There is an inclusion-reversing bijection between the faces of $P$ and the faces of $P^*$. Specifically, if $F$ is a $k$-dimensional face of $P$, then the corresponding face $F^*$ of $P^*$ is a $(d-1-k)$-dimensional face, defined by:
	%	\[
	%	F^* = \left\{ y \in P^* \mid \langle x, y \rangle = 1 \text{ for all } x \in F \right\}.
	%	\]
	%\end{definition}

	\section{Future work}\label{future}
	
	Note that Theorem \ref{chromatic-number} only provides a rough estimate of the upper bound for the chromatic number of graphs embeddable in $\mathbb{R}^d$. Therefore, we are still far from a proof of the Hadwiger conjecture. 
	Based on Corollary~\ref{anti-minor2}, the Hadwiger conjecture is equivalent to the following two statements. 
	
	\begin{itemize}
		\item Let $G$ be the $1$-skeleton of an $\mathbb{R}^d$-hypergraph, then $\chi(G)\leq d+2$. 
		\item If a graph $G$ does not contain a $K_{d+3}$-minor but does contain a $K_{i, d+4-i}$-minor for $i \in \{2, 3, \dots, \lfloor \frac{d+4}{2} \rfloor \}$, then $\chi(G) \leq d+2$. 
	\end{itemize}

	\section*{Acknowledgement}
	This work was funded by the National Key R \& D Program of China (No.~2022YFA1005102) and the National Natural Science Foundation of China (Nos.~12325112, 12288101).

	\section*{Statements and Declarations}
	
	\begin{itemize}
		%\item Funding: National Key R \& D Program of China (No.~2022YFA1005102) and the National Natural Science Foundation of China (Nos.~12325112, 12288101). 
		\item Competing interests: We hereby declare that there are no conflicts of interest to disclose in relation to this submission. We have no affiliations, relationships, or financial interests that could be perceived as influencing this work.
		%\item Ethics approval and consent to participate: Not applicable
		%\item Consent for publication: Not applicable
		\item Data availability: This research does not involve any data generation or analysis. Therefore, no data are available. 
		%\item Materials availability: Not applicable
		%\item Code availability: Not applicable
		%\item Author contribution: These authors contributed equally to this work. 
	\end{itemize}

	%%
	%% The next two lines define the bibliography style to be used, and
	%% the bibliography file.
	
	%\bibliographystyle{ACM-Reference-Format}

	\bibliographystyle{plain}
	\bibliography{sample-base}
	%\addbibresource{sample-base}
	
	%plain: 按字母顺序排列文献，使用数字作为引用标识。
	%unsrt: 按引用顺序排列文献，使用数字作为引用标识。
	%ieeetr: IEEE 风格，按引用顺序排列文献，使用数字作为引用标识。
	%acm: ACM 风格，使用数字标识。

\end{document}